\documentclass[aap]{imsart}

\RequirePackage{amsthm,amsmath,amsfonts,amssymb}
\RequirePackage[numbers]{natbib}
\RequirePackage{color}
\RequirePackage{mathtools,enumitem}
\RequirePackage{csquotes}
\RequirePackage{subcaption}
\RequirePackage[colorlinks,citecolor=blue,urlcolor=blue]{hyperref}
\usepackage{thmtools}
\usepackage{cleveref}
\AtBeginEnvironment{appendix}{\crefalias{section}{appendix}}
\RequirePackage{graphicx}
\usepackage{dsfont}
\usepackage{graphicx}
\usepackage{bbm}
\usepackage{mathrsfs}
\usepackage{xcolor}
\usepackage{svg}
\usepackage{bm}
\usepackage{ushort}

\startlocaldefs
\theoremstyle{plain}
\newtheorem{theorem}{Theorem}[section]

\newtheorem{lemma}[theorem]{Lemma}

\newtheorem{proposition}[theorem]{Proposition}

\newtheorem{op}[theorem]{Open problem}
\newtheorem{definition}[theorem]{Definition}

\newtheorem{remarkx}[theorem]{Remark}
\newenvironment{rem}
  {\pushQED{\qed}\remarkx}
  {\popQED\endremarkx}
\makeatletter

\newcommand{\invisible}[1]{}

\newcommand{\todotemp}[1]{}
\numberwithin{equation}{section}
\numberwithin{theorem}{section}

\makeatletter

\newcommand{\bfT} {{\mathbf T}}

\newcommand{\Rbold} {{\mathbb R}}

\newcommand{\e}{{\mathrm e}}

\newcommand{\PAnamdel}{{\rm PA}_{n}^{\sss (m,\delta)}(a)}

\newcommand{\PAnaonedelcol}{{\rm PA}_{mn}^{\sss (1,\delta/m)}(a)}
\newcommand{\PAnaonedel}{{\rm PA}_{n}^{\sss (1,\delta)}(a)}

\newcommand{\PAnbonedel}{{\rm PA}_{n}^{\sss (1,\delta)}(b)}

\newcommand{\PAnbonedelcol}{{\rm PA}_{mn}^{\sss (1,\delta/m)}(b)}

\newcommand{\PAncmdel}{{\rm PA}_{n}^{\sss (m,\delta)}(c)}

\newcommand{\PAndmdel}{{\rm PA}_{n}}

\newcommand{\PAndonedelcol}{{\rm PA}_{mn}^{\sss (1,\delta/m)}(d)}

\newcommand{\PAn}{{\rm PA}_n}

\newcommand{\PAnbmdel}{{\rm PA}_{n}^{\sss (m,\delta)}(b)}

\newcommand{\gdist}[1]{{\rm dist}_{\sss #1}}

\newcommand{\Ver}{o}

\newcommand{\dRG}{{\mathrm{d}}}

\newcommand{\dint}{\mathrm{d}}

\newcommand{\Bin}{{\sf Bin}}

\newcommand {\convp}{\stackrel{\sss {\mathbb P}}{\longrightarrow}}

\newcommand\1{\mathbbm{1}}
\newcommand{\indic}[1]{\1_{\{#1\}}}
\newcommand{\indicwo}[1]{\1_{#1}}

\newcommand{\expec}{{\mathbb{E}}}

\newcommand{\prob}{{\mathbb{P}}}

\newcommand{\Var}{\mathrm{Var}}

\newcommand{\R}{\Rbold}

\newcommand{\sss}   { \scriptscriptstyle }

\newcommand{\rO}{{\mathrm{\scriptstyle{O}}}}
\newcommand{\rY}{{\mathrm{\scriptstyle{Y}}}}
\newcommand{\srO}{{\mathrm{\scriptscriptstyle{O}}}}
\newcommand{\srY}{{\mathrm{\scriptscriptstyle{Y}}}}

\newcommand {\vep}{\varepsilon}

\newcommand{\jconn}[1]{\stackrel{#1}{\rightsquigarrow}}

\def\eqalign#1\enalign{
    \begin{align}#1\end{align}
    }
\newcommand{\eq}{\begin{equation}}
\newcommand{\en}{\end{equation}}
\newcommand{\ben}{\begin{enumerate}}
\newcommand{\een}{\end{enumerate}}
\newcommand{\eqn}[1]{\begin{equation} #1 \end{equation}}
\newcommand{\eqan}[1]{\eqalign #1 \enalign}

\newcommand{\nn}{\nonumber}

\renewcommand{\to}{\rightarrow}

\newcommand{\RvdH}[1]{\todo[inline, color=magenta]{Remco: #1}}
\newcommand{\beginsol}[1]{}

\newcommand{\customitem}[1]{%
\item[\rm#1]\protected@edef\@currentlabel{#1}%
}
\makeatother
\newcommand\nka{\hat{\kappa}}
\newcommand\rou[1]{\lfloor{#1}\rfloor}
\newcommand\ceil[1]{\lceil{#1}\rceil}
\newcommand{\zhu}[1]{{#1}}

\newcommand\bc[1]{\left({#1}\right)}
\newcommand\cbc[1]{\left\{{#1}\right\}}

\newcommand\brk[1]{\left\lbrack{#1}\right\rbrack}

\newcommand\norm[1]{\left\|{#1}\right\|}
\newcommand\abs[1]{\left|{#1}\right|}
\usepackage{dsfont}
\newcommand{\vpie}{\vec\pi^{\sss e}}
\newcommand{\vrhoe}{\vec\rho^{\sss e}}
\newcommand\Erw{\mathbb{E}}
\newcommand\bdt{b}

\newcommand\PP{\mathbb{P}}

\newcommand\RR{\mathbb{R}}
\newcommand\nVer{a}
\newcommand{\psE}{p_s^{\sss E}}
\newcommand{\qsE}{q_s^{\sss E}}
\newcommand{\pspi}{p_s^{\sss \pi}}
\newcommand{\qspi}{q_s^{\sss \pi}}
\newcommand{\pspiu}{p_s^{\sss \pi_{\vec u}}}
\newcommand{\qspiu}{q_s^{\sss \pi_{\vec u}}}
\newcommand{\psrho}{p_s^{\sss \rho}}
\newcommand{\qsrho}{q_s^{\sss \rho}}
\newcommand{\lb}{{\rm lb}}
\newcommand{\rangeoy}{[0,1]\times\cbc{\rO,\rY}}

\newcommand{\uhvep}{\underline{\nu}_{\zeta}}
\newcommand{\rmq}{{\rm q}}
\newcommand{\SL}{{\rm SL}}

\newcommand{\PPT}{{\rm PPT}^{\sss (m,\delta)}}
\newcommand{\PU}{ {\rm PU}_n}
\newcommand{\PUm}{{\rm PU}_{mn}}
\newcommand{\SPU}{{\rm SPU}_n}
\newcommand{\SPUm}{{\rm SPU}_{mn}}
\newcommand{\SA}{\mathscr{S}}
\newcommand{\ESA}{\mathscr{S}}
\newcommand{\TESA}{\tilde{\mathscr{S}}}
\newcommand{\BESA}{\bar{\mathscr{S}}}
\newcommand{\TSA}{\tilde{\mathscr{S}}}
\newcommand{\NESA}{\tilde{\mathscr{S}}}
\newcommand{\NA}{\mathscr{N}}
\newcommand{\NP}{\mathscr{P}}
\newcommand{\GNW}{\mathcal{E}}
\DeclareMathOperator{\sgn}{sgn}
\newcommand{\imag}{\mathrm{i}}
\newcommand{\ch}[1]{{#1}}
\newcommand{\longversion}[1]{#1}
\newcommand{\shortversion}[1]{}
\endlocaldefs
\newcommand{\PAnjmdel}{{\rm PA}_{n,j}}
\endlocaldefs

\begin{document}

\begin{frontmatter}
\title{Logarithmic typical distances in\\
preferential attachment models}
\runtitle{Typical distances in preferential attachment models}

\begin{aug}
\author[A]{\fnms{Remco} \snm{van der Hofstad }\ead[label=e1]{r.w.v.d.hofstad@tue.nl}}
\author[A]{\fnms{Haodong} \snm{Zhu}\ead[label=e2]{h.zhu1@tue.nl}}
\address[A]{Department of Mathematics and Computer Science, Eindhoven University of Technology\printead[presep={,\\}]{e1,e2}}
\end{aug}
\begin{abstract}
We prove that the typical distances in a preferential attachment model with out-degree $m\geq 2$ and strictly positive fitness parameter are close to $\log_\nu{n}$, where $\nu$ is the exponential growth parameter of the local limit of the preferential attachment model. The proof relies on a path-counting technique, the first- and second-moment methods, as well as a novel proof of the convergence of the spectral radius of the offspring operator under a certain truncation.
\end{abstract}



\end{frontmatter}
\tableofcontents


\section{Introduction and main result}

\invisible{\RvdH{- In terms of notation, do we wish to simplify $\PAn=\PAndmdel$ after the first section?\\
 -Abbreviate left- and rhs by l.h.s. and r.h.s.?\\
- Are there further parts that can be moved to the appendix? The paper is somewhat longish.\\
- It would be good also to anticipate a short and long version, where the short version will be submitted to the journal, the long one to the arxiv. I added the macros on the last line of the preamble that will allow us to do this.}}
\subsection{Introduction}
\label{sec-intro}
 \paragraph{Common properties of real-world networks} 
One of the primary objectives of random graph theory is to effectively model the diverse and complex networks present in the real world. Many of these networks share three common properties (see, e.g., \cite[Chapter 10]{Newm18}):
\begin{enumerate}
    \item \textit{Small-world phenomenon}: The shortest path length between a randomly selected pair of vertices in a real-world network is typically much smaller than the overall size of the network. A well-known example of this is the "six degrees of separation" paradigm, which posits that any two people in the world are connected through six or fewer social relationships (see, e.g., \cite{Watt03}).
    \item \textit{Power-law degree distribution}: The degree distribution in many real-world networks is close to a power law. Thus, while most vertices have a relatively small number of connections, a few vertices (often termed ``hubs") possess an extraordinarily high number of connections. 
    \item \textit{Giant  component}: Many real-world networks have a giant connected component that encompasses a positive proportion, if not all, of the network's vertices.
\end{enumerate}
\ch{The} preferential attachment model (PAM) has been developed as a probabilistic framework that captures these three essential features. Unlike traditional random graph models, such as the inhomogeneous random graph or the configuration model, the PAM is \textit{dynamic} and reflects the ongoing growth of networks over time. This model is grounded in the principle of cumulative advantage, meaning that vertices that already have a high degree of connectivity are more likely to attract new connections. As a result, PAMs have been proposed as models for a wide range of real-world networks, including sexual networks \cite{PAinSN,JonHan03}, the World-Wide Web \cite{AdaHub00,kbmm13}, networks that epidemics spread on \cite{MR2298278,Jones2003,Watt03}, and citation networks \cite{Csar06,WanYuYu08}.

As one of the central properties of real-world networks, the length of the shortest path between two uniformly chosen vertices in a network, conditioned on their being connected, has a long history of study. In random graph theory, this length is referred to as the \textit{typical distance}. Understanding \ch{the} typical distance is critical, as it has significant implications for various real-world phenomena. For example, in the context of epidemics, the typical distance serves as an indicator of the speed at which a disease can propagate from one individual to another \cite{Newm18}. 

\paragraph{Universality of the typical distance and the exponential growth parameter $\nu$} Numerous observations and proofs have established that the behavior of typical distances in sparse random graph models is closely linked to a spectral radius $\nu$ \footnote{$\nu$ is sometimes also called spectral norm, which is the same in the cases we care about.}. This spectral radius serves as a crucial indicator of the exponential growth of neighborhoods within the graph. This exponential growth has been shown for local limits of various sparse random graph models, including inhomogeneous random graphs 
\cite{BolJanRio07,ChuLu03,DerMonMor12,EskHofHoo06}, configuration models \cite{HofHooVan05a,HofHooZna04a} and PAMs with out-degree  $m\geq 2$ \cite{DerMonMor12,DomHofHoo10,Pitt94}.
More precisely,
\begin{enumerate}
\item when $\nu= \infty$, the typical distances grow sub-logarithmically with the network size; such network are often called \textit{ultra-small worlds};

    \item when $\nu< \infty$, the typical distances grow  logarithmically with
the network size\ch{;} such network are often called as \textit{small worlds}. 
\end{enumerate}

In the first case, for a network of size $n$, a sharp bound of $\frac{c}{\tau-2}\log \log n$ is established\footnote{Throughout this paper we use $\log$  to denote the natural logarithm when the base is not specified.}  for the typical distances in a wide class of random graphs where vertex degrees follow an asymptotic power-law distribution with exponent $\tau\in (2,3)$. Specifically, for PAMs, this setting applies when the out-degree \ch{$m$ satisfies} $m\geq 2$ and the fitness parameter $\delta\in (-m,0).$ \zhu{\ch{Here, the} fitness parameter acts as a weight in the probability that an edge is connected to a vertex in the PAM, \ch{where} a vertex \ch{has} weight $\delta$, and an edge weight 
$1$, \ch{and} connection \ch{occurs with a probability that is proportional to the weight of the vertex and its incident edges}; see \eqref{eq-def-edge-prob-pand} for an explicit formula.} \ch{The degree} power-law exponent \ch{is} given by $\tau=3+\delta/m$.\footnote{When $m=1$, for any $\delta>-1$, the PAMs are trees, and a different sharp bound of the order of $\log n$ was proved for their typical distances (see, e.g., \cite[Theorem 8.1]{Hofs24}).} The constant 
$c$ varies depending on the graph model:
For inhomogeneous random graphs and configuration models, $c=2$, while for PAMs, $c=4$
\cite{DerMonMor12,HofHooZna04a}. 
When $\delta=0$, the typical distance for PAM is around $\frac{\log n}{\log \log n}$ \cite{BolRio04b,Hofs24}.

In the second case, corresponding to $\delta>0$ for PAMs, a sharp bound of $\log_\nu n$ for the typical distance is established for both inhomogeneous random graphs and configuration models.  In $2010$, with $m\geq 2$, Dommers, the first author  and Hooghiemstra proved an upper bound $c_1\log n$ and a lower bound $c_2\log n$, for some constants $c_1<c_2$ in \cite{DomHofHoo10}, leaving the precise sharp bound as an open problem. Later, in $2024$, the first author  provided a lower bound of $\log_\nu{n}$ in \cite[Theorem 8.13]{Hofs24}, and conjectured this bound to be sharp.

In this article, we show that this conjecture is indeed correct, by proving that the typical distance in PAMs with out-degree $m\geq 2$ and fitness parameter $\delta>0$ is indeed \zhu{of order} $\log_\nu{n}$. Before formulating this result, we first formally define our PAMs.

\subsection{Model definition}\label{sec_mod-def}
There are several variations of PAMs, \zhu{and the model we primarily focus on in this article \ch{will be denoted by}  $\PAndmdel^{\sss (m,\delta)}$; \ch{see Remark \ref{rem-related-PAMs} for extensions to related PAMs}.} Here, $n$ represents the number of vertices in the graph. Here, $n$ represents the number of vertices in the graph,  $m$ denotes the number of out-edges from vertices other than $1$, and the fitness parameter $\delta$ impacts the law of the other endpoint of these edges.  For simplicity, we use $[n]=\cbc{1,2,\ldots,n}$ to represent the vertex set of $\PAndmdel^{\sss (m,\delta)}$.

\paragraph{Construction of $\PAndmdel^{\sss (m,\delta)}$} We start with $n=2$. In this case, the model consists of $2$ vertices labeled $1$ and $2$, with $m$ edges directed from $2$ to $1$, labeled $1,2,\ldots,m$. At time $n+1$, the vertex $n+1$ and $m$ out-edges from $n+1$  labeled by $1,2,\ldots, m$, are added to the graph. 
The probabilities that the other endpoints of these $m$ out-edges from $n+1$ are given by
\eqan{
\PP\bc{n+1\jconn{j+1}  i\mid \PAnjmdel^{\sss (m,\delta)}}=\frac{D_i(n,j)+\delta}{2(n-1)m+n\delta+j}~\forall~i\in [n],~j\in\cbc{0,\ldots,m-1},\label{eq-def-edge-prob-pand}
}
where $\PAnjmdel^{\sss (m,\delta)}$ is the graph after adding vertex $n+1$ and the first $j$ out-edges of vertex $n+1$ to $\PAndmdel^{\sss (m,\delta)}$, and $D_i(n,j)$ is the degree of $i$ in $\PAnjmdel^{\sss (m,\delta)}$. 

In the construction, vertex $i$ is added at time $i$. Hence, we often refer to  $i$ as the vertex's \textit{age}. Furthermore, we say \ch{that} a vertex $x$ is \textit{older}  than vertex $y$ if $x<y$, and \textit{younger} if $x>y$. 
It is worth mentioning that although PAMs are constructed as directed graphs, they are often treated as undirected graphs in many real-world applications, such as sexual networks and epidemics, which is also the perspective used in this paper.

\paragraph{Connectivity of $\PAndmdel^{\sss (m,\delta)}$} According to \eqref{eq-def-edge-prob-pand}, each vertex in $\PAndmdel^{\sss (m,\delta)}$, except for vertex $1$, is always linked to an older vertex. Therefore, each vertex belongs to the same connected component as vertex $1$. Consequently, $\PAndmdel^{\sss (m,\delta)}$ forms a connected graph.

\zhu{
 \begin{rem}[\ch{Related PAMs}]
 \label{rem-related-PAMs}
     Sometimes $\PAndmdel^{\sss (m,\delta)}$ is written as $\PAndmdel^{\sss (m,\delta)}(d)$ or $\PAndmdel^{\sss (D)}{ (m,\delta)}$ \cite{GarHazHofRay22,Hofs24}.  \ch{The name $\PAndmdel^{\sss (m,\delta)}(d)$ correctly suggests that} there are models like $\PAnamdel$, $\PAnbmdel$, and $\PAncmdel$. \ch{These} models only differ slightly in their starting states and edge probabilities. We
\longversion{defer the introduction of other PAMs to \Cref{sec-td-other-pams}.}\shortversion{refer to \cite[Appendix C]{HofZhu25long} for the introduction of other PAMs.}
 \end{rem}
}
\subsection{Main result}
The main result of this article is the following theorem, which identifies the asymptotics of the typical distance in $\PAndmdel^{\sss (m,\delta)}$:
\begin{theorem}[Convergence of the typical distances]
\label{thm-log-PA-delta>0}
Fix $m\geq 2$ and $\delta>0$, and consider $\PAndmdel^{\sss (m,\delta)}$. As $n\rightarrow \infty$, 
with $\Ver_1,\Ver_2$ uniformly distributed in $[n]$, 
\begin{align*}
    \frac{\gdist{\PAndmdel^{\sss (m,\delta)}}(\Ver_1, \Ver_2)}{\log_\nu{n}}\convp 1,
\end{align*}
where $\gdist{\PAndmdel^{\sss (m,\delta)}}(\Ver_1, \Ver_2)$ denotes the length of the shortest path between $\Ver_1$ and $\Ver_2$ in $\PAndmdel^{\sss (m,\delta)}$, and
	\eqn{\label{spectral_radius}
	\nu=2\frac{m(m+\delta)+\sqrt{m(m-1)(m+\delta)(m+1+\delta)}}{\delta}>1.
	}
\end{theorem}
\Cref{thm-log-PA-delta>0} also holds for other versions of the PAM. We \longversion{defer the proofs of \Cref{thm-log-PA-delta>0} for $\PAnamdel$ and $\PAnbmdel$ to \Cref{sec-td-other-pams}, using the collapsing procedure in \cite{GarHazHofRay22,Hofs24}}\shortversion{refer to  \cite[Appendix C]{HofZhu25long} for the proofs for $\PAnamdel$ and $\PAnbmdel$}, while the result of $\PAncmdel$ follows directly from that of $\PAnamdel$, since $\PAncmdel$ can be reduced to $\PAnamdel$ (see, e.g., \cite[Section 8.2]{Hofs24}).


\subsection{Discussion}
\paragraph{Challenges in analyzing the typical distances for PAMs}

Since the typical distances for inhomogeneous random graphs and configuration models were resolved over a decade ago \cite{BolJanRio07,ChuLu03,DerMonMor12,EskHofHoo06,HofHooVan05a,HofHooZna04a}, it is natural to wonder why it took so much longer for PAMs in the logarithmic regime. Here we list some reasons:

\begin{itemize}
    \item[$\rhd$]\textit{Lack of independence in PAMs.} In inhomogeneous random graphs, given the vertex weights, the edge statuses are independent. In configuration models, given the vertex degrees, half-edges are paired uniformly at random. As a result, both types of random graphs exhibit either complete independence or manageable dependencies between edges, making them more tractable. In PAMs, however, only {\em conditional independence} exists, as proved by Berger et al.\ \cite{BerBorChaSab14}, \ch{making these models} less tractable since \ch{the dependencies  depend} on certain beta random variables \ch{that act as random vertex weights} (see \Cref{sec_polya_pam}). 
    \item[$\rhd$]\textit{Coupling to the local limit is much harder.} In rank-1 inhomogeneous random graphs and configuration models, if the variance of the degree distribution is finite, a common approach to computing typical distances is to couple the random graph with its local limit. Due to the independence or controllable dependencies, this coupling holds up to a depth of 
$(1/2+\eta)\log_\nu n$ for configuration models \cite[Section 3]{HofHooVan05a} and rank-1 inhomogeneous random graphs \cite[Sections 2 and 3]{EskHofHoo06}, where 
$n$ is the graph size and $\eta>0$ is sufficiently small. However, for PAMs, to the best of our knowledge, such sharp coupling results have not been proved, as they are challenging. 
\item[$\rhd$]\textit{Truncation technically challenging.} An alternative approach, namely, the path-counting technique used in this paper, can also be applied. In this method, often vertex weights in inhomogeneous random graphs, and vertex degrees in configuration models, are truncated \cite{Hofs24}. As a result, the focus shifts to determining a new exponential growth parameter, defined as the spectral radius of the offspring operator in the local limit of the corresponding graphs. After truncation, the graph retains its structure as an inhomogeneous random graph or a configuration model, respectively, simplifying the computation of this new exponential growth parameter.
However, for preferential attachment models (PAMs), truncating vertex ages alters the graph in such a way that it no longer is a PAM. Additionally, for PAMs, the offspring operator is neither compact nor self-adjoint, further complicating the calculation of its spectral radius after truncation. Indeed, even the spectral radius before truncation was only recently identified in \cite{hazra2023percolation}.
\end{itemize}

\paragraph{Diameter of PAMs.} The diameter of a graph is the largest distance between any pair of vertices in it, provided the two vertices are in the same connected component. \zhu{Once we have a good characterisation of the typical distances of PAMs, it is then an intriguing question to study their diameters. Since the diameter is the largest distance between any pair of vertices, the difference between typical distances and the diameter often comes from special vertices in the graph where the size of their neighborhoods grows more slowly than the rest. Based on this obervation, in}  \cite{CarGarHof19}, Caravenna,  Garavaglia and the first author proved that the diameter of PAMs with $m\geq 2$ and $\delta\in(-m,0)$ is approximately $(\frac{4}{\tau-2}+\frac{2}{\log m})\log \log n$, where again $\tau=3+\delta/m$. Compared to its typical distance, this additional $(2\log \log n/\log m)$ term arises from the occurrence of \textit{minimally connected vertices}, whose $(\log \log n/\log m)$-neighborhoods contain only vertices with degree $m$ (the root) or $m+1$ (other vertices). They also derived a similar diameter for the configuration model when vertex degrees follow an asymptotic power-law distribution with exponent $\tau\in (2,3)$.
Assuming $d_{\rm min}\geq 3$ as the minimal vertex degree, the authors showed that the diameter is approximately $(\frac{2}{\tau-2}+\frac{2}{\log (d_{{\rm min}}-1)})\log \log n$.

On the other hand, when $\nu\in (1,\infty)$, 
\begin{itemize}
    \item Fernholz and Ramachandran proved in \cite{FerRam07} that, for the configuration model, the diameter is approximately $\log_\nu n+c_1\indic{p_1>0}\log n+c_2\indic{p_1=0,p_2>0}\log n$, where $p_1$ and $p_2$ represent the proportions of vertices with degree $0$ and $1$, respectively, and $c_1$ and $c_2$ depend on the degree distribution. 
    \item Bollobás, Janson, and  Riordan proved in \cite{BolJanRio07} that, for the inhomogeneous random graph, the diameter is approximately $\log_\nu n+c_3\log n$, where $c_3$ depends on the so-called {\em dual} kernel of the Poisson multi-type branching process arising as its local limit. This dual kernel describes the local limit of connected components conditioned to be finite.
\end{itemize}
\zhu{Recently, \ch{relying in an essential way on} the results in this paper, Du, Gong, Li, and the second author proved that the diameter of PAMs with $m \geq 2$ and $\delta > 0$ is also approximately $\log_\nu n$, i.e., it agrees up to the first order with the typical distances \cite{du2026asymptotic}. This result is consistent with the behavior of the typical distance in the configuration model when $p_1 = p_2 = 0$. The core of the proof is to show that, whp, there exists some $R_n$ of the order of $(\log n)^{2/3}$ such that the $R_n$-neighborhood of every vertex in $G_n$ has size much larger than $\log n$. \ch{After that, our results can be used to show the sharp diameter bound.}

}
\smallskip

\paragraph{Open problem}This paper introduces an approach to studying graph distances in PAMs. However, there are variations of PAMs whose typical distances remain unknown:
\begin{itemize}
    \item \textit{The PAM with random out-\ch{degrees}.} For each vertex, its out-degree is no longer a fixed number $m$, but an i.i.d.\ random variable whose support is at least $-\delta$. Various aspects of this model are  investigated in \cite{DeiEskHofHoo09,GaovdV17,GarHazHofRay22}.

    \item \textit{The PAM with additive fitness.} In this model, each vertex has an i.i.d.\ fitness parameter. In \cite{Jor18}, Jordan examined the case where the fitness parameter is determined by the vertex's color, with its distribution influenced by the colors of the other endpoints of its out-edges. In \cite{Lo21,LoOr20}, first Lodewijks and Ortgiese, followed by Y.Y.\ Lo,  studied aspects of this model.

    \item \textit{Bernoulli PAM.} In this model, the new vertex $n+1$ is connected to {\em each} vertex $i\in [n]$ conditionally independently with probability $f(D_i^{\sss(\text{in})}(n))/n$ for some given function $f$, where $D_i^{
    \sss(\text{in})}(n)$ represents the in-degree of vertex $i$. Hence, the number of edges in this random graph is \textit{not} fixed. This model is investigated in \cite{DerMor09,DerMor11,DerMor13}.
\end{itemize}
\zhu{We believe that our results can be a foundation for the proof of these models.}


\paragraph{Organization} The structure of this paper is as follows: In \Cref{sec_proof_overview}, we give an overview of the proof of \Cref{thm-log-PA-delta>0}. The proofs are presented in Sections \ref{sec_lower_bound}--\ref{sec-second-moment-proof-thm-log-PA-delta>0}, using first- and second-moment methods on the number of paths of appropriate number of steps, and between appropriate sets of vertices. This transforms \ch{our} probability question into a \zhu{combinatorics} problem of computing the asymptotic number of such paths. 
In \Cref{sec_lower_bound}, we examine the upper bound on the expected number of paths between two uniformly chosen vertices, $\Ver_1$ and $\Ver_2$, and use this to establish the lower bound on the typical distances. 
\Cref{sec_converge_spectral_radius} focuses on demonstrating the convergence of the spectral radius of the offspring operator under truncation of the vertices' age, using a probabilistic interpretation of this functional analysis problem. This is one of the novel aspects that makes the proof for PAMs significantly more challenging than the proofs for inhomogeneous random graphs and configuration models, as the offspring operator of PAM is non-self-adjoint. 
With this convergence, in Section \ref{sec-first-moment-proof-thm-log-PA-delta>0}, we establish a lower bound on the expected number of paths connecting vertices at distance $r$ from $\Ver_1$ and vertices at distance $r$ from $\Ver_2$, conditionally on their $r$-neighborhoods.
Finally, in Section \ref{sec-second-moment-proof-thm-log-PA-delta>0}, we complete the proof by deriving an upper bound on the conditional variance given the two neighborhoods. 

\paragraph{Asymptotic symbols and abbreviations} In this paper, we use the standard Landau notation $O(1)$ to denote a \ch{function of $n$} whose absolute value has a uniform upper bound for all $n$, while $o(1)$ denotes a \ch{function of $n$ that} converges to $0$ as $n\to\infty$. From now on, we will abbreviate $\PAndmdel=\PAndmdel^{\sss (m,\delta)}$. We further write lhs and rhs for left- and right-hand side, and wrt for with respect to. Finally, for a sequence of events $(\ch{\mathcal{E}_n})_{n\geq 1}$, we write that $\mathcal{E}_n$ occurs  whp when $\prob(\mathcal{E}_n)\rightarrow 1.$

\section{Proof overview}\label{sec_proof_overview}
In this section, we give an overview of the proof of \Cref{thm-log-PA-delta>0}, which is divided into two parts: the upper and \ch{the lower bound}, that is,  for any $\vep\in(0,1)$, as $n\rightarrow \infty$,
\ch{we prove that}
\eqn{
	\label{conc-LB-dist-PAd}
	\prob\Big(\gdist{\PAndmdel}(\Ver_1, \Ver_2)\geq (1-\vep)\log_\nu{n}\Big)\rightarrow 1,
	}
 and
\eqn{
	\label{conc-UB-dist-PAd}
	\prob\Big(\gdist{\PAndmdel}(\Ver_1, \Ver_2)\leq (1+\vep)\log_\nu{n}\Big)\rightarrow 1.
	}
These two bounds are derived using path-counting techniques, specifically by calculating the expected number of paths between $\Ver_1$ and $\Ver_2$ for the upper and lower bounds, and bounding the variance for the upper bound on the typical distance. 

We note that \eqref{conc-LB-dist-PAd} shows that it is  unlikely to find a path of length at most $(1-\vep)\log_\nu{n}$ in $\PAndmdel$, while \eqref{conc-UB-dist-PAd} demonstrates that it is likely to find one with length at most $(1+\vep)\log_\nu{n}$. To convert results on the expectation into results on the probability, we apply first- and second-moment methods. 

This section is organized as follows: In \Cref{sec_polya_pam}, we introduce the Pólya urn representation of $\PAndmdel$, which allows us to calculate the expected
number of paths between any two vertices in \Cref{sec-edge-set-path-prob}, ultimately yielding the lower bound in \eqref{conc-LB-dist-PAd}. In \Cref{sec-log-upperbds-PA}, we explain how the second-moment method is used in the proof of the upper bound \eqref{conc-UB-dist-PAd}, and why the vertex ages along the path should be truncated appropriately. In \Cref{sec-proof-thm-2.1}, we prove \Cref{thm-log-PA-delta>0} based on the stated propositions.
Finally, in \Cref{sec_offspring_operator}, we explain how the spectral radius $\nu$ relates to typical distances.

\subsection{The Pólya urn representation of $\PAndmdel$}\label{sec_polya_pam}
One may notice from \eqref{eq-def-edge-prob-pand} that the probability of an edge connecting to a vertex depends on the current degree of that vertex. This dependence models the \textit{rich-get-richer} phenomenon in real-world networks, but it also makes the model quite involved. A less obvious property from \eqref{eq-def-edge-prob-pand} is that if we condition on the degree of vertex $i$ in $\PAndmdel$, the incoming edges of $i$ are chosen uniformly without replacement from all the out-edges of younger vertices,  which is similar in spirit to a Pólya urn.  This conditional independence is subtly embedded in \eqref{eq-def-edge-prob-pand}. Indeed, the degrees of the vertices are {\em not} independent of each other. An alternative description of $\PAndmdel$, with additional randomness, is such that edges are \ch{{\em conditionally} independent given the additional randomness, which can be interpreted as random vertex strengths}. This Pólya urn representation of $\PAndmdel$, originally proved in \cite{MR2298278}, provides this rather elegant construction. 
In this section, we introduce \ch{the} Pólya urn graph representation, as presented in \cite{GarHazHofRay22}.

\paragraph{Pólya urn graph $\PU$ with weights $(\psi_j)_{2\leq j\leq n}$} Let $\psi_1=1$ and $(\psi_j)_{2\leq j\leq n}$ be independent beta random
variables with certain parameters $\alpha_j$ and $\beta_j$. Define 
\eqan{
\varphi_j^{\sss(n)}=\psi_j\prod_{i=j+1}^n(1-\psi_i)\quad\text{and}\quad S_\ell^{\sss(n)}=\sum_{j=1}^\ell\varphi_j^{\sss(n)}=\prod_{i=\ell+1}^n(1-\psi_i).
}
Then, conditional on the weights $\psi_2,\ldots,\psi_n$, with $I_j^{\sss(n)}=[S_{j-1}^{\sss(n)},S_j^{\sss(n)})$ such that $|I_j^{\sss(n)}|=\varphi_j^{\sss(n)}$, the Pólya urn graph $\PU$  with weights $(\psi_j)_{2\leq j\leq n}$ is defined as follows:
\begin{enumerate}
    \item  Let $(U_{\ell,i})_{\ell\in[n],i\in[m]}$ be mutually independent random variables, uniformly distributed on $[0,1]$.

\item For each integer $\ell\in[2,n]$ and $i\in [m]$, with $I_j^{\sss(n)}$ being the interval containing $S_{\ell-1}^{\sss(n)} U_{\ell,i}$, add an edge labeled $i$ from $\ell$ to $j$. Note that $j\leq \ell-1$, which follows from the definition of $I_j^{\sss(n)}$.\label{it-def-pu-2}
\end{enumerate}

By \cite[Theorem 5.10]{Hofs24}, when we choose the weights $(\psi_j)_{2\leq j\leq n}$ such that $(\psi_j)_{2\leq j\leq n}$ are independent beta random
variables with parameters 
\begin{align}\label{eq-alpha-beta-pand}
    \alpha_j=\alpha=m+\delta\quad \text{and}\quad \beta_j=(2j-3)m+\delta(j-1), 
\end{align}
$\PU$ has the same distribution as the $\PAndmdel$ defined in \Cref{sec_mod-def}.
Furthermore, the construction of $\PU$ ensures that the edges in this graph are \textit{independent}, conditionally on $(\psi_j)_{2\leq j\leq n}$, which will turn out to be very convenient.
Consequently, in the remainder of this paper, unless specified otherwise, we work with this Pólya urn graph with the parameters as in \eqref{eq-alpha-beta-pand}.

\begin{rem}[Pólya urn graph allowing self-loops]\label{re-pama-spu}
Another Pólya urn graph, $\SPU$, is introduced in \cite{GarHazHofRay22}, where self-loops are allowed. This graph is defined in almost the same way as $\PU$, except that item \ref{it-def-pu-2} in the definition of $\PU$ is replaced by the following:
\begin{itemize}
    \item[$\rhd$] For each integer $\ell\in[2,n]$ and $i\in [m]$, with $I_j^{\sss(n)}$ being the interval containing $S_{\ell}^{\sss(n)} U_{\ell,i}$ (instead of $S_{\ell-1}^{\sss(n)} U_{\ell,i}$), we add an edge labeled $i$ from $\ell$ to $j$. Note that $j\leq \ell$, which follows from the definition of $I_j^{\sss(n)}$.
\end{itemize}
This Pólya urn graph $\SPU$ is used in  \longversion{\Cref{sec-td-other-pams}}\shortversion{ \cite[Appendix C]{HofZhu25long}} to derive the typical distances for $\PAnamdel$.
\end{rem}

\subsection{Edge set and path probabilities}\label{sec-edge-set-path-prob}
With the Pólya urn representation of $\PAndmdel$ in  \Cref{sec_polya_pam}, we aim to investigate the probability of the presence of a path in $\PAndmdel$ using the \textit{conditional independence} of the edges in the path. First, we introduce the notation for an edge in $\PAndmdel$.

\begin{definition}[Edge]
In this paper, for integers $1\leq y<x\leq n$ and $i\in[m]$, we denote our edges in terms of three coordinates, that is, we use $e=(x,i,y)$ to represent that the $i$th out-edge of $x$ is connected to $y$. To emphasize that $e$ is an edge from $x$ to $y$, we further denote  
\begin{align*}
    \bar{e}=x \quad\text{and}\quad  \ushort{e}=y
\end{align*}
for the younger and older endpoints of $e$, respectively.\end{definition}

Note that if $(x,i,y)$ forms an edge in $\PAndmdel$, then $(x,i,z)$ can not be present in the same graph, unless $y=z$. This property leads to the following definition of the edge set:
\begin{definition}[Edge set]\label{def:edge-set}
    For any $k\geq 0$, we call the set $$E=\cbc{(\ell_h,i_h,j_h) \in [n]\times [m]\times [n] :~h\in[k]}$$ an edge set if  $j_h\leq \ell_h-1$ for each $h\in [k]$ and $(\ell_h,i_h)\neq (\ell_{h'},i_{h'})$ for each $h'\neq h$.
\end{definition}
With the notation of edge sets, we next give a formal definition of a path in $\PAndmdel$:

\begin{definition}[Path and edge-labeled paths]\label{hd_def_labeled_self_avoiding_path}
For each non-negative $k$,
\begin{itemize}
    \item[$\rhd$] a $k$-step path $\vec\pi=(\pi_0,\pi_1,\ldots,\pi_k)\in [n]^{k+1}$ is a $(k+1)$-dimensional vector of vertices in $[n]$.
    We write $\vec\pi\subseteq\PAndmdel$ to indicate that there is an edge between $\pi_{h-1}$ and $\pi_h$ for each $h\in [k]$ in $\PAndmdel$;

    \item[$\rhd$] a $k$-step edge-labeled path $\vpie=\cbc{(\pi_{h-1}\vee \pi_h,i_h,\pi_{h-1}\wedge \pi_h):~h\in[k]}$
is a multiset of edges, while its underlying set is an edge set. We write $\vpie\subseteq\PAndmdel$ to indicate that $\pi_{h-1}$ and $\pi_h$ are connected by the $i$th out-edge of $\pi_{h-1}\vee \pi_h$  in $\PAndmdel$. 
\end{itemize}

We further say that 
\begin{itemize}
 \item[$\rhd$] both $\vec\pi$ and $\vpie$ are neighbor-avoiding if $\pi_{i-1}\neq \pi_i$ for each $i\in [k]$;
\item[$\rhd$] both $\vec\pi$ and $\vpie$ are self-avoiding if each vertex occurs at most once in the vector
 $(\pi_0,\ldots,\pi_k)$.
\end{itemize}
\end{definition}


Denote the law of $\PAndmdel$ conditionally on $(\psi_j)_{j\in[n]}$ by $\PP_n$. Then, by the Pólya urn representation of $\PAndmdel$ in \Cref{sec_polya_pam},
\eqan{\label{eq_con_prob_present}
\PP_n\bc{(\ell,i,j)\in \PAndmdel}=\PP_n\bc{S_{\ell-1}^{\sss(n)} U_{\ell,i}\in I_j^{\sss(n)}}=\frac{\varphi_j^{\sss(n)}}{S_{\ell-1}^{\sss(n)}}=\psi_j\prod_{i=j+1}^{\ell-1}(1-\psi_i).
}
Thus, loosely speaking, one can view $\varphi_j^{\sss(n)}$ as the weight of vertex $j$ in an inhomogeneous random graph, while $S_{\ell-1}^{\sss(n)}$ is the normalization constant. Note, however, that this is {\em not} an inhomogeneous random graph in the sense of Bollob\'as et al.\ \cite{BolJanRio07}, since each vertex has precisely $m$ out-edges.
\smallskip

Recall from \Cref{sec_polya_pam} that the uniform random variables $(U_{\ell,i})_{\ell\in[n],i\in[m]}$ are mutually independent. Hence, given an edge set $E=\cbc{(\ell_h,i_h,j_h)\in [n]\times [m]\times [n]:~h\in[k]}$, the edge probability in \eqref{eq_con_prob_present} yields that
{
\eqan{\label{eq_con_indpedent_present}
\PP_n(E \subseteq\PAndmdel)&=\prod_{h\in[k]} \PP_n\bc{S_{\ell_h-1}^{\sss(n)} U_{\ell_h,i_h}\in I_{j_h}^{\sss(n)}}=\prod_{s=2}^n \psi_s^{\psE }(1-\psi_s)^{\qsE },
}
}
where 
\eqan{\label{def_hd_pq_gen}\psE =\sum_{e\in E}\indic{s=\ushort{e}} \quad\text{and}\quad
\qsE =\sum_{e\in E}\indic{s\in (\ushort{e},\bar{e})} \quad\text{for $s\in [n]$},
}
and we use $\psi_1=1$ in the last equation of \eqref{eq_con_indpedent_present}. Specifically, for an edge-labeled self-avoiding path $\vpie$,
\begin{align}\label{eq_hd_lpath_probability}
    \PP_n\bc{\vpie\subseteq \PAndmdel}=\prod_{s=2}^n \psi_s^{p_s}(1-\psi_s)^{q_s},
\end{align}
where now 
\eqan{\label{def_hd_pq}p_s=p_s^{\sss \vec\pi}=\sum_{e\in \vpie}\indic{s=\ushort{e}}\quad \text{and}\quad
q_s=q_s^{\sss \vec\pi}=\sum_{e\in \vpie}\indic{s\in (\ushort{e},\bar{e})}.
}

The proofs of both the lower bound (\ref{conc-LB-dist-PAd}) and the upper bound (\ref{conc-UB-dist-PAd}) rely on (\ref{eq_hd_lpath_probability}). Indeed, if we  sum \eqref{eq_hd_lpath_probability} over all possible edge-labeled self-avoiding paths with $\pi_0=\Ver_1$ and $\pi_k=\Ver_2$, and take the expectation wrt $(\psi_j)_{j\in[n]}$, we can bound the probability of the distance between $\Ver_1$ and $\Ver_2$ being $k$ from above using a first-moment argument as follows: 
\begin{align}\label{eq-first-moment-po}
     \prob\big(\gdist{\PAndmdel}(\Ver_1, \Ver_2)=k\big)&\leq \Erw\Big[\sum_{\vpie\in \BESA_{\Ver_1,\Ver_2}^{\sss e,k}}\PP_n\bc{\vpie\subseteq \PAndmdel}\Big]\nn\\
     &=\frac{1}{n^2}\sum_{\nVer_1,\nVer_2\in[n]}\sum_{\vpie\in \BESA_{\nVer_1,\nVer_2}^{\sss e,k}}\Erw\Big[\prod_{s=2}^n \psi_s^{p_s}(1-\psi_s)^{q_s}\Big],
\end{align}
where $\BESA_{\nVer_1,\nVer_2}^{\sss e,k}$ is the set of all possible $k$-step edge-labeled self-avoiding path between $\nVer_1$ and $\nVer_2$ in \zhu{$[n]$}.
Thus, if the expectation in \eqref{eq-first-moment-po} grows exponentially in $k$ as $\nu^k$, then \eqref{conc-LB-dist-PAd} follows from \eqref{eq-first-moment-po} by summing over $k\leq (1-\vep)\log_\nu{n}$. 
Indeed, we will prove the following lower bound on the typical distances:

\begin{proposition}[Lower bound on the typical distances]\label{pro_lowerbound}
    Fix $m\geq 2$, $\delta>0$ and $\vep\in(0,1)$, and consider $\PAndmdel$. Then, for any $k\leq (1-\vep)\log_\nu{n}$, as $n\to\infty$,
    \begin{align*}
        \prob\big(\gdist{\PAndmdel}(\Ver_1, \Ver_2)\leq k\big)\to 0.
    \end{align*}
\end{proposition}

The argument for (\ref{conc-UB-dist-PAd}) is considerably more complicated, and we will elaborate on it in the next section.

\subsection{Logarithmic upper bound on paths: Second-moment method}
\label{sec-log-upperbds-PA}
In this section, we describe  the proof strategy for the upper bound \eqref{conc-UB-dist-PAd}. Let $B_r^{\sss(G)}(x)$ be the $r$-neighborhood of a vertex $x$ in the graph $G$, and $\bar{B}_r^{\sss(G)}(x)$ be \ch{the} edge-labeled neighborhood \ch{of $x$}, i.e., the neighborhood that includes the labels of the edges. \zhu{Let $\partial B_r^{\sss(G)}(x)$ denote} the set of vertices at distance exactly equal to $r$ from $x$ in $G$.
\smallskip

Fix $G_n=\PAndmdel$ and \ch{let}  $k_\nu^\star=\rou{(1+\vep)\log_\nu{n}}$ \ch{denote} the target upper bound. In order to prove that $\gdist{\PAndmdel}(\Ver_1,\Ver_2)\leq k_\nu^\star$ holds whp, we use a conditional second-moment method on the number of paths of length $k_\nu^\star$, between $\partial B_r^{\sss(G_n)}(\Ver_1)$ and $\partial B_r^{\sss(G_n)}(\Ver_2)$, \ch{conditional} on $\bar{B}_r^{\sss(G_n)}(\Ver_1)$ and $\bar{B}_r^{\sss(G_n)}(\Ver_2)$. 
\smallskip

The reason that we consider a \textit{conditional} second-moment method is that there is inherent variability in $|\partial B_r^{\sss(G_n)}(\Ver_1)|$ and $|\partial B_r^{\sss(G_n)}(\Ver_2)|$, due to the variability of the degrees of the vertices close to $\Ver_1$ and $\Ver_2$, which influences the number of $k$-step paths between $\Ver_1$ and $\Ver_2$. As such, one cannot expect the number of paths between $\Ver_1$ and $\Ver_2$ to be sufficiently concentrated. However, this variability diminishes when we condition on $\partial B_r^{\sss(G_n)}(\Ver_1)$ and $\partial B_r^{\sss(G_n)}(\Ver_2)$ with $r$ large, in the same manner as the variability of the branching process generation size $Z_k$ of generation $k$, \ch{conditional} on $Z_r$, becomes small for large $r$ and even larger $k$. Therefore, we leverage the {\em local convergence} properties of $B_r^{\sss(G_n)}(\Ver_1)$ and $B_r^{\sss(G_n)}(\Ver_2)$, along with the appropriate {\em conditional second-moment method}, to show that $\gdist{\PAndmdel}(\Ver_1,\Ver_2)\leq k_\nu^\star+2r$ whp.
\smallskip

Let us now explain in more detail what such a conditional second-moment method looks like for $\PAndmdel$. Let $\mathscr{F}_{r}$ denote the sigma-algebra generated by $\bar{B}_r^{\sss(G_n)}(\Ver_1)$ and $\bar{B}_r^{\sss(G_n)}(\Ver_2)$, as well as $\psi_v$ for $v\in B_{r-1}^{\sss(G_n)}(\Ver_1)\cup B_{r-1}^{\sss(G_n)}(\Ver_2)$, that is,
	\eqn{
	\label{sigma-alg-SMM-PAM}
	\mathscr{F}_{r}=\sigma\Big(\bar{B}_r^{\sss(G_n)}(\Ver_1), \bar{B}_r^{\sss(G_n)}(\Ver_2), (\psi_v)_{v\in B_{r-1}^{\sss(G_n)}(\Ver_1)\cup B_{r-1}^{\sss(G_n)}(\Ver_2)}\Big).
	}
We let $\prob^r$ and $\expec^r$ be the conditional probability and expectation given $\mathscr{F}_{r}$. 
\smallskip

In \Cref{sec_polya_pam}, we have shown that edges are present independently, \ch{conditional} on $(\psi_v)_{v\in[n]},$ but they do depend sensitively on $(\psi_v)_{v\in[n]}.$ In particular, for the probability of an edge to $v$, we obtain a factor that is proportional to $\psi_v$ (recall \eqref{eq_con_prob_present}). Thus, if we were to condition on $(\psi_v)_{v\in B_{r}^{\sss(G_n)}(\Ver_1)\cup B_{r}^{\sss(G_n)}(\Ver_2)}$ rather than on $(\psi_v)_{v\in B_{r-1}^{\sss(G_n)}(\Ver_1)\cup B_{r-1}^{\sss(G_n)}(\Ver_2)}$, one could expect the conditional probability of $\vpie\subseteq \PAndmdel$ to depend on $\psi_{\pi_0}$ and $\psi_{\pi_k}$, with $\pi_0\in \partial B_{r}^{\sss(G_n)}(\Ver_1)$ and $\pi_k\in \partial B_{r}^{\sss(G_n)}(\Ver_2)$. When conditioning on $(\psi_v)_{v\in B_{r-1}^{\sss(G_n)}(\Ver_1)\cup B_{r-1}^{\sss(G_n)}(\Ver_2)}$, instead, the conditional probability of $\vpie\subseteq \PAndmdel$ becomes, to leading order, {\em deterministic}. This explains why we condition on the precise sigma-algebra in \eqref{sigma-alg-SMM-PAM}, particularly why we only condition on the weights in $(\psi_v)_{v\in B_{r-1}^{\sss(G_n)}(\Ver_1)\cup B_{r-1}^{\sss(G_n)}(\Ver_2)}$. 

To ensure the necessary concentration, we define the following good events on the pair of neighborhoods $\big(\bar{B}_r^{\sss(G_n)}(\Ver_1), \bar{B}_r^{\sss(G_n)}(\Ver_2)\big)$ and the weights $(\psi_v)_{v\in B_{r-1}^{\sss(G_n)}(\Ver_1)\cup B_{r-1}^{\sss(G_n)}(\Ver_2)}$ to exclude pathological behavior:
\begin{definition}[Good neighborhood pair]\label{ass_Br}
Given the tree truncation parameter $\eta$, the leaf truncation parameter $\zeta$, the size upper bound $M$ and the depth parameter $r$ such that $\eta,\zeta \in(0,1/10]$, and the integers $M,r\geq 1$, we call the neighborhood pair $\big(\bar{B}_r^{\sss(G_n)}(\Ver_1), \bar{B}_r^{\sss(G_n)}(\Ver_2)\big)$  {\rm good} if it is a pair of labeled trees $(\bar{\bf t}_1,\bar{\bf t}_2)$ satisfying the following properties:
\begin{enumerate}[label=(\alph*)]
    \item the vertices in trees $\bar{\bf t}_1$ and $\bar{\bf t}_2$ are labeled with distinct integers from $[\eta n,n]$;\label{it_pb_1}
        \item for each $i=1,2$, the number of vertices in $\bar{\bf t}_i$ is no more than $M$;\label{it_pb_2}
    \item for each $i=1,2$, at least $(3/2)^r$ leaves of $\bar{\bf t}_i$ are labeled in $[4{\zeta}n,(1-3{\zeta})n]$;\label{it_pb_3}
            \item for each $i=1,2$, the tree $\bar{\bf t}_i$ has a root $\varnothing_i$ where each leaf in $\bar{\bf t}_i$ has a distance $r$ from $\varnothing_i$;\label{it_pb_4}
        \item the edges in $\bar{\bf t}_1$ and $\bar{\bf t}_2$ are labeled in $[m]$;\label{it_pb_5}
            \item  for each $i=1,2$,  each non-leaf vertex $v$ in  $\bar{\bf t}_i$ and each $j\in[m]$, there is precisely one \zhu{edge-labeled} $j$ such that $v$ is the younger endpoint of this edge.\label{it_pb_6}
\end{enumerate}
\end{definition}
\begin{definition}[Good weights]\label{ass_Br_2}
We call the weights $(\psi_v)_{v\in B_{r-1}^{\sss(G_n)}(\Ver_1)\cup B_{r-1}^{\sss(G_n)}(\Ver_2)}$ {\rm good} when
\begin{align}\label{eq-ass-on-psi}
    1-\sum_{v\in B_{r-1}^{\sss(G_n)}(\Ver_1)\cup B_{r-1}^{\sss(G_n)}(\Ver_2)}\psi_v\geq \e^{-\frac{1}{\log^2 n}}.
\end{align}
\end{definition}
\begin{definition}[Good event $\GNW$]\label{def_gnw}
    Given $\Ver_1$, $\Ver_2$, $r$, $M$, $\eta$ and ${\zeta}$ in \Cref{ass_Br}, we define the {\rm good event} $\GNW$ as the event that both $\big(\bar{B}_r^{\sss(G_n)}(\Ver_1), \bar{B}_r^{\sss(G_n)}(\Ver_2)\big)$ and $(\psi_v)_{v\in B_{r-1}^{\sss(G_n)}(\Ver_1)\cup B_{r-1}^{\sss(G_n)}(\Ver_2)}$ are good.
\end{definition}

We note that \ref{it_pb_4}-\ref{it_pb_6} in \Cref{ass_Br} are necessary conditions for trees $\bar{\bf t}_1$ and $\bar{\bf t}_2$ to be the edge-labeled $r$-neighborhoods of $\Ver_1$ and $\Ver_2$. \zhu{\ch{These conditions} only require non-leaf vertices to have precisely $m$ older neighbors, but \ch{they do not impose sharp restrictions} on the number of younger neighbors. As a result, \ch{the restrictions do not rule out} the tree $\bar{\bf t}_1$ \ch{being a} proper subgraph of \ch{the} neighborhood of 
$o_1$. \ch{Below, we impose extra conditions to avoid additional younger neighbors.}} 

  By the \textit{local convergence} of $\PAndmdel$, which we will discuss in more detail in \Cref{sec_offspring_operator}, with $\Ver_1$ and $\Ver_2$ chosen uniformly at random from $[n]$, we claim that $\big(\bar{B}_r^{\sss(G_n)}(\Ver_1), \bar{B}_r^{\sss(G_n)}(\Ver_2)\big)$ and $(\psi_v)_{v\in B_{r-1}^{\sss(G_n)}(\Ver_1)\cup B_{r-1}^{\sss(G_n)}(\Ver_2)}$ in $\PAndmdel$ are good with probability close to 1:
\begin{lemma}[Good events have high probability]
\label{LEM_RE_BR}
Fix $m\geq 2$, $\delta>0$ and the leaf truncation parameter  ${\zeta}\in(0,1/10]$. With $\Ver_1,\Ver_2$ uniformly distributed on $[n]$,
\begin{align*}
\limsup_{r\to\infty}\limsup_{M\to\infty}\limsup_{\eta\searrow 0}\limsup_{n\to\infty} \PP\bc{\GNW}=1.
\end{align*}
\end{lemma}
We further note that the order of the limits in \Cref{LEM_RE_BR} is important, as the rhs becomes $0$ when we exchange the two limits ${r\to\infty}$ and ${M\to\infty}$. \longversion{The proof of \Cref{LEM_RE_BR} is deferred to \Cref{sec-proof-ass-Br}.}\shortversion{We refer to \cite[Appendix A]{HofZhu25long} for the proof of \Cref{LEM_RE_BR}.}
\smallskip


We denote the number of all $k$-step self-avoiding edge-labeled paths between $\partial B_r^{\sss(G_n)}(\Ver_1)$ and $\partial B_r^{\sss(G_n)}(\Ver_2)$, whose vertices are restricted to $[{\zeta} n,n]$, by 
	\eqn{
	\label{Nnk-vep-PAM}
	N_{n,r}(k)=\sum_{\vpie\in \NESA^{\sss e,k,r}} \indic{\vpie\subseteq \PAndmdel},
	}
where $\NESA^{\sss e,k,r}$ is the set of all possible  $k$-step edge-labeled self-avoiding paths $\vpie$ with $\pi_0\in \partial B_r^{\sss(G_n)}(\Ver_1)$ and $\pi_k\in \partial B_r^{\sss(G_n)}(\Ver_2)$ in \zhu{$[n]$}, and we restrict $\pi_i$ such that $\pi_i\not\in B_r^{\sss(G_n)}(\Ver_1)\cup B_r^{\sss(G_n)}(\Ver_2)$ for $i\in[k-1]$ and $\pi_i\geq {\zeta} n$ for all $0\leq i\leq k$. 
\smallskip

Here, we only consider $\vpie$ such that $\pi_i\geq {\zeta} n$ for all $ 0\leq i\leq k$ to ensure the concentration of $N_{n,r}(k)$ for $k=k_\nu^\star$. The underlying reason is that vertices with lower indices tend to have higher degrees on average. Once a vertex with a very small index appears in one of the $k$-step self-avoiding paths between $\partial B_r^{\sss(G_n)}(\Ver_1)$ and $\partial B_r^{\sss(G_n)}(\Ver_2)$, it implies the presence of numerous other $k$-step self-avoiding paths between them, leading to a variance that can be too large.

\smallskip

In order to prove \eqref{conc-UB-dist-PAd}, we rely on Chebyshev's inequality, which  gives
\begin{align}\label{eq_secondmomentbound}
    \prob^r\big(N_{n,r}(k_\nu^\star)=0\big)\leq \frac{\Var^r(N_{n,r}(k_\nu^\star))}{\expec^r[N_{n,r}(k_\nu^\star)]^2},
\end{align}
where $\expec^r$ and $\Var^r$ denote the conditional expectation and variance given $\mathscr{F}_{r}$ (recall \eqref{sigma-alg-SMM-PAM}). If we can show that the rhs of \eqref{eq_secondmomentbound} is bounded by a small number $\vep_0$, then \eqref{eq_secondmomentbound} yields that $N_{n,r}(k_\nu^\star)>0$ with probability at least $1-\vep_0$. Therefore, we also have $\gdist{\PAndmdel}(\Ver_1, \Ver_2)\leq k_\nu^\star+2r$, with probability at least $1-\vep_0$. 

Often, to show that the rhs of \eqref{eq_secondmomentbound} is negligible, it is necessary to prove that the first moment converges to infinity. Before turning to the statement of this result, we first introduce an important notation in this article:
\begin{definition}[A variable with uniform positive upper and lower bound.]\label{def-Theta}
    We use $\Theta$ to denote a number in $[\Theta_\star,\Theta^\star]$, where $\Theta_\star\in (0,1)$ and $\Theta^\star\in (1,\infty)$ are certain fixed positive numbers that only depend on $m$ and $\delta$. By a slight abuse of notation, we allow bounds like $2\Theta=\Theta$, $\e^{\Theta}=\Theta$, $\Theta^{-1}=\Theta$ and $\Theta\cdot\Theta=\Theta$.
\end{definition}
In \Cref{sec-first-moment-proof-thm-log-PA-delta>0}, we prove the following lower bound on $\expec^r[N_{n,r}(k_\nu^\star)]$: 
\begin{proposition}[First moment is large \ch{on the good event}]
\label{prop-path-count-PAMc-LB_simple}
Fix $m\geq 2$, $\delta>0$ and $\vep\in(0,1)$, and consider $\PAndmdel$. There exists a $\zeta_0 \in(0,1/10]$ such that for $k_\nu^\star=\rou{(1+\vep)\log_{\nu}{n}}$, every ${\zeta}\in(0,\zeta_0]$,  $r,M\geq 1$,  $\eta\in(0,1/10]$, on the good event $\GNW$, for  $n\geq n_0$, where $n_0$    is an integer that depends only on $m,\delta,M,r,\vep$ and $\zeta$, 
	\eqn{
	\label{Nnk-PAMc-infty}
	\expec^r[N_{n,r}(k_\nu^\star)] \geq \Theta \bc{3/2}^{2r}{\zeta}^2c_{\zeta}n^{\vep/2},
	}
where $c_{\zeta}>0$ only depends on the leaf truncation parameter $\zeta$.
\end{proposition}

We next bound the conditional variance from above in \Cref{sec-second-moment-proof-thm-log-PA-delta>0}:
\begin{proposition}[Variance is negligible compared to the first moment squared]
\label{prop-path-count-PAMc-Var-UB}
Fix $m\geq 2$, $\delta>0$ and $\vep\in(0,1)$, and consider $\PAndmdel$. There exists a $\zeta_0 \in(0,1/10]$ such that for $k_\nu^\star=\rou{(1+\vep)\log_{\nu}{n}}$, every ${\zeta}\in(0,\zeta_0]$,  $r,M\geq 1$,  $\eta\in(0,1/10]$, on the good event $\GNW$, for  $n\geq n_0$, where $n_0$    is an integer that depends only on $m,\delta,M,r,\vep$ and $\zeta$, with $c_{\zeta}>0$ defined as in \Cref{prop-path-count-PAMc-LB_simple}, 
	\eqn{
	\Var^r(N_{n,r}(k_\nu^\star))\leq \Theta \bc{3/2}^{-r} {\zeta}^{-4} c_{\zeta}^{-1} \expec^r[N_{n,r}(k_\nu^\star)]^2.\nn
	}
\end{proposition}

\subsection{Proof of main result \Cref{thm-log-PA-delta>0} subject to \Cref{pro_lowerbound,prop-path-count-PAMc-Var-UB}}\label{sec-proof-thm-2.1}
In this section, we prove \Cref{thm-log-PA-delta>0} using \Cref{pro_lowerbound,prop-path-count-PAMc-Var-UB}.

\begin{proof}[Proof of \Cref{thm-log-PA-delta>0} subject to \Cref{pro_lowerbound,prop-path-count-PAMc-Var-UB}]
We first note that \eqref{conc-LB-dist-PAd} follows directly from \Cref{pro_lowerbound}. To show that \eqref{conc-UB-dist-PAd} holds, we use the Chebyshev’s inequality, as mentioned in \Cref{sec-log-upperbds-PA}. 
\smallskip

For an event $A$, we use $\indicwo{A}$ to denote the indicator function of $A$. By \eqref{eq_secondmomentbound} and \Cref{prop-path-count-PAMc-Var-UB}, for any ${\zeta}\in(0,\zeta_0]$, 
\begin{align}\label{eq-limit-prob-nk}
    &\limsup_{r\to\infty}\limsup_{M\to\infty}\limsup_{\eta\searrow 0}\limsup_{n\to\infty}\prob^r\big(N_{n,r}(k_\nu^\star)=0\big)\indicwo{\GNW}\\
    &\qquad\leq \limsup_{r\to\infty}\limsup_{M\to\infty}\limsup_{\eta\searrow 0}\limsup_{n\to\infty}\frac{\Var^r(N_{n,r}(k_\nu^\star))}{\expec^r[N_{n,r}(k_\nu^\star)]^2}\indicwo{\GNW}\nn\\
    &\qquad\leq \Theta{\zeta}^{-4} c_{\zeta}^{-1}\limsup_{r\to\infty}\limsup_{M\to\infty}\limsup_{\eta\searrow 0}\limsup_{n\to\infty} \bc{3/2}^{-r} =0.\nn
\end{align}
On the other hand,  the law of total probability gives that
\begin{align}\label{eq-law-tp-decompose}
   \prob\big(N_{n,r}(k_\nu^\star)=0\big)\leq \Erw\brk{\prob^r\big(N_{n,r}(k_\nu^\star)=0\big)\indicwo{\GNW}}+\PP(\GNW^c).
\end{align} 
Then, by \zhu{the} dominated convergence theorem and \Cref{LEM_RE_BR}, we conclude from \eqref{eq-limit-prob-nk} and \eqref{eq-law-tp-decompose} that
\begin{align*}
&\limsup_{r\to\infty}\limsup_{n\to\infty}\prob\big(N_{n,r}(k_\nu^\star)=0\big)=\limsup_{r\to\infty}\limsup_{M\to\infty}\limsup_{\eta\searrow 0}\limsup_{n\to\infty}\prob\big(N_{n,r}(k_\nu^\star)=0\big) \\ 
&\qquad\leq  \Erw\Big[\limsup_{r\to\infty}\limsup_{M\to\infty}\limsup_{\eta\searrow 0}\limsup_{n\to\infty}\prob^r\big(N_{n,r}(k_\nu^\star)=0\big)\indicwo{\GNW}\Big]\\
&\qquad\quad+\limsup_{r\to\infty}\limsup_{M\to\infty}\limsup_{\eta\searrow 0}\limsup_{n\to\infty}\PP(\GNW^c)=0,
\end{align*}
where the first equation follows from the fact that $\prob\big(N_{n,r}(k_\nu^\star)=0\big)$ is independent of the size upper bound $M$ and the tree truncation parameter $\eta$, and the  inequality follows from Fatou's lemma.

Given $\vep\in(0,1)$, for any $\vep_0>0$, we can choose a sufficiently small $\eta$ and sufficiently large $M,r$ such that $\prob\big(N_{n,r}(k_\nu^\star)=0\big)<\vep_0$ for any $n>n_0$, where $n_0$ is defined in \Cref{prop-path-count-PAMc-Var-UB}. Then, for $n>n_0$,
\begin{align}\label{eq_dis-upper-with-vep_0}
    \PP\big(\gdist{\PAndmdel}(\Ver_1,\Ver_2)\leq (1+\vep)\log_\nu n+2r\big)\geq \prob\big(N_{n,r}(k_\nu^\star)\geq 1\big)\geq  1-\vep_0.
\end{align}
Rescaling $\vep$ in  \eqref{eq_dis-upper-with-vep_0} to be $\vep/2$ and letting $n$ be sufficiently large such that $(\vep/2)\log_\nu n\geq 2r$, we arrive at
\begin{align*}
    \PP\big(\gdist{\PAndmdel}(\Ver_1,\Ver_2)\leq (1+\vep)\log_\nu n\big)\geq 1-\vep_0.
\end{align*}
Consequently, \eqref{conc-UB-dist-PAd} follows since $\vep_0>0$ is arbitrary.

Finally, \Cref{thm-log-PA-delta>0} follows directly from the combination of \eqref{conc-LB-dist-PAd} and \eqref{conc-UB-dist-PAd}.
\end{proof}

\subsection{Exponential growth of the Pólya point tree}\label{sec_offspring_operator}
In \Cref{sec_polya_pam,sec-edge-set-path-prob,sec-log-upperbds-PA,sec-proof-thm-2.1}, 
we have outlined the proof strategy for \Cref{thm-log-PA-delta>0}. However, we did not explain how the term $\log_\nu n$ arises in the proof. In this section, we aim to provide some intuition for this, using the \textit{local limit} of PAMs.

\paragraph{Pólya point tree $\PPT$} Local convergence was first introduced by Benjamini and Schramm in \cite{BenSch01}, and independently by Aldous and Steele in \cite{AldSte04}, to describe the limiting local structure around a uniformly chosen vertex in a sequence of graphs. Specifically, if we examine the neighborhood of a uniformly chosen vertex in $\PAndmdel$, as $n$ tends to infinity, it will closely resemble the neighborhood of the root in the following recursively defined (rooted) Pólya point tree $\PPT$ \cite{GarHazHofRay22,Hofs24}:
\begin{enumerate}
\item The Pólya point tree $\PPT$ is a \textit{marked-rooted} tree arising from a multi-type branching process. To avoid confusion, we call the vertices in the $\PPT$ {\em nodes}, and reserve the word {\em vertex} for the \ch{vertices in the} graph $\PAndmdel$.
        \item Each \ch{node} in $\PPT$ is labeled by an index, using the \textit{Ulam–Harris word}. Specifically, the root is indexed by $\varnothing$; for the other nodes in the tree,  the index of the $i$th child of a node $\omega$  is recursively defined as $\omega i$.
        \item For each node $\omega$, there is a random variable $A_\omega$ in $[0,1]$, acting as its \textit{mark}. 
        Each non-root node $\omega i$, is further assigned the label $\rO$ when $A_{\omega i}<A_\omega$, and $\rY$ when $A_{\omega i}>A_\omega$. Each edge also has a \textit{mark} in $[m]$.

        \item Each node labeled $\rO$ or the root has precisely $m$ children labeled $\rY$. The $m$ edges connected to a node are marked uniformly at random without replacement from $[m]$, independent of everything else.
        
        \item Each node with label $\rY$ has precisely $m-1$ children labeled $\rY$. The $m-1$ edges connected to a node are marked uniformly at random without replacement from $[m]$, excluding the mark of the edge connecting this vertex to its parent, and independent of everything else.
        
        \item  The mark of the root $\varnothing$ is uniformly distributed in $[0,1]$. \zhu{For $\omega$, given \ch{its mark} $A_\omega$},  with $\chi=\frac{m+\delta}{2(m+\delta)}$,
        \begin{enumerate}
            \item the marks of its children $\omega i$ with label $\rO$ \ch{have} distribution  $U_{\omega i}^{1/\chi}A_\omega$, where $U_{\omega i} $ is uniformly distributed in $[0,1]$, independent of everything else;
            \item the sequence of its children with label $\rY$ is a sequence of the points of a Poisson point process on $[A_\omega,1]$ with intensity
            \begin{align*}
                \rho_\omega=\frac{\Gamma_\omega(1-\chi)}{x^{\chi}A_\omega^{1-\chi}},
            \end{align*}
            independent of everything else, where $\Gamma_\omega$ has a gamma distribution with shape $m+\delta$ and rate $1$ when $\omega$ is the root or labeled $\rY$, and shape $m+\delta+1$ and rate $1$ when $\omega$ is labeled $\rO$. These gamma random variables are also independent of each other and everything else.
        \end{enumerate}
         \end{enumerate}

\paragraph{Offspring operator of $\PPT$}
By the definition of the $\PPT$, the offspring distribution of each non-root \ch{node} is determined by its \ch{{\em type}, consisting of its} mark in $[0,1]$ and its label in $\cbc{\rO,\rY}$, according to the same law. 

In \cite[Section 5.4.3]{Hofs24}, the first author defined an integral kernel $\kappa\big((x,s), (y,t)\big)$ and
showed that $\int_0^y \kappa\big((x,s), (z,t)\big)\dint z$ represents the expected number
of children \ch{of type} $(z,t)$ such that $z\in[0,y]$ for a \ch{node of type} $(x,s)$ in $\PPT$, where
\eqn{
	\label{kappa-PAM-fin-def-form-a-rep}
	\kappa\big((x,s), (y,t)\big)=\frac{c_{st}(\indic{x>y, t=\srO}+\indic{x< y, t=\srY})}{(x\vee y)^{\chi}(x\wedge y)^{1-\chi}},
	}
while $\chi=(m+\delta)/(2m+\delta)\in (\tfrac{1}{2},1)$ and 
	\eqn{
	\label{cst-def-PAM-rep}
	c_{st}=
	\begin{cases}
	\frac{m(m+\delta)}{2m+\delta} &\text{ for }st=\rO\rO,\\
	\frac{m(m+1+\delta)}{2m+\delta} &\text{ for }st=\rO\rY,\\
	\frac{(m-1)(m+\delta)}{2m+\delta} &\text{ for }st=\rY\rO,\\
	\frac{m(m+\delta)}{2m+\delta} &\text{ for }st=\rY\rY.
	\end{cases}
	}

Given an integral kernel $K\colon \mathcal{S}\times \mathcal{S}\mapsto \R$, its {\em integral operator} $\bfT_K\colon L^2(\mathcal{S})\to L^2(\mathcal{S})$ is defined as
	\eqn{
	\label{T-kappa-def}
	(\bfT_K f)(x)=\int_{\mathcal{S}} K(x,y)f(y)\mu(\mathrm{d}y).
	}
The offspring operator of $\PPT$ is defined as the integral operator ${\bfT_\kappa}$ with the integral kernel $\kappa$. We now explain why ${\bfT_\kappa}$ is referred to as the offspring operator.

With ${\bf 1}_{([0,y],t)}(x,s)=\indic{x\in[0,y],s=t}$, \eqref{T-kappa-def} yields that
\begin{align*}
    (\bfT_\kappa {\bf 1}_{([0,y],t)})(x,s)=\int_0^y \kappa\big((x,s), (z,t)\big)\dint z.
\end{align*}
Furthermore, by induction, $(\bfT_\kappa^k {\bf 1}_{([0,y],t)})(x,s)$ represents the expected number
of \ch{nodes of type} $(z,t)$ such that $z\in[0,y]$ in the $k$th generation of a \ch{node of type} $(x,s)$ in $\PPT$. Consequently, with ${\bf 1}$ as the function identically equal to $1$, the expected number of vertices in the $k$th generation of a \ch{node} $v$, with a uniformly chosen mark in $[0,1]$ and a uniformly chosen label in $\cbc{\rO,\rY}$, is equal to $\frac{1}{2}\langle {\bf 1},\bfT_\kappa^k {\bf 1}\rangle$, 
where $\langle \cdot, \cdot\rangle$ denotes the $L^2(\rangeoy)$ inner product, \ch{and the factor $\tfrac{1}{2}$ arises from the uniform label in $\cbc{\rO,\rY}$.} Hence, loosely speaking,  the number of \ch{nodes} in the $k$th generation of the root $\varnothing$ in $\PPT$ is of the order of $\langle {\bf 1},\bfT_\kappa^k {\bf 1}\rangle$, as is the expected number of vertices at distance $k$ to a uniformly chosen vertex $v$ in $\PAndmdel$\longversion{ by \eqref{dis-pam-local-limit-0}}.

In \cite{hazra2023percolation}, Hazra, the first author  and Ray proved the following results on the $L^2(\rangeoy)$ operator norm and spectral radius of $\bfT_\kappa$:
\begin{proposition}[{\cite[Theorem 2.5 and Remark 5.1]{hazra2023percolation}}]\label{pro_spectral_radius}
As an operator in $L^2(\rangeoy)$, both the operator norm and the spectral radius of the offspring operator $\bfT_\kappa$ are equal to 
\eqn{\nu=2\frac{m(m+\delta)+\sqrt{m(m-1)(m+\delta)(m+1+\delta)}}{\delta}.\label{eq_value_nu}}
\end{proposition}
Let $r(T)$ denote the spectral radius of \ch{an} operator $T$ \ch{from $L^2(\rangeoy)$ to $L^2(\rangeoy)$}.
Then $\nu=r(\bfT_\kappa)=\lim_{k\to\infty} \lVert \bfT_\kappa^k\rVert^{1/k}$. Heuristically, $\langle {\bf 1},\bfT_\kappa^k {\bf 1}\rangle$ should be of the same order as $\nu^k$, which is also true for the  expected number of the vertices at distance $k$ to a uniformly chosen vertex $v$ in $\PAndmdel$\longversion{ by \eqref{dis-pam-local-limit-0}}. This leads to the following realization:
\begin{itemize}
    \item[$\rhd$] If $k\leq (1-\vep)\log_\nu n$, then $\nu^k$ is negligible compared to the graph size $n$. Hence, it should be very unlikely that another uniformly chosen vertex is in the $k$-neighborhood of  $v$. 
    \item[$\rhd$] If $k\geq (1+\vep)\log_\nu n$, on the other hand, then $\nu^k$ is much larger than $n$. Hence, it can be expected that another uniformly chosen vertex is in the $k$-neighborhood of  $v$. 
\end{itemize}

The path-counting technique is a method to make this argument rigorous. In a tree, the number of $k$-step self-avoiding paths from a vertex $v$ is equal to the number of vertices at distance $k$ to $v$. Since $\PAndmdel$ is locally tree-like, it follows that, loosely speaking, when $k$ is not too large, the probability that there exists a $k$-step self-avoiding path between a uniformly chosen vertex $v$ and another uniformly chosen vertex $u$ should be close to the probability that the distance between $u$ and $v$ is equal to $k$.

\paragraph{Impact of the truncation on the local limit} Recall from \eqref{Nnk-vep-PAM} that in our path-counting approach, we only consider $k$-step edge-labeled paths $\vpie$ such that $\pi_i\geq {\zeta} n$ for all $0\leq i\leq k$. Hence, we can simply remove all vertices with an age less than ${\zeta}n$ from $\PAndmdel$. The impact of this vertex removal on the local limit is that the vertices with an age-mark less than $\zeta$ in $\PPT$ are also removed. To apply the previous intuition, we need to consider the spectral radius of the new \ch{truncated} offspring operator. Thus, to obtain \Cref{thm-log-PA-delta>0}, we are required to show that this new spectral radius converges to $\nu$ as $\zeta\searrow 0$. This property is proved
in \Cref{sec_converge_spectral_radius}, where a clever isometric transformation is applied to the truncated offspring operator, allowing us to solve this deep functional analysis problem \zhu{by studying a related random-walk problem}.

We note that  $\PPT$ is a rather involved multi-type branching process. \longversion{In fact, we only use it in the proof of \Cref{LEM_RE_BR} in \Cref{sec-proof-ass-Br}.}\shortversion{In fact, it is not used in this paper; nevertheless, it is required in the proof of \Cref{LEM_RE_BR} in \cite[Appendix A]{HofZhu25long}.} The integral kernel $\kappa$ also appears straightforwardly in the later path-counting arguments in \Cref{sec_lower_bound}. The purpose of introducing this local limit is to illustrate the idea behind the typical distance computation in random graphs with finite-variance degrees, and to provide an intuitive explanation of {how the $\log_\nu n$ asymptotics arises}.

\section{Analysis of the first moment and the proof of the lower bound}\label{sec_lower_bound}
In this section, we prove \Cref{pro_lowerbound}, the lower bound on the typical distances, using the first-moment argument. A proof that utilizes the negative correlation of edge connections in preferential attachment models is presented in \cite[Section 8.5]{Hofs24}. Instead, we employ the Pólya urn representation of $\PAndmdel$ for the proof, as many of the results presented here are subsequently used in proving the upper bound {on the typical distances} as well. \footnote{Indeed, there is an alternative proof in \cite[Section 8.5]{Hofs24}, from which our proof is adapted. However, there is an error in that proof when dealing with \eqref{eq_hd_si_prod_2}, which we address in this section. See \cite{vancorrigenda} for an alternative solution of this  problem.}

Recall from \eqref{eq_hd_lpath_probability} that for each edge-labeled self-avoiding path $\vpie$,
\begin{align}\label{eq_hd_lpath_probability-rep}
    \PP_n\bc{\vpie\subseteq \PAndmdel}=\prod_{s=2}^n \psi_s^{p_s}(1-\psi_s)^{q_s}.
\end{align}
For any $\nVer_1,\nVer_2\in[n]$, if $\gdist{\PAndmdel}(\nVer_1, \nVer_2)=k$, then there exists a $k$-step edge-labeled self-avoiding path between $\nVer_1$ and $\nVer_2$. Hence, by \eqref{eq_hd_lpath_probability-rep}, along with the independence of the beta random variables $(\psi_s)_{2\leq s\leq n}$, 
\begin{align}\label{eq_hd_pre_first-moment}
  \prob\big(\gdist{\PAndmdel}(\nVer_1, \nVer_2)=k\big)&\leq \sum_{\vpie\in \BESA_{\nVer_1,\nVer_2}^{\sss e, k}}\PP\bc{\vpie\subseteq \PAndmdel}\\
  &=\sum_{\vpie\in \BESA_{\nVer_1,\nVer_2}^{\sss e,k}}  \prod_{s=2}^n\Erw\big[ \psi_s^{p_s}(1-\psi_s)^{q_s}\big],\nn
\end{align}
where we recall that $\BESA_{\nVer_1,\nVer_2}^{\sss e,k}$ is the set of all possible $k$-step edge-labeled self-avoiding paths between $\nVer_1$ and $\nVer_2$ in \ch{$[n]$}.
To handle the expectation on the rhs of \eqref{eq_hd_pre_first-moment}, we apply \cite[Lemma 5.14]{Hofs24}, which states that, for any beta random variables $\psi$ with parameters $\alpha$ and $\beta$, and any non-negative integers $p$ and $q$, 
\begin{align}\label{expctation_function_beta}
    \Erw\big[\psi^p(1-\psi)^q\big]=\frac{(\alpha+p-1)_p(\beta+q-1)_q}{(\alpha+\beta+p+q-1)_{p+q}},
\end{align}
where $(x)_m=x(x-1)\ch{\cdots} (x-m+1)$ denotes the $m$th falling factorial of $x$. 

Hence, a straightforward idea would be to bound $\prob\big(\gdist{\PAndmdel}(\nVer_1, \nVer_2)= k\big)$ by using \eqref{eq_hd_pre_first-moment} and  \eqref{expctation_function_beta}, and then taking the average over all $\nVer_1,\nVer_2\in[n]$. However, due to the possible presence of very young vertices
in the path, this bound is difficult to control. To address this issue, we introduce a positive integer $\bdt =\lceil \log^3 n\rceil$, and exclude all the paths containing any vertex whose age is younger than $\bdt$. To account for these excluded paths, we observe that the existence of such a path implies that either the distance between $\nVer_1$ and the vertex set $[\bdt -1]$, or the distance between $\nVer_2$ and $[\bdt -1]$, is no more than $k/2$.
\smallskip

Let $\ESA_{u,v}^{\sss e,\ell}=\ESA_{u,v}^{\sss e,\ell,(b)}$ be the set of all possible $\ell$-step edge-labeled self-avoiding paths $\vpie$ such that $\pi_0=u$, $\pi_\ell=v$ and $\pi_i\in[\bdt ,n]$ for all $i\in[\ell-1]$. Then,  the above observation gives that
\begin{align}
    \sum_{\nVer_1,\nVer_2=\bdt }^n\prob\big(\gdist{\PAndmdel}(\nVer_1, \nVer_2)= k\big)\label{eq_hd_divide_lower_bound_probability}
    \leq &\sum_{\nVer_1,\nVer_2=\bdt }^n\sum_{\vpie\in \ESA_{\nVer_1,\nVer_2}^{\sss e,k}} \PP\bc{\vpie\subseteq \PAndmdel}\nn\\
    &+\sum_{\nVer_1,\nVer_2=\bdt }^n\sum_{i=1}^2\prob\big(\gdist{\PAndmdel}(\nVer_i, [ \bdt -1])\leq k/2\big),
\end{align}
and we handle the two summations on the rhs of \eqref{eq_hd_divide_lower_bound_probability} separately in the following two lemmas:
\begin{lemma}[Contribution from paths without very old vertices]\label{lem_hd_lower_part1}
For any integer $\bdt \geq 2$, any $\nVer_1,\nVer_2\in [\bdt ,n]$ and any $k\leq (1-\vep)\log_\nu n$,
\begin{align*}
&\sum_{\nVer_1,\nVer_2=\bdt }^n\sum_{\vpie\in \ESA_{\nVer_1,\nVer_2}^{\sss e,k}}\PP\bc{\vpie\subseteq \PAndmdel}\leq \Theta^{1+\frac{\log^2 n}{\bdt }}n\nu^k.
\end{align*}    
\end{lemma}
\begin{lemma}[Contribution from paths with very old vertices]\label{lem_hd_lower_part2}
For any integer $\bdt \geq 2$, any $\nVer_1,\nVer_2\in [\bdt ,n]$ and any $k\leq (1-\vep)\log_\nu n$,
\begin{align*}
\sum_{i=1}^2\sum_{\nVer_1,\nVer_2=\bdt }^n\prob\big(\gdist{\PAndmdel}(\nVer_i, [ \bdt -1])\leq k/2\big)\leq \Theta^{1+\frac{\log^2 n}{\bdt }} n^{3/2}\bdt ^{1/2}\nu^{k/2}.
\end{align*}
\end{lemma}
 \Cref{pro_lowerbound} then follows from \Cref{lem_hd_lower_part1,lem_hd_lower_part2} as we demonstrate next:
\begin{proof}[Proof of \Cref{pro_lowerbound} subject to \Cref{lem_hd_lower_part1,lem_hd_lower_part2}] 
We first recall that we set $\bdt =\lceil\log^3 n\rceil$. Note that $\gdist{\PAndmdel}(\nVer_1, \nVer_2)= \ell$ implies \ch{that} either \zhu{there} exists a path $\vpie\in \ESA_{\nVer_1,\nVer_2}^{\sss e,\ell}=\ESA_{\nVer_1,\nVer_2}^{\sss e,\ell,(b)}$, or the distance between $\cbc{\nVer_1,\nVer_2}$ and the vertex set $[\bdt -1]$ is no more than $\ell/2$.
    Then, by \eqref{eq_hd_divide_lower_bound_probability} and \Cref{lem_hd_lower_part1,lem_hd_lower_part2}, for any $\ell\leq (1-\vep)\log_\nu n$,
\begin{align}\label{eq_compute_pathp}
\sum_{\nVer_1,\nVer_2=\bdt }^n&\prob\big(\gdist{\PAndmdel}(\nVer_1, \nVer_2)= \ell\big)\nn\\
    &\leq \sum_{\nVer_1,\nVer_2=\bdt }^n\Big(\sum_{\vpie\in \ESA_{\nVer_1,\nVer_2}^{\sss e,\ell}} \PP\bc{\vpie\subseteq \PAndmdel}+\sum_{i=1}^2\prob\big(\gdist{\PAndmdel}(\nVer_i, [ \bdt -1])\leq \ell/2\big)\Big) \nn\\
    &\leq \Theta^{1+\frac{\log^2 n}{\bdt }}\bc{n\nu^\ell+n^{3/2}\bdt ^{1/2}\nu^{\ell/2}}=\Theta\bc{n\nu^\ell+n^{3/2}\bdt ^{1/2}\nu^{\ell/2}},
\end{align}
where the last equation holds because we can write $\Theta^{1+\frac{\log^2 n}{\bdt }}=\Theta^{1+\frac{\log^2 n}{\lceil\log^3 n\rceil}}=\Theta$ by \Cref{def-Theta}.

On the other hand, by the law of total probability,
\begin{align}\label{eq_compute_pathp-2}
  \prob\big(\gdist{\PAndmdel}(\Ver_1, \Ver_2)\leq k\big)
  \leq&\sum_{i=1}^2\PP\big(\Ver_i\leq \bdt -1\big)+n^{-2}\sum_{\ell=0}^k\sum_{\nVer_1,\nVer_2=\bdt }^n\prob\big(\gdist{\PAndmdel}(\nVer_1, \nVer_2)= \ell\big).
\end{align}
Combining \eqref{eq_compute_pathp-2} with \eqref{eq_compute_pathp} yields that
\begin{flalign}\label{eq-dist-sum-nu}
   \prob\big(\gdist{\PAndmdel}(\Ver_1, \Ver_2)\leq k\big)
   \leq 2\bdt /n+n^{-2}\Theta\Big(n\sum_{\ell=0}^k\nu^\ell+n^{3/2}\bdt ^{1/2}\sum_{\ell=0}^k\nu^{\ell/2}\Big)\nn\\
   = 2\bdt /n+\Theta\Big(n^{-1}\frac{\nu^{k+1}-1}{\nu^{k+1}-\nu^k}\nu^k+n^{-1/2}\bdt ^{1/2}\frac{\nu^{(k+1)/2}-1}{\nu^{(k+1)/2}-\nu^{k/2}}\nu^{k/2}\Big).
\end{flalign}
Recall from \eqref{spectral_radius} that $\nu>1$. Then, by the definition of $\Theta$ in \Cref{def-Theta}, we can safely absorb $(\nu^{k+1}-1)/(\nu^{k+1}-\nu^k)$ and $(\nu^{(k+1)/2}-1)/(\nu^{(k+1)/2}-\nu^{k/2})$ into  $\Theta$. Then, \eqref{eq-dist-sum-nu} can be simplified as
\begin{align}\label{eq-dist-sum-nu2}
   \prob\big(\gdist{\PAndmdel}(\Ver_1, \Ver_2)\leq k\big)
   \leq 2\bdt /n+\Theta\bc{n^{-1}\nu^k+n^{-1/2}\bdt ^{1/2}\nu^{k/2}}.
\end{align}
Since $\bdt =\lceil\log^3 n\rceil$, with $k\leq (1-\vep)\log_\nu n$, we conclude from \eqref{eq-dist-sum-nu2} that
\begin{align*}
   \prob\big(\gdist{\PAndmdel}(\Ver_1, \Ver_2)\leq k\big)
\leq2\bdt /n+\Theta\bc{n^{-\vep}+n^{-\vep/2}\bdt ^{1/2}}=o(1),
\end{align*}
as desired.
\end{proof}
In the remainder of this section, using the Pólya urn representation of $\PAndmdel$ in \Cref{sec_polya_pam}, we prove  \Cref{lem_hd_lower_part1} in \Cref{sec_trun_lowerbound}. By a similar argument, we prove \Cref{lem_hd_lower_part2} in \Cref{sec-con-small-age-first-moment}.

\subsection{Proof of \Cref{lem_hd_lower_part1}: The first moment under the truncation on $[\bdt ,n]$}\label{sec_trun_lowerbound}
In this section, we prove \Cref{lem_hd_lower_part1}. We first calculate and simplify the probability that an edge-labeled self-avoiding path occurs in $\PAndmdel$. Using this path probability, we then transform the summation in \Cref{lem_hd_lower_part1} into, first, an integral, and then an inner product. Finally, we use the operator norm of the offspring operator $\bfT_\kappa$ to bound this inner product from above.

\begin{proof}[Proof of \Cref{lem_hd_lower_part1}]
We first note from \eqref{eq-alpha-beta-pand}, \eqref{eq_hd_lpath_probability} and \eqref{expctation_function_beta} that
\begin{align}
 \PP\bc{\vpie\subseteq \PAndmdel}= \prod_{s=2}^n\frac{(\alpha+p_s-1)_{p_s}(\beta_s+q_s-1)_{q_s}}{(\alpha+\beta_s+p_s+q_s-1)_{p_s+q_s}},\label{eq_hd_lower_part0}
\end{align}
where $\alpha=m+\delta$ and $\beta_s=(2s-3)m+(s-1)\delta$.


To bound the rhs of \eqref{eq_hd_lower_part0}, we deal with the product $\prod_{s=2}^n(\alpha+p_s-1)_{p_s}$ first. Recall from \eqref{def_hd_pq} that
\eqan{p_s=p_s^{\sss \vec\pi}=\sum_{e\in \vpie}\indic{s=\ushort{e}}\quad \text{and}\quad
q_s=q_s^{\sss \vec\pi}=\sum_{e\in \vpie}\indic{s\in (\ushort{e},\bar{e})}.\nn
}
Hence, $p_s\in\cbc{0,1,2}$, $p_s=0$ for $s\notin\vec\pi$, and
 \eqan{(\alpha+p_s-1)_{p_s}=\begin{cases}
     1,\quad&p_s=0;\\
     m+\delta,\quad&p_s=1;\\
     (m+\delta)(m+\delta+1),\quad&p_s=2.
 \end{cases}}
Furthermore, $\sum_{s\in \vec\pi}p_s=\sum_{s\in \vec\pi}\sum_{e\in \vpie}\indic{s=\ushort{e}}=k$. Consequently, we compute $\prod_{s=2}^n(\alpha+p_s-1)_{p_s}$ as 
\begin{align}\label{eq_path-first-prod}
    \prod_{s=2}^n(\alpha+p_s-1)_{p_s}=(m+\delta)^{k-\sum_{s\in\vec\pi}\indic{p_s=2}}(m+\delta+1)^{\sum_{s\in\vec\pi}\indic{p_s=2}}.
\end{align}

In the study of PAMs, 
we often  label  some vertices according to a direction, as we did in constructing the Pólya point tree $\PPT$ in \Cref{sec_offspring_operator}. Given a $k$-step path $\vec\pi$ and $i\in[k]$, we set
\begin{align}\label{def_label_direction}
    q(x)=\begin{cases}
        \rO,\quad&\text{if }x\leq 0;\\
        \rY,\quad&\text{if }x>0;
    \end{cases}\quad\text{and}\quad \lb^\pi_i=q(\pi_i-\pi_{i-1}),
\end{align}
while vertex $\pi_0$ has no label\zhu{, since there is no $\pi_{-1}$}. 
 \zhu{However, in the \ch{summations or products} related to $\lb^\pi_i$, we do not want to write the term for $\pi_0$ separately, \ch{see e.g., \eqref{eq_hd_lower_trans_sum_int_j}}. To simplify the equations,} we sometimes use $\lb^\pi_0$ as a variable, with range in $\cbc{\rO,\rY}$. With the notation in \eqref{def_label_direction}, we let 
\begin{align}
    N_{\srO\srY}=\sum_{i\in[k-1]}\indic{(\lb^\pi_i,\lb^\pi_{i+1})=(\srO,\srY)}=\sum_{s\in\vec\pi}\indic{p_s=2}.\label{def_NOY}
\end{align}
Then, the combination of  \eqref{eq_path-first-prod} and \eqref{def_NOY} yields that
\begin{align}\label{eq_hd_prod_ap-1}
    \prod_{s=2}^n (\alpha+p_s-1)_{p_s}&=(m+\delta)^{k-N_{\srO\srY}}(m+1+\delta)^{N_{\srO\srY}}.
\end{align}

On the other hand, we \ch{can decompose} the remaining factors in \eqref{eq_hd_lower_part0} \ch{as}
\eqan{&\prod_{s=2}^n \frac{(\beta_s+q_s-1)_{q_s}}{(\alpha+\beta_s+p_s+q_s-1)_{p_s+q_s}}\nn\\
&\qquad=\prod_{s=2}^n \frac{1}{(\alpha+\beta_s+p_s+q_s-1)_{p_s}}\prod_{i=0}^{q_s-1}\bc{1-\frac{\alpha}{\alpha+\beta_s+i}}.\label{eq_hd_cal_1}}
Recall from \eqref{def_hd_pq_gen} that
\eqan{\psE =\sum_{e\in E}\indic{s=\ushort{e}} \quad\text{and}\quad
\qsE =\sum_{e\in E}\indic{s\in (\ushort{e},\bar{e})} \quad\text{for $s\in [n]$}.\nn
}
For a set $S$, $\abs{S}$ denotes the cardinality of $S$. For an edge set $E$, we denote the set of vertices appearing in $E$ by $V_E=\cup_{e\in E}\cbc{\ushort{e},\bar{e}}$. 
{For future use, and} to bound the rhs of \eqref{eq_hd_cal_1}, we extend the case from {a self-avoiding path $\vpie$ to good edge sets $E$, that are defined as follows:}
\begin{definition}[Good edge set]\label{ass_e}
Given positive integers $n$ and $\bdt$, we say that an edge set $E$ is a {\em good edge set}, if 
\begin{itemize}
    \item[$\rhd$] each vertex appears at most $4$ times in $E$, and thus $\psE\leq 4$ and $\qsE\leq 4s$;
    \item[$\rhd$] no vertex with age less than $\bdt$ is present in $V_E$, and thus $\qsE=0$ for $s<\bdt $;
    \item[$\rhd$] $\abs{E}\leq 4\log_\nu n$, and thus $\qsE\leq \abs{E}\leq 4\log_\nu n$.
\end{itemize}    
\end{definition}
We will next address the following general case of  \eqref{eq_hd_cal_1} with a good edge set $E$:
\eqan{\prod_{s=2}^n \frac{1}{(\alpha+\beta_s+\psE+\qsE-1)_{\psE}}\prod_{i=0}^{\qsE-1}\bc{1-\frac{\alpha}{\alpha+\beta_s+i}},\label{eq_hd_cal_1_gen}}
where $p_s=p_s^{\sss \vec\pi}$ and $q_s=q_s^{\sss \vec\pi}$ are replaced by $\psE $ and $\qsE $.
This generalization is utilized in the subsequent variance calculation, and the proofs for the path $\vpie$ and the good edge set $E$ are identical. 

Recall that $\alpha=m+\delta$ and $\beta_s=(2s-3)m+(s-1)\delta$. For the first product in \eqref{eq_hd_cal_1_gen}, since $(x)_0=1$ and $\psE=\sum_{e\in E}\indic{s=\ushort{e}} =0$ for $s\notin V_E$, 
\begin{flalign}
    \label{eq_hd_fi_prod-before}
\prod_{s=2}^n \frac{1}{(\alpha+\beta_s+\psE+\qsE-1)_{\psE}}=\prod_{s\in V_E}\frac{((2m+\delta)s)^{\psE}}{(\alpha+\beta_s+\psE+\qsE-1)_{\psE}}((2m+\delta)s)^{-\psE}\nn\\
=(2m+\delta)^{-\abs{E}}\prod_{s\in V_E} s^{-\sum_{e\in E}\indic{s=\ushort{e}}} \prod_{i=0}^{\psE-1}\bc{1+\frac{\qsE+i-2m}{(2m+\delta)s}}^{-1},
\end{flalign}
where we use that $\sum_{s\in V_E}\psE=\abs{E}$ in the last equation.

Since $E$ is a good edge set, \ch{we have that} $\psE\leq 4$ and $\qsE\leq \abs{E}\leq 4\log_\nu n$. Then, for $s\geq \bdt$,
\begin{align}\label{eq-before-meanvalue-ln(1+x)}
    \frac{\qsE+i-2m}{(2m+\delta)s}= O(1) \frac{\log n}{\bdt}.
\end{align}
Hence, applying the mean value theorem on $\log (1+x)$, 
\begin{align}\label{eq-meanvalue-ln(1+x)}
    \sum_{s\in V_E}\sum_{i=0}^{\psE-1}\log\bc{1+\frac{\qsE+i-2m}{(2m+\delta)s}}=O(1)\sum_{s\in V_E}\sum_{i=0}^{\psE-1}\frac{\qsE+i-2m}{(2m+\delta)s}=O(1)\frac{\log^2 n}{\bdt },
\end{align}
where the three $O(1)$s in \eqref{eq-before-meanvalue-ln(1+x)} and \eqref{eq-meanvalue-ln(1+x)} have uniform upper and lower bounds for all $n\geq 2$,\footnote{Recall that we allow $O(1)$ to be negative.} and these bounds depend only on $m$ and $\delta$. Hence, $\e^{O(1)}=\Theta$, where we recall from \Cref{def-Theta} that $\Theta$ is a number with positive uniform upper and lower bounds throughout the paper. 

On the other hand, 
\begin{align*}
    \prod_{s\in V_E} s^{-\sum_{e\in E}\indic{s=\ushort{e}}}=\prod_{e\in E}\ushort{e}^{-1}.
\end{align*}
Hence, we conclude from \eqref{eq_hd_fi_prod-before} and \eqref{eq-meanvalue-ln(1+x)} that
\eqan{\label{eq_hd_fi_prod}
&\prod_{s=2}^n \frac{1}{(\alpha+\beta_s+\psE+\qsE-1)_{\psE}}\\
&\qquad=(2m+\delta)^{-\abs{E}}\prod_{s\in V_E} s^{-\sum_{e\in E}\indic{s=\ushort{e}}} \prod_{i=0}^{\psE-1}\bc{1+\frac{\qsE+i-2m}{(2m+\delta)s}}^{-1}\nn\\
&\qquad=(2m+\delta)^{-\abs{E}}\Theta^{\frac{\log^2 n}{\bdt }}\prod_{e\in E}\ushort{e}^{-1}.\nn}


For the second product in \eqref{eq_hd_cal_1_gen}, as we did in \eqref{eq-meanvalue-ln(1+x)}, by applying Taylor's theorem with a second-order mean-value form of the remainder on $\log(1+x)$, 
\eqan{\label{eq_hd_si_prod_1}
\prod_{s=2}^n\prod_{i=0}^{\qsE-1}\bc{1-\frac{\alpha}{\alpha+\beta_s+i}}&=\exp\bigg(\sum_{s=2}^n\sum_{i=0}^{\qsE-1}\log\bc{1-\frac{\alpha}{\alpha+\beta_s+i}}\bigg)\nn\\
&=\exp\bigg(-\sum_{s=2}^n \sum_{i=0}^{\qsE-1}\frac{\alpha}{\alpha+\beta_s+i}+O(1)\sum_{s=2}^n  \frac{\qsE}{s^2}\bigg)\nn\\
&=\exp\bigg(-\sum_{s=2}^n \sum_{i=0}^{\qsE-1}\frac{\alpha}{\alpha+\beta_s+i}\bigg)\Theta^{\sum_{s=2}^n  \frac{\qsE}{s^2}}.
}
Applying the mean value theorem on $1/(1+x)$ to the double summation in \eqref{eq_hd_si_prod_1},
\begin{align}\label{eq_hd_si_prod_2}
    \sum_{s=2}^n \sum_{i=0}^{\qsE-1}\frac{\alpha}{\alpha+\beta_s+i}&=\sum_{s=2}^n \sum_{i=0}^{\qsE-1}\frac{\alpha}{(2m+\delta)s}\frac{1}{1+(i-2m)/((2m+\delta)s)}\nn\\
    &=\sum_{s=2}^n \sum_{i=0}^{\qsE-1}\frac{\alpha}{(2m+\delta)s}+O(1)\sum_{s=2}^n \sum_{i=0}^{\qsE-1}\frac{\qsE}{s^2}\nn\\
    &=\frac{\alpha}{2m+\delta}\sum_{s=2}^n \frac{\qsE}{s}+O(1)\sum_{s=2}^n \bc{\frac{\qsE}{s}}^2,
\end{align}
where again the $O(1)$ factors satisfy $\e^{O(1)}=\Theta$.
Further, for any positive integer $2\leq a\leq a'$,
\begin{align}\label{eq_ln-int-s^-1}
    \sum_{s=a+1}^{a'} \frac{1}{s}\leq\log\frac{a'}{a}=\sum_{s=a}^{a'-1}\int_{s}^{s+1} \frac{1}{t}\dint t \leq\sum_{s=a}^{a'-1} \frac{1}{s}.
\end{align}
Then, since $\qsE\leq \abs{E}\leq 4\log_\nu n$, $s\geq \bdt $ for $s\in V_E$ and $\qsE=0$ for $s< \bdt $, \eqref{eq_ln-int-s^-1} yields that
\begin{align}\label{eq_hd_si_prod_3}
    \sum_{s=2}^n \frac{\qsE}{s}=\sum_{e\in E}\sum_{s\in (\ushort{e},\bar{e})}\frac{1}{s}=\sum_{e\in E}\log\bc{\frac{\bar{e}}{\ushort{e}}}+O(1)\frac{\log n}{\bdt },
\end{align}
and
\begin{align}\label{eq_hd_si_prod_4}
    \sum_{s=2}^n \frac{\qsE}{s^2}\leq \sum_{s=2}^n \bc{\frac{\qsE}{s}}^2\leq \sum_{s= \bdt }^\infty \frac{16\log_\nu^2 n}{s(s-1)}= \frac{16\log_\nu^2 n}{\bdt -1}=O(1)\frac{\log^2 n}{\bdt},
\end{align}
and the factor $O(1)$ in \eqref{eq_hd_si_prod_3} satisfies $\e^{O(1)}=\Theta$. We thus conclude from the combination of (\ref{eq_hd_si_prod_1})-(\ref{eq_hd_si_prod_4}) that
\begin{align}\label{eq_hd_si_prod}
    \prod_{s=2}^n\prod_{i=0}^{\qsE-1}\bc{1-\frac{\alpha}{\alpha+\beta_s+i}}=\Theta^{\frac{\log^2 n}{\bdt }}\prod_{e\in E}\bc{\frac{\ushort{e}}{\bar{e}}}^\chi,
\end{align}
where 
\begin{align*}
    \chi=\frac{\alpha}{2m+\delta}=\frac{m+\delta}{2m+\delta}.
\end{align*}
Consequently, by \eqref{eq_hd_fi_prod} and \eqref{eq_hd_si_prod},
\begin{align}\label{eq_hd_cal_1_E}
    \prod_{s=2}^n \frac{(\beta_s+\qsE-1)_{\qsE}}{(\alpha+\beta_s+\psE+\qsE-1)_{\psE+\qsE}}=\Theta^{\frac{\log^2 n}{\bdt }}(2m+\delta)^{-\abs{E}}\prod_{e\in E}\frac{1}{\ushort{e}^{1-\chi}\bar{e}^{\chi}},
\end{align}
where we recall \Cref{def-Theta} and we use $\Theta\cdot\Theta=\Theta$.
Specifically, for $E=\vpie\in \ESA_{\nVer_1,\nVer_2}^{\sss e,k}$ with $\nVer_1,\nVer_2\geq \bdt$,
\begin{align}\label{eq_hd_cal_1_pi}
    &\prod_{s=2}^n \frac{(\beta_s+q_s-1)_{q_s}}{(\alpha+\beta_s+p_s+q_s-1)_{p_s+q_s}}\nn\\
    &\qquad=\Theta^{\frac{\log^2 n}{\bdt }}(2m+\delta)^{-k}\prod_{i\in[k]}\frac{1}{(\pi_{i-1}\wedge\pi_i)^{1-\chi}(\pi_{i-1}\vee\pi_i)^{\chi}}.
\end{align}
Hence, by \eqref{eq_hd_lower_part0}, \eqref{eq_hd_prod_ap-1} and  \eqref{eq_hd_cal_1_pi},
\begin{align}\label{eq_hd_prob_one_path}
    \PP&\bc{\vpie\subseteq \PAndmdel}=\prod_{s=2}^n\frac{(\alpha+p_s-1)_{p_s}(\beta_s+q_s-1)_{q_s}}{(\alpha+\beta_s+p_s+q_s-1)_{p_s+q_s}}\\
    &= \Theta^{\frac{\log^2 n}{\bdt }}\bc{\frac{m+\delta}{2m+\delta}}^{k-N_{\srO\srY}}\bc{\frac{m+1+\delta}{2m+\delta}}^{N_{\srO\srY}}\prod_{i\in[k]}\frac{1}{(\pi_{i-1}\wedge\pi_i)^{1-\chi}(\pi_{i-1}\vee\pi_i)^\chi }.\nn
\end{align}

Let $\SA_{u,v}^{\sss \ell}=\SA_{u,v}^{\sss \ell,(b)}$ be the set of all possible $\ell$-step self-avoiding paths $\vec\pi$ such that $\pi_0=u$, $\pi_\ell=v$ and  $\pi_i\in[\bdt ,n]$ for all $i\in[\ell-1]$. Compared to edge-labeled self-avoiding paths in $\ESA_{\nVer_1,\nVer_2}^{\sss e,k}$, the elements in $\SA_{\nVer_1,\nVer_2}^{\sss k}$  are paths without edge labels.
We next consider the number of possible $k$-step edge-labeled self-avoiding paths $\vpie\in\ESA_{\nVer_1,\nVer_2}^{\sss e,k}$, given  $\vec\pi\in \SA_{\nVer_1,\nVer_2}^{\sss k}$, and we fix the edge-labels in $\vec\pi$ one by one. 

For each $0\leq \ell\leq k-2$, since $\vec\pi$ is self-avoiding, when we have fixed the edge labels of edges in the subpath $(\pi_0,\pi_1,\ldots,\pi_\ell)$, the number of choices for the labeled edge between $\pi_\ell$ and $\pi_{\ell+1}$ is equal to $m-1$ when $\lb^\pi_\ell=\rY$ and $\lb^\pi_{\ell+1}=\rO$, and $m$ otherwise. Indeed, when $\lb^\pi_\ell=\rY$ and $\lb^\pi_{\ell+1}=\rO$, the edges between $\pi_\ell$ and $\pi_{\ell+1}$, and between $\pi_\ell$ and $\pi_{\ell-1}$ are both out-edges of $\pi_\ell$. Since the out-degree of $\pi_\ell$ is equal to $m$ and we have used one for the edge between $\pi_\ell$ and $\pi_{\ell-1}$, the number of choices between $\pi_\ell$ and $\pi_{\ell+1}$ is equal to $m-1$. Consequently,  for 
\begin{align}
    N_{\srY\srO}=\sum_{i\in[k-1]}\indic{(\lb^\pi_i,\lb^\pi_{i+1})=(\rY,\rO)},\label{def_NYO}
\end{align}
the number of possible edge-labeled self-avoiding paths $\vpie$ given  $\vec\pi$ is equal to $m^{k-N_{\srY\srO}}(m-1)^{N_{\srY\srO}}$. Therefore, by \eqref{eq_hd_prob_one_path},
\begin{align}\label{eq_sum_no_con_part1}
\sum_{\vpie\in \ESA_{\nVer_1,\nVer_2}^{\sss e,k}}\PP\bc{\vpie\subseteq \PAndmdel}&=\Theta^{\frac{\log^2 n}{\bdt }}\sum_{\vec\pi\in \SA_{\nVer_1,\nVer_2}^{\sss k}}m^{k-N_{\srY\srO}}(m-1)^{N_{\srY\srO}}\bc{\frac{m+\delta}{2m+\delta}}^{k-N_{\srO\srY}}\nn\\
    &\quad\times\bc{\frac{m+1+\delta}{2m+\delta}}^{N_{\srO\srY}}\prod_{i\in[k]}\frac{1}{(\pi_{i-1}\wedge\pi_i)^{1-\chi}(\pi_{i-1}\vee\pi_i)^\chi }\\
    &= \Theta^{\frac{\log^2 n}{\bdt }}n^{-k}\sum_{\vec\pi\in \SA_{\nVer_1,\nVer_2}^{\sss k}}m^{k-N_{\srY\srO}}(m-1)^{N_{\srY\srO}}\bc{\frac{m+\delta}{2m+\delta}}^{k-N_{\srO\srY}}\nn\\
    &\quad\times\bc{\frac{m+1+\delta}{2m+\delta}}^{N_{\srO\srY}}\prod_{i\in[k]}\frac{1}{(\frac{\pi_{i-1}\wedge\pi_i}{n})^{1-\chi}(\frac{\pi_{i-1}\vee\pi_i}{n})^\chi }\nn.
\end{align}

Recall the definitions of $\kappa$ in \eqref{kappa-PAM-fin-def-form-a-rep} and $c_{st}$ in \eqref{cst-def-PAM-rep}. Then, for $2\leq i\leq k$, also using \eqref{def_label_direction},
\begin{align}\label{eq_sum_no_con_part2}
    \frac{1}{(\frac{\pi_{i-1}\wedge\pi_i}{n})^{1-\chi}(\frac{\pi_{i-1}\vee\pi_i}{n})^\chi }=\frac{\kappa\big((\pi_{i-1}/n,\lb^\pi_{i-1}),(\pi_{i}/n,\lb^\pi_i)\big)}{c_{\lb^\pi_{i-1}\lb^\pi_{i}}},
\end{align}
while, by \eqref{def_NOY} and \eqref{def_NYO},
\begin{align}\label{eq_sum_no_con_part3}
    \prod_{i=2}^k c_{\lb^\pi_{i-1},\lb^\pi_{i}}&=c_{\srY\srO}^{N_{\srY\srO}}c_{\srO\srY}^{N_{\srO\srY}}c_{\srO\srO}^{k-1-N_{\srY\srO}-N_{\srO\srY}}\\
    &=c_{\srO\srO}^{-1}m^{k-N_{\srY\srO}}(m-1)^{N_{\srY\srO}}\bc{\frac{m+\delta}{2m+\delta}}^{k-N_{\srO\srY}}\bc{\frac{m+1+\delta}{2m+\delta}}^{N_{\srO\srY}},\nn
\end{align}
where we use $c_{\srO\srO}=c_{\srY\srY}$ in the first equation. 
Hence, with $\pi_0=\nVer_1$ and $\pi_k=\nVer_2$, by \eqref{eq_sum_no_con_part1}-\eqref{eq_sum_no_con_part3}, 
\begin{align}\label{eq_hd_lower_form_prod_kappa-old}
    &\sum_{\vpie\in \ESA_{\nVer_1,\nVer_2}^{\sss e,k}} \PP\bc{\vpie\subseteq \PAndmdel}= \sum_{\vpie\in \ESA_{\nVer_1,\nVer_2}^{\sss e,k}}\prod_{s=2}^n\frac{(\alpha+p_s-1)_{p_s}(\beta_s+q_s-1)_{q_s}}{(\alpha+\beta_s+p_s+q_s-1)_{p_s+q_s}}\\
    &= \Theta^{1+\frac{\log^2 n}{\bdt }} n^{-k}\sum_{\vec\pi\in \SA_{\nVer_1,\nVer_2}^{\sss k}}\frac{1}{(\frac{\pi_{0}\wedge\pi_1}{n})^{1-\chi}(\frac{\pi_{0}\vee\pi_1}{n})^\chi }\prod_{i=2}^k\kappa\big((\pi_{i-1}/n,\lb^\pi_{i-1}),(\pi_{i}/n,\lb^\pi_i) \big).\nn
\end{align}
Furthermore, by \eqref{kappa-PAM-fin-def-form-a-rep},
\begin{align}\label{eq_hd_lower_form_prod_kappa-old2}
\sum_{\lb^\pi_0\in\cbc{\srO,\srY}}\kappa\big((\pi_0/n,\lb^\pi_0),(\pi_1/n,\lb^\pi_1)\big)=\frac{\Theta}{(\frac{\pi_{0}\wedge\pi_1}{n})^{1-\chi}(\frac{\pi_{0}\vee\pi_1}{n})^\chi }.
\end{align}
Let $\NA_{u,v}^{\sss \ell}=\NA_{u,v}^{\sss \ell,(b)}$ be the set of all possible $\ell$-step neighbor-avoiding paths $\vec\pi$ such that $\pi_0=u$, $\pi_\ell=v$,  $\pi_i\in[\bdt ,n]$ for all $i\in[\ell-1]$ and $\pi_i\neq \pi_{i-1}$ for all $i\in[\ell]$. Since $\SA_{\nVer_1,\nVer_2}^{\sss k}\subseteq \NA_{\nVer_1,\nVer_2}^{\sss k}$, we conclude from \eqref{eq_hd_lower_form_prod_kappa-old} and \eqref{eq_hd_lower_form_prod_kappa-old2} that
\begin{align}\label{eq_hd_lower_form_prod_kappa}
    \sum_{\vpie\in \ESA_{\nVer_1,\nVer_2}^{\sss e,k}} &\PP\bc{\vpie\subseteq \PAndmdel}\\
    &\leq \Theta^{1+\frac{\log^2 n}{\bdt }}n^{-k}\sum_{\vec\pi\in \NA_{\nVer_1,\nVer_2}^{\sss k}}\sum_{\lb^\pi_0\in\cbc{\srO,\srY}}\prod_{i=1}^k\kappa\big((\pi_{i-1}/n,\lb^\pi_{i-1}),(\pi_{i}/n,\lb^\pi_i)\big)\nn.
\end{align}

For $x_0,x_1,\ldots,x_k\in (0,1]$, let $\rmq_i=q(x_i-x_{i-1})$ for $i\in [k]$ (recall the definition of $q(\cdot)$ in  \eqref{def_label_direction}).
Given $\vec\pi\in  \NA_{\nVer_1,\nVer_2}^{\sss k}$, when $x_j\in((\pi_j-1)/n,\pi_j/n]$ $(0\leq j\leq k)$, $\rmq_i=\lb^\pi_i$ for all $i\in [k]$. \zhu{Since $(x\vee y)^{-\chi}(x\wedge y)^{-(1-\chi)}$ is non-increasing for all positive $x,y$, 
it follows from \eqref{kappa-PAM-fin-def-form-a-rep} that, for any $x_j\in((\pi_j-1)/n,\pi_j/n]$,}
\begin{align*}
    \kappa\big((\pi_{i-1}/n,\lb^\pi_{i-1}),(\pi_{i}/n,\lb^\pi_i)\big)\leq \kappa\big((x_{i-1},\rmq_{i-1}),(x_i,\rmq_i)\big).
\end{align*}
Consequently, for $i\in [2,k]$,
\begin{align}\label{decrease-kappa}
    \kappa\big((\pi_{i-1}/n,\lb^\pi_{i-1}),(\pi_{i}/n,\lb^\pi_i)\big)\leq &n\int_{(\pi_i-1)/n}^{\pi_i/n}\kappa\big((x_{i-1},\rmq_{i-1}),(x_i,\rmq_i)\big)\dint x_i.
\end{align}
If we further define $\rmq_0=\lb^\pi_0$, then
\begin{align}\label{decrease-kappa-0}
        \kappa\big((\pi_0/n,\lb^\pi_0),(\pi_1/n,\lb^\pi_1)\big)\leq &n^2\int_{(\pi_0-1)/n}^{\pi_0/n}\int_{(\pi_1-1)/n}^{\pi_1/n}\kappa\big((x_0,\rmq_0),(x_1,\rmq_1)\big)\dint x_0\dint x_1.
\end{align}
Hence, for all $\nVer_1,\nVer_2\in[n]$, we can bound the summation of the kernel product in \eqref{eq_hd_lower_form_prod_kappa} from above by a multiple integral as
 \begin{align}
&\sum_{\vec\pi\in \NA_{\nVer_1,\nVer_2}^{\sss k}}\sum_{\lb^\pi_0\in\cbc{\srO,\srY}}\prod_{i=1}^k\kappa\big((\pi_{i-1}/n,\lb^\pi_{i-1}),(\pi_{i}/n,\lb^\pi_i)\big)\label{eq_hd_lower_trans_sum_int}\\
    \leq n^{k+1}&\sum_{\rmq_0\in\cbc{\srO,\srY}}\int_{(\nVer_1-1)/n}^{\nVer_1/n}\int_0^1\cdots \int_0^1\int_{(\nVer_2-1)/n}^{\nVer_2/n} \prod_{i=1}^k\kappa\big((x_{i-1},\rmq_{i-1}),(\zhu{x_i},\rmq_i)\big)\dint x_0\cdots \dint x_k.\nn
\end{align}
Combining \eqref{eq_hd_lower_form_prod_kappa} and \eqref{eq_hd_lower_trans_sum_int} gives that
\begin{align}\label{eq-1_add-round_1}
&\sum_{\nVer_1,\nVer_2=\bdt }^n\sum_{\vpie\in \ESA_{\nVer_1,\nVer_2}^{\sss e,k}}\PP\bc{\vpie\subseteq \PAndmdel}\\
&\qquad\leq \Theta^{1+\frac{\log^2 n}{\bdt }}n\sum_{\rmq_0\in\cbc{\srO,\srY}}\int_0^1\cdots \int_0^1 \prod_{i=1}^k\kappa\big((x_{i-1},\rmq_{i-1}),(x_{i},\rmq_i)\big)\dint x_0\cdots \dint x_k.\nn
\end{align}
On the other hand, we note from \eqref{kappa-PAM-fin-def-form-a-rep} and \eqref{T-kappa-def} that, for $q(x)$ defined as in \eqref{def_label_direction},
\begin{align}\label{eq-new-Tkappa}
    (\bfT_{\kappa}^\ell f)(x_0,s_0)=&\int_{0}^1 \cdots \int_0^1 f(x_\ell,q(x_\ell-x_{\ell-1}))\kappa\big((x_0,s_0),(x_1,q(x_1-x_0))\big)\\
    &\times\prod_{i=2}^\ell\kappa\big((x_{i-1},q(x_{i-1}-x_{i-2})),(x_i,q(x_i-x_{i-1}))\big)\dint x_1\cdots \zhu{\dint} x_\ell\nn,
\end{align}
Then, with $\langle\ ,\ \rangle$ the inner product on $[0,1]\times\cbc{\rO,\rY}$ and ${\bf 1}$ the function that is identical to $1$, \eqref{eq-new-Tkappa} yields that
\begin{align}\label{eq-new-Tkappa-2}
\sum_{\rmq_0\in\cbc{\srO,\srY}}\int_{0}^{1}\cdots \int_0^1 \prod_{i=1}^k\kappa\big((x_{i-1},\rmq_{i-1}),(x_{i},\rmq_i)\big)\dint x_0\cdots \dint x_k=\langle {\bf 1}, \bfT_\kappa^{k} {\bf 1}\rangle.
\end{align}
Hence, we conclude from the combination of  \Cref{pro_spectral_radius}, \eqref{eq-1_add-round_1} and \eqref{eq-new-Tkappa-2} that, for $k\leq (1-\vep)\log_\nu n$, 
\begin{align}\label{eq_hd_lower_part1-ded}
&\sum_{\nVer_1,\nVer_2=\bdt }^n\sum_{\vpie\in \ESA_{\nVer_1,\nVer_2}^{\sss e,k}}\PP\bc{\vpie\subseteq \PAndmdel}\\
    &\qquad\leq \Theta^{1+\frac{\log^2 n}{\bdt }}n\langle {\bf 1}, \bfT_\kappa^{k} {\bf 1}\rangle\leq \Theta^{1+\frac{\log^2 n}{\bdt }}n\norm{\bf 1}\norm{\bfT_\kappa}^k \norm{\bf 1}\leq \Theta^{1+\frac{\log^2 n}{\bdt }}n\nu^k,\nn
\end{align}
as desired.
\end{proof}

Recall that $\zeta$ is the leaf truncation parameter from \Cref{ass_Br}. Later in this paper, during the proof, we will rely on a crude path upper bound for $\bdt =\lceil \zeta n\rceil$ as follows:
\begin{lemma}[Crude upper bound on the  kernel products summation]\label{lem-gen-sum-product-kappa}
    Let $\bdt =\lceil \zeta n\rceil$. For any integers $n\geq \zeta^{-2}$, $\nVer_1,\nVer_2\in [\zeta n,n]$, any $\lb_0^\pi\in\cbc{\rO,\rY}$, and any $\ell\geq 1$, 
    \begin{align}\label{eq-gen-sum-product-kappa}
         \sum_{\vec\pi\in\NA_{\nVer_1,\nVer_2}^{\sss \ell}}\prod_{i=1}^{\ell}\kappa\big((\pi_{i-1}/n,\lb^\pi_{i-1}),(\pi_{i}/n,\lb^\pi_i)\big)\leq \Theta\zeta^{-2}\nu^{\ell}n^{\ell-1}.
    \end{align}
\end{lemma}
The proof of \Cref{lem-gen-sum-product-kappa} is analogous to that of \Cref{lem_hd_lower_part1}, except that we do not sum over all $\nVer_1$ and $\nVer_2$ between $\bdt$ and $n$. Instead, we bound the first and last factors in the products by $\Theta \zeta^{-1}$ using \eqref{kappa-PAM-fin-def-form-a-rep}. \longversion{We defer its proof to \Cref{sec-path-upper-bound}} \shortversion{We refer to \cite[Appendix D]{HofZhu25long} for its proof}.
\subsection{Proof of \Cref{lem_hd_lower_part2}: Contribution of paths with small age vertices}\label{sec-con-small-age-first-moment}
In this section, we prove \Cref{lem_hd_lower_part2}, again using the path probability \eqref{eq_hd_lpath_probability-rep}. This proof is quite similar to the proof of \Cref{lem_hd_lower_part1}.

\begin{proof}[Proof of \Cref{lem_hd_lower_part2}]
Given an integer $a_1\in[\bdt,n]$,  we  note that the distance between $\nVer_1$ and $[\bdt -1]$ is no more than $k/2$ only if there exists an edge-labeled path $\vpie\in \ESA_{\nVer_1,j}^{\sss e,k'}$ such that $\vpie\subseteq\PAndmdel$ for some integer $k'\leq k/2$ and $j\in[\bdt-1]$. Hence,
\begin{align}\label{eq_hd_dist_decom}
    \prob\big(\gdist{\PAndmdel}(\nVer_1, [ \bdt -1])\leq k/2\big)\leq \sum_{k'\leq k/2}\sum_{j=1}^{\bdt -1}\sum_{\vpie\in \ESA_{\nVer_1,j}^{\sss e,k'}}\PP\bc{\vpie\subseteq \PAndmdel}.
\end{align}
Then, a straightforward idea would be to use \eqref{eq_hd_lower_form_prod_kappa} and \eqref{eq_hd_lower_trans_sum_int} to bound the rhs of \eqref{eq_hd_dist_decom} from above. However, for the edge-labeled paths considered in \Cref{sec_trun_lowerbound}, all vertices in these paths have ages no less than $\bdt$, while in \eqref{eq_hd_dist_decom}, the paths have one endpoint $j\in[\bdt -1]$.  Nevertheless, we note that $\eqref{eq_hd_lower_trans_sum_int}$ only relies on \eqref{decrease-kappa} and \eqref{decrease-kappa-0}, neither of which depends on the lower bound of the vertex ages in the path. Hence, by the same argument used in deriving \eqref{eq_hd_lower_trans_sum_int},
 \begin{align}
&\sum_{\vec\pi \in \ESA_{\nVer_1,j}^{\sss k'}}\sum_{\lb^\pi_0\in\cbc{\srO,\srY}}\prod_{i=1}^{k'}\kappa\big((\pi_{i-1}/n,\lb^\pi_{i-1}),(\pi_{i}/n,\lb^\pi_i)\big)\label{eq_hd_lower_trans_sum_int_j}\\
    \leq n^{k+1}&\sum_{\rmq_0\in\cbc{\srO,\srY}}\int_{(\nVer_1-1)/n}^{\nVer_1/n}\int_0^1\cdots \int_0^1\int_{(j-1)/n}^{j/n} \prod_{i=1}^{k'}\kappa\big((x_{i-1},\rmq_{i-1}),(\pi_{i}/n,\rmq_i)\big)\dint x_0\cdots \dint x_{k'}.\nn
\end{align}
To handle \eqref{eq_hd_lower_form_prod_kappa}, we trace back its proof in \Cref{sec_trun_lowerbound}. There are only three places, namely
\eqref{eq-meanvalue-ln(1+x)}, \eqref{eq_hd_si_prod_3} and \eqref{eq_hd_si_prod_4}, where we use the fact that  all $s\in\vec\pi$ are no less than $\bdt$. Let  $E$ be the edge set of a path $\vec\pi\in \ESA_{\nVer_1,j}^{\sss k'}$. Then, there is precisely one edge $e_0\in E$ such that $s_0=\ushort{e}_0$ is the only vertex in $V_E$ with age less than $\bdt$. By the definition of $\psE$ and $\qsE$ in \eqref{def_hd_pq_gen}, $p_{s_0}^{\sss E}=1$, $q_{s_0}^{\sss E}=0$ and $\qsE\leq 1$ for $s<\bdt$. Furthermore,  $\psE\leq 2$ and $\qsE\leq \min\cbc{2s,\abs{E}}\leq\min\cbc{2s,\log_{\nu} n}$.
Consequently, \ch{analogously} to \eqref{eq-meanvalue-ln(1+x)}, 
\begin{align*}
    &\sum_{s\in V_E}\sum_{i=0}^{\psE-1}\log\bc{1+\frac{\qsE+i-2m}{(2m+\delta)s}}\\
    &\qquad=\log\bc{1-\frac{2m}{(2m+\delta)s_0}}+O(1)\sum_{s\in V_E\backslash\cbc{s_0}}\sum_{i=0}^{\psE-1}\frac{\qsE+i-2m}{(2m+\delta)s}\\
    &\qquad=\log \Theta +O(1)\frac{\log^2 n}{\bdt }.
\end{align*}
\ch{Analogously} to \eqref{eq_hd_si_prod_3}, by \eqref{eq_ln-int-s^-1},
\begin{align*}
    \sum_{s=2}^n \frac{\qsE}{s}=\sum_{e\in E\backslash\cbc{e_0}}\sum_{s\in (\ushort{e},\bar{e})}\frac{1}{s}+\sum_{s\in (\ushort{e}_0,\bar{e_0})}\frac{1}{s}=\sum_{e\in E}\log\bc{\frac{\bar{e}}{\ushort{e}}}+O(1)\bc{\frac{\log n}{\bdt }+1}.
\end{align*}
Analogous to \eqref{eq_hd_si_prod_4}, since $\qsE\leq 1$ for $s\leq \bdt $,
\begin{align*}
    \sum_{s=2}^n \frac{\qsE}{s^2}\leq \sum_{s=2}^n \bc{\frac{\qsE}{s}}^2\leq \sum_{s= \bdt }^\infty \frac{16\log_\nu^2 n}{s^2}+\sum_{s=2}^{\bdt -1}\frac{1}{s^2}\leq O(1)\frac{\log^2 n}{\bdt}+1.
\end{align*}
Further, the three $O(1)$ factors in \zhu{the} above have uniform upper and lower bounds for all $n\geq 2$.
\smallskip

In summary, the occurrence of a unique edge $e_0\in E$ such that $s_0=\ushort{e}_0<b$ gives rise to an extra $\exp(\log \Theta+O(1))\Theta=\Theta$ factor on the rhs of \eqref{eq_hd_lower_form_prod_kappa}, which can be absorbed into the factor $\Theta^{1+\frac{\log^2 n}{\bdt }}$ (recall \Cref{def-Theta}). Then, combined with \eqref{eq_hd_lower_trans_sum_int_j}, we still arrive at 
\begin{align}\label{eq-path-prob-upper-a1-j}
\sum_{\vpie\in \ESA_{\nVer_1,j}^{\sss e,k'}} \PP\bc{\vpie\subseteq \PAndmdel}
    \leq &\Theta^{1+\frac{\log^2 n}{\bdt }}n\sum_{\rmq_0\in\cbc{\srO,\srY}}\int_{(\nVer_1-1)/n}^{\nVer_1/n}\int_0^1\cdots \int_0^1\int_{(j-1)/n}^{j/n}\nn\\
    &\times\prod_{i=1}^{k'}\kappa\big((x_{i-1},\rmq_{i-1}),(\pi_{i}/n,\rmq_i)\big)\dint x_0\cdots \dint x_{k'}.
\end{align}

Let ${\bf 1}_{[0,a]}(x,s)$ be the function on $[0,1]\times\cbc{\rO,\rY}$ that equals $1$ when $x\in[0,a]$ and $0$ when $x>a$. Then, under the $L^2$ norm we have $\lVert{\bf 1}_{[0,a]}\rVert=\sqrt{2a}$. Hence, by \eqref{eq-new-Tkappa}, \eqref{eq-path-prob-upper-a1-j} and \Cref{pro_spectral_radius},
\begin{align}\label{eq_hd_upperbound_dist_case2}
    \sum_{\nVer_1=\bdt }^n\sum_{j=1}^{\bdt -1}\sum_{\vpie\in \ESA_{\nVer_1,j}^{\sss e,k'}} \PP\bc{\vpie\subseteq \PAndmdel}&\leq \Theta^{1+\frac{\log^2 n}{\bdt }}n\langle {\bf 1},\bfT_{\kappa}^{k'}{\bf 1}_{[0,\bdt /n]}\rangle\\
    \leq \Theta^{1+\frac{\log^2 n}{\bdt }}n \norm{\bf 1}&\norm{\bfT_{\kappa}}^{k'}\norm{{\bf 1}_{[0,\bdt /n]}} = \Theta^{1+\frac{\log^2 n}{\bdt }} n^{1/2}\bdt ^{1/2} \nu^{k'}.\nn
\end{align}
We next sum \eqref{eq_hd_upperbound_dist_case2} over all $k'\leq k/2$. Then, by \eqref{eq_hd_dist_decom},
\begin{align}\label{eq_hd_final_dist_case2}
\sum_{\nVer_1=\bdt }^n\prob\big(\gdist{\PAndmdel}(\nVer_1, [ \bdt -1])\leq k/2\big)&\leq \Theta^{1+\frac{\log^2 n}{\bdt }} n^{1/2}\bdt ^{1/2} \sum_{k'=1}^{\lfloor k/2\rfloor}\nu^{k'}\\
&\leq \Theta^{1+\frac{\log^2 n}{\bdt }} n^{1/2}\bdt ^{1/2}\nu^{k/2},\nn
\end{align}
where in the last equation we use $\nu> 1$ from \Cref{pro_spectral_radius} and the same argument used in deriving \eqref{eq-dist-sum-nu2}.

Analogously, 
\begin{align}\label{eq_hd_final_dist_case2_2}
\sum_{\nVer_2=\bdt }^n\prob\big(\gdist{\PAndmdel}(\nVer_2, [ \bdt -1])\leq k/2\big)\leq \Theta^{1+\frac{\log^2 n}{\bdt }} n^{1/2}\bdt ^{1/2}\nu^{k/2},
\end{align}
and \Cref{lem_hd_lower_part2} follows directly from the combination of \eqref{eq_hd_final_dist_case2} and \eqref{eq_hd_final_dist_case2_2}, as desired.\end{proof}

\section{Spectral radius of the truncated offspring operator}\label{sec_converge_spectral_radius}
As we \zhu{have} discussed in \Cref{sec_offspring_operator}, the age restriction on the paths we have counted \zhu{suggests considering the offspring distribution} after removing all the vertices with age less than $\zeta n$. In fact, we \zhu{have} performed a similar \ch{restriction to} $[\bdt , n]$ in \Cref{sec_trun_lowerbound} with $\bdt=\ceil{\log^3 n}$ when proving the lower bound \eqref{conc-LB-dist-PAd}. However, in that case, only  an {\em upper bound} on the path existence probability was needed. Consequently, this truncation is ignored in \eqref{eq_hd_lower_trans_sum_int}.  When 
proving \Cref{prop-path-count-PAMc-LB_simple}, if we follow the approach in \Cref{sec_trun_lowerbound},   we  \ch{instead} need a {\em lower bound} on the path existence probability to show that $\Erw\brk{N_{n,r}(k)}$ tends to infinity. Then, two main problems arise:
\begin{itemize}
 \item[$\rhd$] The offspring operator $\bfT_\kappa$ should be {truncated to involve only vertices with rescaled ages at least $\zeta$; and}
 
\item[$\rhd$] \ch{we} can not easily bound the inner product $\langle {\bf 1}, \bfT_\kappa^k {\bf 1}\rangle$ from {\em below} by using the spectral norm for a general operator $\bfT_\kappa$,  as we did in \eqref{eq_hd_lower_part1-ded}. Instead, using the spectral radius could be more appropriate, since $r(\bfT_\kappa)=\lim_{k\to \infty}\lVert\bfT_\kappa^k\rVert^{1/k}$. {However, it is far from obvious that the spectral radius of the truncated operator is close to that of the untruncated one, since the operator is not compact, nor self-adjoint.}
\end{itemize}
Hence, to obtain \eqref{conc-UB-dist-PAd}, it is necessary to show that the inner products, or perhaps the spectral \zhu{\ch{radii}} of the offspring operators before and after truncation, are close.
Due to some technical reasons that will become clear in later sections, the kernel that we consider is the following  restriction of $\kappa$ on $[3{\zeta},1-{\zeta}]$:
 \eqn{
	\label{hatkappa-vepn-PA}
	\kappa^\circ_{\zeta}((x,s),(y,t))=\indic{x,y\in [3{\zeta},1-{\zeta}]} \kappa\big((x,s),(y,t)\big).
	}

In this section, we relate $\bfT_{\kappa^\circ_{\zeta}}$ to $\bfT_\kappa$. {The following is our main result}:
\begin{proposition}[Lower bound for the inner product]\label{pro_hd_spe_con}
There exists a $\zeta_0\in (0,1/10]$ such that, for any ${\zeta}\in(0,\zeta_0]$, there exist $\uhvep\in (0,\nu]$ and $c_{\zeta}\in (0,1)$ such that $\lim_{{\zeta}\searrow 0}\uhvep=\nu$ and
\begin{align*}
    \langle {\bf 1}, \bfT_{\kappa^\circ_{\zeta}}^{k}{\bf 1}\rangle\geq c_{\zeta}\uhvep^k\quad\text{for any $k\geq 0$.}
\end{align*}
\end{proposition}
\begin{rem}[The convergence of the spectral radius]\label{RE-CON-SPE-RADIUS}
\Cref{pro_hd_spe_con} yields that the spectral radius of $\bfT_{\kappa^\circ_{\zeta}}$ converges to $\nu$ as $\zeta\searrow 0$; \longversion{see \Cref{sec-app-conv-nu-kappa-hvep} for a proof.}\shortversion{we refer to \cite[Appendix B]{HofZhu25long} for a proof.} 
\end{rem}

    By \Cref{RE-CON-SPE-RADIUS}, one would expect that the proof of \Cref{pro_hd_spe_con} relies heavily on the application of extensive functional analysis techniques. Unfortunately, the operator $\bfT_\kappa$ is neither \ch{self-adjoint nor  compact}, making it difficult to control changes in its spectral radius under small perturbations through this approach. Indeed, our proof is purely probabilistic.  The remainder of this section focuses on this probabilistic proof and is organized as follows:  In \Cref{sec_isometry}, we use an isometry to obtain a large deviation-type interpretation for $\langle {\bf 1}, \bfT_{\tilde\kappa^\circ_{\zeta}}^k {\bf 1}\rangle$ through a random walk. In \Cref{sec_decoupling},
we simplify this interpretation by a measure transformation, and complete the proof of \eqref{pro_hd_spe_con} \zhu{by studying a related random walk}.

\subsection{Large deviations interpretation through a convenient isometry}\label{sec_isometry}
In \cite[Section 16.1]{BolJanRio07}, Bollobás, Janson and Riordan proved a sharp bound for the operator norm of a particular inhomogeneous random graph          called the {\em uniformly grown random graph}. For this, they introduce an isometry $f\mapsto \e^{-x/2} f(\e^{-x})$ mapping $L^2(0,1)$ to $L^2(0,\infty)$ to calculate the norm of the offspring operator. 
Recall the definition of $\kappa$ in \eqref{kappa-PAM-fin-def-form-a-rep}. In terms of the same isometry under a rescaling, for $x\neq y$, we define 
	\eqan{
	\tilde \kappa^\circ_{\zeta}((x,s),(y,t))&=\e^{-x/(2\chi-1)} \kappa((\e^{-x/(\chi-\tfrac{1}{2})},s),(\e^{-y/(\chi-\tfrac{1}{2})},t))\e^{-y/(2\chi-1)}\indic{b_1\leq x,y\leq b_2}\nn\\
	&=c_{st}(\indic{x<y, t=\srO}+\indic{x>y, t=\srY})\e^{-|x-y|}\indic{b_1\leq x,y\leq b_2},\label{eq_hd_tkappae}
	}
where $b_1=b_1({\zeta})=(\chi-\tfrac{1}{2})\log(1/(1-{\zeta}))$ and $b_2=b_2({\zeta})=(\chi-\tfrac{1}{2})\log(1/(3{\zeta}))$. 
{Note that}
\begin{align}\label{eq_hd_inner_isometry-ori}
    \langle {\bf 1}, \bfT_{\kappa^\circ_{\zeta}}^{k}{\bf 1}\rangle=&\sum_{\substack{q_i\in\cbc{\srO,\srY},\\\forall i\in[0,k]}}\int_0^1\cdots\int_0^1\prod_{i\in[k]}\kappa\big((y_{i-1},q_{i-1}),(y_i,q_i)\big)\nn\\
    &\qquad \times\indic{y_i\in[3{\zeta},1-{\zeta}]~\forall 0\leq i\leq k}\dint y_0\cdots\dint y_k.
\end{align}
Then, combining \eqref{eq_hd_tkappae} and \eqref{eq_hd_inner_isometry-ori} with the substitutions $y_i=\e^{-x_i/(\chi-\frac{1}{2})}$ gives that
\begin{align}\label{eq_hd_inner_isometry}
    \langle {\bf 1}, \bfT_{\kappa^\circ_{\zeta}}^{k}{\bf 1}\rangle=&\sum_{\substack{q_i\in\cbc{\srO,\srY},\\\forall i\in[0,k]}}\int_0^\infty\cdots\int_0^\infty\prod_{i\in[k]}\kappa\big((\e^{-x_{i-1}/(\chi-\frac{1}{2})},q_{i-1}),(\e^{-x_i/(\chi-\frac{1}{2})},q_i)\big)\\
    &\qquad\times\indic{x_i\in[b_1,b_2]~\forall 0\leq i\leq k}\dint \e^{-x_0/(\chi-\frac{1}{2})}\cdots\dint \e^{-x_k/(\chi-\frac{1}{2})}\nn\\
=&\bc{\frac{2}{2\chi-1}}^{k+1}\sum_{\substack{q_i\in\cbc{\srO,\srY},\\\forall i\in[0,k]}}\int_0^\infty\cdots\int_0^\infty\e^{-(x_0+x_k)/(2\chi-1)}\nn\\
&\qquad \times\prod_{i\in[k]}\tilde\kappa^\circ_{\zeta}\big((x_{i-1},q_{i-1}),(x_i,q_i)\big)\dint x_0\cdots\dint x_k\nn\\
\geq& 3{\zeta}\bc{\frac{2}{2\chi-1}}^{k+1}\langle {\bf 1}, \bfT_{\tilde\kappa^\circ_{\zeta}}^{k}{\bf 1}\rangle,\nn
\end{align}
where we use $\e^{-x/(\chi-\frac{1}{2})}\geq 3{\zeta}$ for $x\leq b_2$ in the last inequality. 

We now provide a \zhu{probabilistic} large-deviation interpretation for $\langle {\bf 1}, \bfT_{\tilde\kappa^\circ_{\zeta}}^{k}{\bf 1}\rangle$. For this, we define the random walk
	\eqn{
	S_\ell=\sum_{i=0}^\ell X_i,\nn
	}
 where $(X_i)_{i\geq 0}$ are independent random variables with probability density on $\mathbb{R}$ given by
\eqn{
	\label{density-X}
    f_{X_i}(x)=\begin{cases}
    \frac{1}{b_2-b_1}\indic{x\in[b_1,b_2]},\quad&i=0;\\
        \frac{1}{2}\e^{-\abs{x}},\quad&i\geq 1.
    \end{cases}
}

\zhu{Recall $q(x)$ from \eqref{def_label_direction} that
    \begin{align*}
    q(x)=\begin{cases}
        \rO,\quad&x\leq 0,\\
        \rY,\quad&\ch{x>0},
    \end{cases} 
\end{align*}}
and let $Q_0$ be a random label with equal probability to be $\rO$ and $\rY$, and independent of $(X_i)_{i\geq 0}$. Hence, $\zhu{q(-X_i)}$ represents the sign of $X_i$. 
Let 
\begin{align}
    \label{N(k)-def}N(k)=\abs{\cbc{i\in[\zhu{k-1}]}:~\zhu{q(-X_i)\neq q(-X_{i+1})}},
\end{align}
that is, the number of sign flips in $(X_i)_{i\in[k]}$. Then, $N(k)$ follows a binomial distribution with $k-1$ trials and success probability $1/2$.  

Using the notation above, our  large-deviation interpretation is as follows:
\begin{lemma}[Large-deviation interpretation]
\label{lem-LD-interpretation}
For all $k$, with $S_\ell$ and $N(k)$ defined in \eqref{N(k)-def},
	\eqn{
	\langle {\bf 1}, \bfT_{\tilde\kappa^\circ_{\zeta}}^{k}{\bf 1}\rangle
	=\Theta(b_2-b_1) (2c_{\srO\srO})^k \expec\Big[\Big(\frac{c_{\srO\srY}c_{\srY\srO}}{c_{\srO\srO}^2}\Big)^{N(k)/2}
	\indic{b_1\leq S_i\leq b_2~\forall i\in[k]}\Big].\label{eq_hd_inner_ldp}
	}
\end{lemma}
\begin{proof}
By \eqref{eq_hd_tkappae}, we compute that
    \zhu{\begin{align}
        \langle {\bf 1}, \bfT_{\tilde\kappa^\circ_{\zeta}}^{k}{\bf 1}\rangle=&\sum_{q_0\in\cbc{\srO,\srY}}\int_{-\infty}^{\infty}\cdots\int_{-\infty}^{\infty}\e^{-\sum_{i\in[k]}\abs{y_i-y_{i-1}}}c_{q_0,q(y_0-y_1)}\prod_{i=1}^{k-1} c_{q(y_{i-1}-y_i),q(y_i-y_{i+1})}\nn\\
    &\qquad\times\indic{b_1\leq y_i\leq b_2~\forall 0\leq i\leq k}\dint y_0\cdots\dint y_k\nn\\
    =&\sum_{q_0\in\cbc{\srO,\srY}}\int_{-\infty}^{\infty}\cdots\int_{-\infty}^{\infty}\e^{-\sum_{i\in[k]}\abs{x_i}}c_{q_0,q(-x_1)}\prod_{i=1}^{k-1} c_{q(-x_i),q(-x_{i+1})}\nn\\
    &\qquad\times\indic{b_1\leq \sum_{j=0}^i x_j\leq b_2~\forall 0\leq i\leq k}\dint x_0\cdots\dint x_k.\label{eq_ldp_re_1}
    \end{align}}
where in the last equation we use the change of variables by setting $x_0=y_0$ and $x_i=y_i-y_{i-1}$ for $i\in[k]$. To calculate the number of $c_{st}$ appearing in the above integral for each $s,t\in\cbc{\rO,\rY}$, we define
\zhu{\begin{align*}
    \bar{N}_{st}=\indic{Q_0=s,q(-X_1)=t}+\sum_{i=2}^k \indic{q(-X_{i-1})=s,q(-X_i)=t}\qquad \forall s,t\in\cbc{\rO,\rY}.
\end{align*}}
Then, $\bar{N}_{st}$ denotes the number of $(s,t)$ pairs appearing in \zhu{$\big((q(-X_{i-1}),q(-X_i))\big)_{i\in[k]}$}, where \zhu{$q(-X_0)=Q_0$}. 
By the definition of $N(k)$ {in \eqref{N(k)-def},}
\zhu{\begin{align}
N(k)=\bar{N}_{\srO\srY}+\bar{N}_{\srY\srO}-\indic{Q_0\neq q(-X_1)}.\label{relation_bk_n}
\end{align}}

Recall \eqref{density-X} and that $Q_0\in\cbc{\rO,\rY}$ with equal probability. Then, by \eqref{eq_ldp_re_1} and the definition of the expectation, 
\zhu{\begin{align}\langle {\bf 1}, \bfT_{\tilde\kappa^\circ_{\zeta}}^{k}{\bf 1}\rangle
    =&2^{k+1}(b_2-b_1)\Erw\Big[c_{Q_0,q(-X_1)}\prod_{i=1}^{k-1} c_{q(-X_i),q(-X_{i+1})}\indic{b_1\leq \sum_{j=0}^i x_j\leq b_2~\forall 0\leq i\leq k}\Big]\nn\\
=&2^{k+1}(b_2-b_1)\Erw\brk{c_{\srO\srO}^{\bar{N}_{\srO\srO}}c_{\srY\srY}^{\bar{N}_{\srY\srY}}c_{\srO\srY}^{\bar{N}_{\srO\srY}}c_{\srY\srO}^{\bar{N}_{\srY\srO}}\indic{b_1\leq \sum_{j=0}^i x_j\leq b_2~\forall 0\leq i\leq k}}\label{eq_hd_inner_ldp_0}\\
=&2^{k+1}c_{\srO\srO}^k(b_2-b_1)\Erw\bigg[\bc{\frac{c_{\srO\srY}}{c_{\srO\srO}}}^{\bar{N}_{\srO\srY}}\bc{\frac{c_{\srY\srO}}{c_{\srO\srO}}}^{\bar{N}_{\srY\srO}}\indic{b_1\leq S_i\leq b_2~\forall i\in[k]}\bigg],\nn
\end{align}}
where we recall $c_{\srO\srO}=c_{\srY\srY}$ from \eqref{cst-def-PAM-rep}, and we use the facts that $\bar{N}_{\srO\srO}+\bar{N}_{\srO\srY}+\bar{N}_{\srY\srO}+\bar{N}_{\srY\srY}=k$ and $S_0=X_0\in[b_1,b_2]$ in the last equation. We further note that the factor $2^{k+1}$ in \eqref{eq_hd_inner_ldp_0} comes from the combination of the density function $\frac{1}{2}\e^{-|x|}$ of $X_i$ for $i\in[k]$ and the two choices for $Q_0$, while the factor $b_2-b_1$ arises \zhu{from} the density function of $X_0$.

On the other hand, for the sequence \zhu{$Q_0,q(-X_1),\ldots,q(-X_k)$}, $\bar{N}_{\srY\srO}$ counts the number of flips from $\rY$ to $\rO$ in this sequence, and $\bar{N}_{\srO\srY}$ counts the number of flips from $\rO$ to $\rY$. Furthermore, if  a flip from $\rY$ to $\rO$ has already been observed, the sequence must flip back to $\rY$ before  it can flip again from $\rY$ to $\rO$.
Consequently, $	|\bar{N}_{\srY\srO}-\bar{N}_{\srO\srY}|\leq 1$. Therefore, by \eqref{relation_bk_n},
\begin{align*}
    \frac{N(k)-1}{2}\leq\bar{N}_{\srO\srY},\bar{N}_{\srY\srO}\leq \frac{N(k)+2}{2},
\end{align*}
and we thus conclude from \eqref{eq_hd_inner_ldp_0} that
	\eqn{
	\langle {\bf 1}, \bfT_{\kappa_{{\zeta}}}^{k}{\bf 1}\rangle
	=\Theta(b_2-b_1) (2c_{\srO\srO})^k \expec\Big[\Big(\frac{c_{\srO\srY}c_{\srY\srO}}{c_{\srO\srO}^2}\Big)^{N(k)/2}
	\indic{b_1\leq S_i\leq b_2~\forall i\in[k]}\Big],\nn
	}
as desired.
\end{proof}
By \eqref{lem-LD-interpretation}, the proof of \Cref{pro_hd_spe_con} is reduced to bounding the expectation in \eqref{eq_hd_inner_ldp} from below. We complete this proof in the next section.

\subsection{Proof of \Cref{pro_hd_spe_con}: \zhu{\ch{R}andom walk analysis}}\label{sec_decoupling}
In this section, we prove \Cref{pro_hd_spe_con}. 
This proof is organized as follows:
First, we   simplify the expectation in \Cref{lem-LD-interpretation} by eliminating the term $(c_{\srO\srY}c_{\srY\srO}/c_{\srO\srO}^2)^{N(k)/2}$. Then, we transform the computation of this expectation into \zhu{estimating the probability that a random walk does not exit a given \ch{region}}. Finally, we complete the proof \zhu{using the symmetry and independence of the random walk's increments}.
\paragraph{Simplification of the expectation}
We note that, excluding the term $(c_{\srO\srY}c_{\srY\srO}/c_{\srO\srO}^2)^{N(k)/2}$, the expectation in \eqref{eq_hd_inner_ldp} is simply the probability that the random walk $S_\ell$ stays within $[b_1,b_2]$ for all $0\leq \ell\leq k$, meaning that  the time at which $S_\ell$ hits $\RR\backslash [b_1,b_2]$ is greater than $k$. To eliminate the term $(c_{\srO\srY}c_{\srY\srO}/c_{\srO\srO}^2)^{N(k)/2}$ in the expectation, we construct new random walks on $\RR$  and apply a measure transformation.

With $r=\frac{\sqrt{c_{\srO\srY}c_{\srY\srO}}}{c_{\srO\srO}}$, by the law of total probability, the expectation in \eqref{eq_hd_inner_ldp} {equals}
\begin{align}\label{eq_simp_exp_1}
    \Erw\brk{r^{N(k)}\indic{b_1\leq S_i\leq b_2~\forall i\in[k]}}
    =&\sum_{a=0}^{k-1}\PP\bc{N(k)=a}r^a \Erw\brk{\indic{b_1\leq S_i\leq b_2~\forall i\in[k]}\mid N(k)=a}\nn\\
    =&\sum_{a=0}^{k-1}\binom{k-1}{a}\bc{\frac{1}{2}}^{k-1} r^a \Erw\brk{\indic{b_1\leq S_i\leq b_2~\forall i\in[k]}\mid N(k)=a}\\
    =&\bc{\frac{1+r}{2}}^{k-1}\sum_{a=0}^{k-1}\binom{k-1}{a} \bc{\frac{r}{1+r}}^a\bc{\frac{1}{1+r}}^{k-1-a}\nn\\
    &\qquad\times\Erw\brk{\indic{b_1\leq S_i\leq b_2~\forall i\in[k]}\mid N(k)=a},\nn
\end{align}
where, in the second equation, we use the fact that $N(k)$ follows a binomial distribution with $k-1$ trials and success probability $1/2$.

Note that $\binom{k-1}{a} \big(\frac{r}{1+r}\big)^a\big(\frac{1}{1+r}\big)^{k-1-a}$ $(0\leq a\leq k-1)$ is {the probability mass function of a binomial with $k-1$ trials and success \ch{probability} $r/(1+r)$.}{}
We next construct a sequence of random variables $(A_i)_{i\geq 1}$ such that the number of sign flips \ch{in} the first \ch{$i$} elements \ch{has the same} distribution \ch{as $A_i$ for all $i\in [k]$}. 

To construct $(A_i)_{i\geq 1}$, we first let $(Y_i,Z_i)_{i\geq 1}$ be independent random variables \ch{that are} independent of everything else, such that
\begin{itemize}
    \item[$\rhd$] for any $i\geq 1$, $f_{Y_i}(x)=\e^{-x}\indic{x\geq 0}$;
    \item[$\rhd$] $Z_1$ takes the values $1$ and $-1$ with equal probability;
    \item[$\rhd$] for any $i\geq 2$, $Z_i$ takes \ch{the values} $-1$ with probability $r/(1+r)$, and $1$ with probability $1/(1+r)$.
\end{itemize}
Then, for any $\ell\geq 1$, let
\begin{align}
A_\ell=Y_\ell\prod_{i=1}^\ell Z_i, \quad\text{and}\quad T_\ell=X_0+\sum_{i\in[\ell]}A_i.\label{def_Tl}
\end{align}

We use $\sgn(x)$ to denote the sign function of a real number $x$ such that 
\begin{align*}
\sgn(x)=\begin{cases}
    1 ,\quad&x\geq 0;\\
    -1,\quad&x<0.
\end{cases}   
\end{align*}
Define $\hat{N}(k)$ as the number of $i\in[k-1]$ such that the signs of $A_i$ and $A_{i+1}$ are different, that is, $Z_{i+1}=-1$. Then, conditionally on $\hat{N}(k)=N(k)=a$, the two sets 
\[\cbc{i\in[k-1]:~\sgn\bc{\ch{X_i}}=\sgn\bc{\ch{X_{i+1}}}}
\quad\text{and}\quad \cbc{i\in[k-1]:~\sgn\bc{\ch{A_i}}=\sgn\bc{\ch{A_{i+1}}}}\]
are chosen uniformly at random from all the subsets of $[k-1]$ with size $a$, and independent of $(A_1,X_1)$ and $(\abs{A_i},\abs{X_i})_{i\geq 2}$. 

We further note that $A_1$ and $X_1$ have the same distribution, while $\abs{A_i}$ and $\abs{X_i}$ have the same distribution for $i\geq 2$, independently of $(A_1,X_1)$ and $X_0$. Then,
conditionally on $\hat{N}(k)=N(k)=a$, $(X_i)_{i\in[k]}$ and $(A_i)_{i\in[k]}$ have the same distribution, independent of $X_0$. Hence,
\begin{align*}
    \Erw\brk{\indic{b_1\leq S_i\leq b_2~\forall i\in[k]}\mid N(k)=a}=\Erw\brk{\indic{b_1\leq T_i\leq b_2~\forall i\in[k]}\mid \hat{N}(k)=a}.
\end{align*}
Since $\hat{N}(k)$ follows a binomial distribution with $k-1$ trials and success probability $r/(1+r)$, we conclude that
\begin{align}\label{eq_simp_exp_2}
    &\sum_{a=0}^{k-1}\binom{k-1}{a} \bc{\frac{r}{1+r}}^a\bc{\frac{1}{1+r}}^{k-1-a} \Erw\brk{\indic{b_1\leq S_i\leq b_2~\forall i\in[k]}\mid N(k)=a}\\
    &\qquad=\sum_{a=0}^{k-1}\PP(\hat{N}(k)=a) \Erw\brk{\indic{b_1\leq T_i\leq b_2~\forall i\in[k]}\mid \hat{N}(k)=a}\nn\\
    &\qquad=\PP\bc{b_1\leq T_i\leq b_2~\forall i\in[k]}.\nn
\end{align}
Combining \eqref{eq_simp_exp_1} and \eqref{eq_simp_exp_2} gives that
\begin{align}\label{eq_simp_exp_3}
    \Erw\brk{r^{N(k)}\indic{b_1\leq S_i\leq b_2~\forall i\in[k]}}
    =\bc{\frac{1+r}{2}}^{k-1}\PP\bc{b_1\leq T_i\leq b_2~\forall i\in[k]}.
\end{align}

\paragraph{Random walk with independent increments}According to \eqref{def_Tl},  $T_i$ in \eqref{eq_simp_exp_3} is a random walk with initial value $X_0$ and \textit{dependent} increments $A_i$, where the dependence arises from the signs of  $(A_i)_{i\geq 1}$. \zhu{To make} the expectation of the increments to be independent of the past, we construct a new random walk by summing all the increments $A_i$ between two sign changes.

Define the stopping times $(\tau_i)_{i\geq0}$ recursively such that 
\begin{align}\label{def_taus}
    \tau_0=1\quad\text{and}\quad \tau_i=\inf_{j > \tau_{i-1}}\cbc{\sgn(A_j)\neq \sgn(A_{j-1})}.
\end{align}
For any $\ell\geq 1$, let
\begin{align}
    B_\ell=\sum_{i=\tau_{\ell-1}}^{\tau_\ell-1}A_i\quad\text{and}\quad R_\ell=X_0+\sum_{i=1}^\ell B_i.\label{def_RWr}
\end{align}
{By}{} construction all the $A_j$ have the same sign for $j\in[\tau_{i-1}, \tau_i-1]$. Hence, the sequence
\begin{align*}
T_{\tau_{i-1}-1},T_{\tau_{i-1}},\ldots,T_{\tau_i-1}  
\end{align*}
is either an increasing or a decreasing sequence. Since $R_{i-1}=T_{\tau_{i-1}-1}$ and $R_i=T_{\tau_i-1}$, for any integer $j\in[\tau_{i-1}, \tau_i-1]$,
\begin{align*}
    \min\cbc{R_{i-1},R_i}\leq T_j\leq \max\cbc{R_{i-1},R_i}.
\end{align*}
We further note \ch{from} \eqref{def_taus} that  $\tau_i\geq i+1$, and thus conclude that 
\begin{align}\label{subset_relation_sequence_R_T}
    \cbc{b_1\leq R_i\leq b_2~\forall 0\leq i\leq k}\subseteq \cbc{b_1\leq T_i\leq b_2~\forall i\in[k]}.
\end{align}
 Consequently, by \eqref{eq_hd_inner_ldp}, \eqref{eq_simp_exp_3} and \eqref{subset_relation_sequence_R_T},
\begin{align}\label{eq_lower_inter_1Tk1}
\langle {\bf 1}, \bfT_{\tilde\kappa^\circ_{\zeta}}^{k}{\bf 1}\rangle
    &=\Theta (b_2-b_1) (2c_{\srO\srO})^k\bc{\frac{1+r}{2}}^{k-1}\PP\bc{b_1\leq T_i\leq b_2~\forall i\in[k]}\\
    &\geq \Theta (b_2-b_1) (c_{\srO\srO}+\sqrt{c_{\srO\srY}c_{\srY\srO}})^k\PP\bc{b_1\leq R_i\leq b_2~\forall 0\leq i\leq k}.\nn
\end{align}
Recall the spectral radius $\nu$ from \eqref{spectral_radius} and  $c_{st}$ from \eqref{cst-def-PAM-rep}. We thus can rewrite $\nu$ as
\begin{align*}
\nu=\frac{2}{2\chi-1}(c_{\srO\srO}+\sqrt{c_{\srO\srY}c_{\srY\srO}}).    
\end{align*}
We further recall from \eqref{eq_hd_inner_isometry} that 
\begin{align*}
   \langle {\bf 1}, \bfT_{\kappa^\circ_{\zeta}}^{k}{\bf 1}\rangle&\geq 3{\zeta}\bc{\frac{2}{2\chi-1}}^{k+1}\langle {\bf 1}, \bfT_{\tilde\kappa^\circ_{\zeta}}^{k}{\bf 1}\rangle.
\end{align*}
Then, \eqref{eq_lower_inter_1Tk1} yields that
\begin{align}\label{hd_eq_8_new}
   \langle {\bf 1}, \bfT_{\kappa^\circ_{\zeta}}^{k}{\bf 1}\rangle&\geq \Theta (b_2-b_1){\zeta}\nu^k\PP\bc{b_1\leq R_i\leq b_2~\forall 0\leq i\leq k}.
\end{align}
Hence, to prove \Cref{pro_hd_spe_con}, we are left to bound $\PP\bc{b_1\leq R_i\leq b_2~\forall 0\leq i\leq k}$ from below.

\paragraph{\zhu{Symmetry and independence of increments}}\zhu{To handle the event $\cbc{b_1\leq R_i\leq b_2~\forall 0\leq i\leq k}$, we use the symmetry and independence of increments.} We first note that there exists a sufficiently small $\zeta_0\in (0,1/10]$ such that, for any $\zeta\in (0,\zeta_0]$,
\begin{align*}
    b_2=(\chi-\tfrac{1}{2})\log(1/(3{\zeta}))>4(\chi-\tfrac{1}{2})\log(1/(1-{\zeta}))=4b_1,
\end{align*}
and $c=\min\big\{b_2/128,\sqrt{b_2}/64\big\}>(r+1)/r$. Then, $b_1<b_2/2-c<b_2/2+c<b_2$. 

By \eqref{def_RWr}, $(\abs{B_i})_{i\geq 1}$ are i.i.d.\ random variables. Moreover, with $\imag$ as the  imaginary unit, the characteristic function of $\abs{B_1}$ can be written as 
\begin{align*}
\Erw\brk{\e^{\imag t\abs{B_1}}}=&\sum_{u=1}^\infty \PP\bc{\tau_1-\tau_0=u}\bc{\Erw\brk{\e^{\imag tY_1}}}^u\\
    =&\sum_{u=1}^\infty \bc{\frac{1}{1+r}}^{u-1}\frac{r}{1+r}\bc{\frac{1}{1-\imag t}}^u=\frac{1}{1-\frac{r+1}{r}\imag t}.
\end{align*}
We note that $1/(1-\frac{r+1}{r}\imag t)$ is the characteristic function of the exponential distribution with a mean of $(1+r)/r$. 
Therefore, $\abs{B_i}$ follows this distribution. Next, for each integer $i\geq 1$, we construct a non-negative bounded random variable $C_i$ which is distributed as $(-1)^{i-1}Z_1 B_i$, conditional on $\abs{B_i}\leq c$.

Let $(C_i)_{i\geq 1}$ be i.i.d. random variables on $\RR$, independent of everything else, with probability density function
\begin{align*}
    f_{C_i}(x)=\frac{r}{r+1}\frac{\e^{-\frac{r}{r+1}x}}{1-\e^{-\frac{rc}{r+1}}}\indic{x\in[0,c]}.
\end{align*}

We note that $C_i$ has an exponential distribution with mean $(1+r)/r$, truncated to the interval $[0,c]$, and thus $\Erw\brk{C_i}<c$. Then, {conditionally}{} on the good event $\mathcal{E}_k=\cbc{\abs{B_i}\leq c~\forall i\in[k]}$, $\abs{B_i}$ is distributed as $C_i$, while $\sgn(B_i)=\sgn(A_{\tau_{i-1}})=(-1)^{i-1}Z_1$ by \eqref{def_RWr}. Hence, the random variables $(B_i)_{i\in[k]}$ under $\mathcal{E}_k$ are jointly distributed as $((-1)^{i-1}Z_1 C_i)_{i\in[k]}$. 
\zhu{Let $\tilde{X}_\ell = X_0 + \sum_{j\in[\ell]} (-1)^{j-1}Z_1C_j$. Then
\begin{align}\label{revise_hd_eq_7-1}
 \mathbb{P}\big(b_1\leq R_i\leq b_2~\forall 0\leq i\leq k\big)
 &\geq\mathbb{P}\bigg(X_0+\sum_{i\in[\ell]} B_i\in[b_1,b_2]~\forall 0\leq\ell\leq k~\Big\lvert~ \mathcal{E}_k\bigg)\mathbb{P}\big(\mathcal{E}_k\big)\nn\\
 &=\mathbb{P}\bigg(X_0+\sum_{i\in[\ell]} (-1)^{i-1}Z_1C_i\in[b_1,b_2]~\forall 0\leq\ell\leq k\bigg)\mathbb{P}\big(\mathcal{E}_k\big)\\
 &=\mathbb{P}\bigg(\tilde{X}_\ell\in[b_1,b_2]~\forall 0\leq\ell\leq k\bigg)\mathbb{P}\big(|B_1|\leq c\big)^k.\nn
\end{align}

Fix an even integer $L=2\lfloor(b_2-b_1)/(12c)\rfloor$ and $d=\lceil k/L\rceil$. We consider the event $\ch{\mathcal{M}}_j$ that
\begin{align*}
X_0,X_L,X_{2L},\ldots,X_{jL}\in\bigg[\frac{2b_1+b_2}{3},\frac{b_1+2b_2}{3}\bigg].
\end{align*}
We claim that $\ch{\mathcal{M}}_d$ indicates that $\tilde{X}_\ell\in[b_1,b_2]$ for all $0\leq\ell\leq k$. Indeed, since $C_i\in[0,c]$, for any integer $\ell\in[jL,(j+1)L]$,
\begin{align}\label{revise_hd_eq_7-2}
|X_{\ell}-X_{jL}|=\bigg|\sum_{i=jL+1}^{\ell} (-1)^{i-1}Z_1C_i\bigg|\leq cL\leq \frac{b_2-b_1}{6}.   
\end{align}
Consequently, if $\ch{\mathcal{M}}_d$ holds, for any integer $\ell\in[0,dL]$,
\begin{align*}
    b_1\leq X_{\lfloor \ell/L\rfloor L}-|X_\ell-X_{\lfloor \ell/L\rfloor L}|\leq X_\ell\leq X_{\lfloor \ell/L\rfloor L}+|X_\ell-X_{\lfloor \ell/L\rfloor L}|\leq b_2.
\end{align*}
Hence,
\begin{align}\label{revise_hd_eq_7-3}
    \mathbb{P}\big(\tilde{X}_\ell\in[b_1,b_2]~\forall 0\leq\ell\leq k\big)\geq\mathbb{P}(M_d).
\end{align}

On the other hand, since $\ch{(C_i)}_{i\geq 1}$ are i.i.d.\ random variables and $L$ is an even integer, $X_{(j+1)L}-X_{jL}=\sum_{i=jL+1}^{(j+1)L} (-1)^{i-1}Z_1C_i$ is symmetrically distributed about $0$. As a consequence, \eqref{revise_hd_eq_7-2} yields that
\begin{align*}
    \PP\bc{{X_{(j+1)L}-X_{jL}}\in\brk{0,\frac{b_2-b_1}{6}}}=\PP\bc{\ch{X_{(j+1)L}-X_{jL}}\in\brk{\frac{b_1-b_2}{6},0}}\geq \frac{1}{2}.
\end{align*}
Given $X_{jL} \in \big[(2b_1+b_2)/3,(b_1+2b_2)/3\big]$, since $\ch{(C_i)}_{i\geq 1}$ are i.i.d.\ random variables,
\begin{align*}
&\PP\bc{X_{(j+1)L} \in \brk{\frac{2b_1+b_2}{3},\frac{b_1+2b_2}{3}}\;\Big|\; X_{jL}}\\
    &\qquad\geq \PP\bc{{X_{(j+1)L}-X_{jL}}\in\brk{0,\frac{b_2-b_1}{6}}}\indic{X_{jL}\leq \frac{b_1+b_2}{2}}\\
    &\qquad\quad+\PP\bc{{X_{(j+1)L}-X_{jL}}\in\brk{\frac{b_1-b_2}{6},0}}\indic{X_{jL}> \frac{b_1+b_2}{2}}\geq \frac{1}{2}.
\end{align*}
As a consequence,
\begin{align*}
    \PP\bc{\ch{\mathcal{M}}_{j+1}\mid \ch{\mathcal{M}}_j}&\geq \inf_{X_{jL} \in \big[(2b_1+b_2)/3,(b_1+2b_2)/3\big]}\PP\bc{X_{(j+1)L} \in \brk{\frac{2b_1+b_2}{3},\frac{b_1+2b_2}{3}}\;\Big|\; X_{jL}}\\
    &\geq \frac{1}{2}.
\end{align*}
Hence,
\begin{align}\label{revise_hd_eq_7-4}
    \mathbb{P}(\ch{\mathcal{M}}_d)&=\mathbb{P}(\ch{\mathcal{M}}_0)\prod_{i=1}^d\mathbb{P}(\ch{\mathcal{M}}_i\mid \ch{\mathcal{M}}_{i-1})\nn\\
    &\geq \frac{1}{2^d}\,\mathbb{P}\Big(X_0 \in \Big[\frac{2b_1+b_2}{3},\frac{b_1+2b_2}{3}\Big]\Big)=\frac{1}{3\cdot 2^d}.
\end{align}
Combining \eqref{hd_eq_8_new}, \eqref{revise_hd_eq_7-1}, \eqref{revise_hd_eq_7-3} and \eqref{revise_hd_eq_7-4} yields that
\begin{align}\label{revise_hd_eq_7-5}
    \langle \mathbf{1},\mathbf{T}_{\kappa^\circ_{\zeta}}^{k}\mathbf{1}\rangle\geq \Theta (b_2-b_1)\zeta \, 2^{-6ck/(b_2-b_1)}\nu^k\,\mathbb{P}(|B_1|\leq c)^k.
\end{align}
Let $\Theta_{\star,0}$ be a positive lower bound of all $\Theta$ in this section. Let $\uhvep = 2^{-6c/(b_2-b_1)}\nu\,\mathbb{P}(|B_1|\leq c)$ and $c_{\zeta}=\min\{\Theta_{\star,0} (b_2-b_1)\zeta,1\}\in(0,1]$. We note that $\lim_{\zeta\searrow0}c/(b_2-b_1)=0$ and $\lim_{\zeta\searrow0}c=\infty$, which also implies that $\lim_{\zeta\searrow0}\mathbb{P}(|B_1|\leq c)=1$. Hence,
$\lim_{\zeta\searrow0}\uhvep=\nu$.
\Cref{pro_hd_spe_con} then follows directly from \eqref{revise_hd_eq_7-5}, as desired.\qed
}

\section{Conditional lower bounds on the expected number of $k$-step paths}
\label{sec-first-moment-proof-thm-log-PA-delta>0}
\Cref{pro_hd_spe_con} establishes the relationship between the original graph and its truncated version. Utilizing this relationship, we examine the first moment of the number of 
$k$-step paths in this section, and use this to prove \Cref{prop-path-count-PAMc-LB_simple}.

The proof of \Cref{prop-path-count-PAMc-LB_simple} {is the most technically challenging in this paper, since our path-counting formulas are better suited for upper than for lower bounds. The proof follows that}{} of \Cref{lem_hd_lower_part1}, and is organized as follows: First, in \Cref{sec_fall_fact_path_prob}, we express the conditional path probability $\prob^r(\vpie\subseteq \PAndmdel)$ in a suitable form, similar to \eqref{eq_hd_lower_part0}. Next, in \Cref{sec-rew-pc-bound-kappa}, we count the number of possible edge-labeled self-avoiding paths $\vpie$ given $\vec\pi$, which gives a sum of products of kernels of the path probability as in \eqref{eq_hd_lower_form_prod_kappa-old}. Then, given a lower bound on this sum, we complete the proof of \Cref{prop-path-count-PAMc-LB_simple} in \Cref{sec-given-com-dif-self-con-proof}. \Cref{sub-sec-modify-kappa-lower-bound} proves a weaker version of this lower bound, where certain avoidance constraints are ignored. This lower bound is a key ingredient in the full proof, and it is also where \Cref{pro_hd_spe_con} is used. \longversion{The remainder of its proof is deferred to \Cref{app-E}.}\shortversion{We refer to \cite[Appendix E]{HofZhu25long} for the remainder of its proof.} 

\invisible{These proof parts are organized as follows. In \Cref{sec-path-upper-bound}, we first give a crude upper bound that will later be used to bound error terms. In \Cref{sub-sec-modify-kappa-lower-bound}, we start by lower bounding path counts that do not have the self-avoidance constraint, for which we slightly modify the kernel $\kappa$. In \Cref{eq_hd_sum_restrict_path-sum-gen-1-2}, we deal with the self-loops that have been artificially added, while in \Cref{eq_hd_sum_restrict_path-sum-gen-1-1} we deal with back longer cycles and intersections.
}

\subsection{Falling factorial form of path probability}\label{sec_fall_fact_path_prob}
In \eqref{eq_hd_lower_part0}, we have expressed $\prob(\vpie\subseteq \PAndmdel)$ for $\vpie\in \ESA_{\nVer_1,\nVer_2}^{\sss e,k}$ in terms of falling factorials, which allowed us to handle the subsequent path-counting upper bounds in \Cref{sec_trun_lowerbound}. We now adapt these arguments from upper to lower bounds. Recall that
\begin{align*}
    N_{n,r}(k)=\sum_{\vpie\in \NESA^{\sss e,k,r}} \indic{\vpie\subseteq \PAndmdel}, 
\end{align*}
where $\NESA^{\sss e,k,r}$ is the set of all possible  $k$-step edge-labeled self-avoiding paths $\vpie$ with $\pi_0\in \partial B_r^{\sss(G_n)}(\Ver_1)$ and $\pi_k\in \partial B_r^{\sss(G_n)}(\Ver_2)$ in \ch{$[n]$}, and we restrict to $\pi_i\not\in B_r^{\sss(G_n)}(\Ver_1)\cup B_r^{\sss(G_n)}(\Ver_2)$ for $i\in[k-1]$ and $\pi_i\geq {\zeta} n$ for all $0\leq i\leq k$. 
Given $\Ver_1$ and $\Ver_2$ and $\nVer_1\not\in B_{r-1}^{\sss(G_n)}(\Ver_1)\cup B_r^{\sss(G_n)}(\Ver_2)$ and $\nVer_2\not\in B_r^{\sss(G_n)}(\Ver_1)\cup B_{r-1}^{\sss(G_n)}(\Ver_2)$, 
let $\TESA_{\nVer_1,\nVer_2}^{\sss e,k}$ be the set of paths $\vpie$ with $\pi_0=\nVer_1$ and $\pi_\ell=\nVer_2$, subject to the same avoidance constraints as $\NESA^{\sss e,k,r}$. Then,
    \begin{align*}
         \NESA^{\sss e,k,r}=\bigcup_{\substack{\nVer_1\in \partial B_r^{\sss(G_n)}(\Ver_1),\\\nVer_2\in \partial B_r^{\sss(G_n)}(\Ver_2)}}\TESA_{\nVer_1,\nVer_2}^{\sss e,k}.
    \end{align*}

Let $\TSA_{u,v}^{\sss k}$ be the set of all possible self-avoiding paths $\vec\pi$ of length $k$ in \ch{$[n]$}, where $\pi_0=u$ and $\pi_k=v$, subject to the same restrictions as $\TESA_{u,v}^{\sss e, k}$. 

Recall the good event $\GNW$ from \Cref{def_gnw}, the definitions of $p_s$ and $q_s$ in \eqref{def_hd_pq}, and the definition of $\mathscr{F}_{r}$ in \eqref{sigma-alg-SMM-PAM}. In this section, we give the following bound on path probabilities:
\begin{lemma}[Falling factorial form of $\prob^r(\vpie\subseteq \PAndmdel)$]\label{lem_prob_pie_cond}
    For any integers $\nVer_1,\nVer_2\in[\zeta n,n]$ such that $\nVer_1\not\in B_{r-1}^{\sss(G_n)}(\Ver_1)\cup B_r^{\sss(G_n)}(\Ver_2)$ and $\nVer_2\not\in B_r^{\sss(G_n)}(\Ver_1)\cup B_{r-1}^{\sss(G_n)}(\Ver_2)$, and any $\vpie \in \TESA_{\nVer_1,\nVer_2}^{\sss e,k}$, with $k\leq 2\log_\nu n$, on the good event $\GNW$,
\eqan{\label{eq_hd_prod_simp_pq}
	\prob^r(\vpie\subseteq \PAndmdel) =\Theta^{\frac{1}{\log n}}\prod_{s=2}^n
	\frac{(\alpha+p_s-1)_{p_s}(\beta_s+q_s-1)_{q_s}}{(\alpha+\beta_s+p_s+q_s-1)_{p_s+q_s}},
	}
where $\prob^r$ is  the conditional probability given $\mathscr{F}_{r}$.
\end{lemma}
We note that \eqref{eq_hd_prod_simp_pq} is quite similar to \eqref{eq_hd_lower_part0}, except that in \eqref{eq_hd_prod_simp_pq}, we condition on
\begin{align*}
\mathscr{F}_{r}=\sigma\Big(\bar{B}_r^{\sss(G_n)}(\Ver_1), \bar{B}_r^{\sss(G_n)}(\Ver_2), (\psi_v)_{v\in B_{r-1}^{\sss(G_n)}(\Ver_1)\cup B_{r-1}^{\sss(G_n)}(\Ver_2)}\Big),
\end{align*}
which results in an extra factor of $\Theta^{1/\log n}$.
To obtain \eqref{eq_hd_prod_simp_pq}, we first calculate the probability of 
$
\{E\subseteq \PAndmdel,\bar{B}_r^{\sss(G_n)} (\Ver_1) =\bar{{\bf t}}_1,  \bar{B}_r^{\sss(G_n)}(\Ver_2)=\bar{{\bf t}}_2\}$
for an edge set $E$, given $(\psi_v)_{v\in[n]}$. Then, we take the ratio of this probability with $E=\vpie$ {and}{} the probability with $E=\emptyset$. Finally, the target conditional probability in \eqref{eq_hd_prod_simp_pq}  follows from taking the conditional expectation given $\mathscr{F}_{r}$. 


\begin{proof}
Let $V(\bar{\bf t})$ be the vertex set of \ch{the} tree $\bar{\bf t}$ and $V^{\sss\circ}(\bar{\bf t})$ be the set of \zhu{non-leaf} vertices in $\bar{\bf t}$.
With $\prob_n$ the conditional law given $(\psi_v)_{v\in[n]}$, we consider
\begin{align*}
    \prob_n(E\subseteq \PAndmdel,\bar{B}_r^{\sss(G_n)}(\Ver_1)=\bar{{\bf t}}_1, \bar{B}_r^{\sss(G_n)}(\Ver_2)=\bar{{\bf t}}_2)
\end{align*}
for $(\bar{\bf t}_1,\bar{\bf t}_1)$ and $(\psi_v)_{v\in V^{\sss\circ}(\bar{{\bf t}}_1)\cup V^{\sss\circ}(\bar{{\bf t}}_2)}$ that are good, and edge set $E=\cbc{(\ell_h,i_h,j_h):~h\in [k']}$ such that
\begin{itemize}
    \item[$\rhd$] $E$ is a good edge set (see \Cref{ass_e}), and thus $k'\leq 4\log_\nu n$;
    \item[$\rhd$] $V_E\cap \big(V^{\sss\circ}(\bar{{\bf t}}_1)\cup V^{\sss\circ}(\bar{{\bf t}}_2)\big)=\emptyset$;
    \item[$\rhd$] $E\cap\big( \bar{\bf t}_1\cup\bar{\bf t}_2\big)=\emptyset$;
\end{itemize}
 We further note that the empty set $\emptyset$, or each $\vpie\in \TESA_{\nVer_1,\nVer_2}^{\sss e,k}$, is an edge set satisfying the above properties.

As in \eqref{def_hd_pq}, we define
\begin{align}\label{def_hd_p'-q'}
    p_s'=\sum_{e\in \bar{{\bf t}}_1\cup\bar{{\bf t}}_2}\indic{s=\ushort{e}}\quad\text{and}\quad
    q_s'=\sum_{e\in  \bar{{\bf t}}_1\cup\bar{{\bf t}}_2}\indic{s\in (\ushort{e},\bar{e})}.
\end{align}
If $\bar{B}_r^{\sss(G_n)}(\Ver_1)=\bar{{\bf t}}_1$, $\bar{B}_r^{\sss(G_n)}(\Ver_2)=\bar{{\bf t}}_2$ and $E\subseteq\PAndmdel$, then
\begin{itemize}[leftmargin=20pt]
    \item[$\rhd$] vertex $\Ver_i$ equals $\varnothing_i$, the root of tree ${\bf t}_i$, for $i=1,2$. As $\Ver_1$ and $\Ver_2$ are uniformly distributed within $[n]$, $\PP_n(\Ver_1=\varnothing_1,\Ver_2=\varnothing_2)=n^{-2}$;
    \item[$\rhd$] all the edges in $\bar{{\bf t}}_1$, $\bar{{\bf t}}_2$ and $E$ are present. By \eqref{eq_con_indpedent_present}, 
    \begin{align*}
       \PP_n(E\cup\bar{\bf t}_1\cup\bar{\bf t}_2\subseteq\PAndmdel\mid\Ver_1=\varnothing_1,\Ver_2=\varnothing_2)= \prod_{s=2}^n \psi_s^{\psE+p_s'} (1-\psi_s)^{\qsE+q_s'},
    \end{align*}
where we recall the definitions of $\psE$ and $\qsE$ from \eqref{def_hd_pq_gen};
    
    \item[$\rhd$] as all neighbors of the vertices in $V^{\sss\circ}(\bar{{\bf t}}_1)\cup V^{\sss\circ}(\bar{{\bf t}}_2)$ are present in $\bar{{\bf t}}_1\cup \bar{{\bf t}}_2$, no further incoming edges to the vertices in $V^{\sss\circ}(\bar{{\bf t}}_1)\cup V^{\sss\circ}(\bar{{\bf t}}_2)$ are allowed. By \eqref{eq_con_prob_present}, for each $u(j)$ not present in $ \bar{{\bf t}}_1\cup \bar{{\bf t}}_2\cup E$,
    \begin{align*}
        \PP_n\Big(u(j)\not\in V^{\sss\circ}(\bar{{\bf t}}_1)\cup V^{\sss\circ}(\bar{{\bf t}}_2)\Big)&=1-\sum_{\substack{v\in V^{\sss\circ}(\bar{{\bf t}}_1)\cup V^{\sss\circ}(\bar{{\bf t}}_2),\\v<u}}\PP_n((u,j,v)\in\PAndmdel)\\
        &=1-\sum_{\substack{v\in V^{\sss\circ}(\bar{{\bf t}}_1)\cup V^{\sss\circ}(\bar{{\bf t}}_2),\\v<u}}\psi_v\prod_{s\in(v,u)}(1-\psi_s),
    \end{align*}
    where $u(j)$ is \zhu{the} vertex receiving the $j$th out-edge of vertex $u$.
    We further note that both ends of edges in $ \bar{{\bf t}}_1\cup \bar{{\bf t}}_2\cup E$ are fixed, and, therefore, they should be excluded from this no-further-edge probability calculation. 
    Consequently, the conditional independence of edges yields that
    \begin{align*}
        &\PP_n(E\subseteq \PAndmdel,\bar{B}_r^{\sss(G_n)}(\Ver_i)=\bar{{\bf t}}_i, ~i=1,2\mid E\cup\bar{\bf t}_1\cup\bar{\bf t}_2\subseteq\PAndmdel,\Ver_i=\varnothing_i, i=1,2)\\
        &\qquad=\prod_{u(j)\not\in \bar{{\bf t}}_1\cup \bar{{\bf t}}_2\cup E}\PP_n\Big(u(j)\not\in V^{\sss\circ}(\bar{{\bf t}}_1)\cup V^{\sss\circ}(\bar{{\bf t}}_2)\Big)\\
        &\qquad=\prod_{u(j)\not\in \bar{{\bf t}}_1\cup \bar{{\bf t}}_2\cup E}\Big[1-\sum_{\substack{v\in V^{\sss\circ}(\bar{{\bf t}}_1)\cup V^{\sss\circ}(\bar{{\bf t}}_2),\\v<u}}\psi_v\prod_{s\in(v,u)}(1-\psi_s)\Big].
    \end{align*} 
\end{itemize}
Hence, we conclude that
\begin{align}\label{Pn-two-neighborhood-PA-E}
    &\prob_n(E\subseteq \PAndmdel,\bar{B}_r^{\sss(G_n)}(\Ver_1)=\bar{{\bf t}}_1, \bar{B}_r^{\sss(G_n)}(\Ver_2)=\bar{{\bf t}}_2)\\
    =&\frac{1}{n^2}\prod_{s=2}^n\psi_s^{\psE+p_s'} (1-\psi_s)^{\qsE+q_s'}  \prod_{u(j)\not\in \bar{{\bf t}}_1\cup \bar{{\bf t}}_2\cup E}\Big[1-\sum_{\substack{v\in V^{\sss\circ}(\bar{{\bf t}}_1)\cup V^{\sss\circ}(\bar{{\bf t}}_2),\\v<u}}\psi_v\prod_{s\in(v,u)}(1-\psi_s)\Big].\nn
\end{align}
Specifically, for $E=\varnothing$,  \eqref{Pn-two-neighborhood-PA-E} gives that
	\eqan{
	\label{Pn-two-neighborhood-PA-rep}
	&\prob_n(\bar{B}_r^{\sss(G_n)}(\Ver_1)=\bar{{\bf t}}_1, \bar{B}_r^{\sss(G_n)}(\Ver_2)=\bar{{\bf t}}_2)\\
	&\quad=\frac{1}{n^2}\prod_{s=2}^n \psi_s^{p_s'} (1-\psi_s)^{q_s'} \prod_{u(j)\not\in \bar{{\bf t}}_1\cup \bar{{\bf t}}_2}\Big[1-\sum_{\substack{v\in V^{\sss\circ}(\bar{{\bf t}}_1)\cup V^{\sss\circ}(\bar{{\bf t}}_2),\\v<u}}\psi_v\prod_{s\in(v,u)}(1-\psi_s)\Big].\nn
	}
The proof of \eqref{Pn-two-neighborhood-PA-rep} can also be \zhu{found} in \cite[(5.27)]{GarHazHofRay22}.

Consequently, dividing \eqref{Pn-two-neighborhood-PA-E} by \eqref{Pn-two-neighborhood-PA-rep}, we compute that
\begin{align}\label{eq_hd_lpath_probability_cond_r}
    &\PP(E\subseteq \PAndmdel\mid \bar{B}_r^{\sss(G_n)}(\Ver_1)=\bar{{\bf t}}_1, \bar{B}_r^{\sss(G_n)}(\Ver_2)=\bar{{\bf t}}_2,(\psi_v)_{v\in[n]})\\
    &\qquad= \prod_{s=2}^n\psi_s^{\psE} (1-\psi_s)^{\qsE}  \prod_{u(j)\in E}\Big[1-\sum_{\substack{v\in V^{\sss\circ}(\bar{{\bf t}}_1)\cup V^{\sss\circ}(\bar{{\bf t}}_2),\\v<u}}\psi_v\prod_{s\in(v,u)}(1-\psi_s)\Big]^{-1}.\nn
\end{align}

Comparing \eqref{eq_hd_lpath_probability_cond_r} with the formula for $\PP_n(\vpie\subseteq \PAndmdel)$ in \eqref{eq_hd_lpath_probability}, we have an extra term
\begin{align*}
\prod_{u(j)\in E}\Big[1-\sum_{\substack{v\in V^{\sss\circ}(\bar{{\bf t}}_1)\cup V^{\sss\circ}(\bar{{\bf t}}_2),\\v<u}}\psi_v\prod_{s\in(v,u)}(1-\psi_s)\Big]^{-1},
\end{align*}
which arises from the conditioning on $\mathscr{F}_r$. By \Cref{ass_Br_2}, on the good event $\GNW$, this term can be bounded as 
\begin{align}\label{eq_product_1-psi}
    1&\leq \prod_{u(j)\in E}\Big[1-\sum_{\substack{v\in V^{\sss\circ}(\bar{{\bf t}}_1)\cup V^{\sss\circ}(\bar{{\bf t}}_2),\\v<u}}\psi_v\prod_{s\in(v,u)}(1-\psi_s)\Big]^{-1}\\
    &\leq \Big(1-\sum_{v\in V^{\sss\circ}(\bar{{\bf t}}_1)\cup V^{\sss\circ}(\bar{{\bf t}}_2) }\psi_v\Big)^{\abs{E}}\leq \e^{\abs{E}/\log^2 n} = \Theta^{\frac{1}{\log n}}\nn.
\end{align}

Hence, taking the conditional expectation to \eqref{eq_hd_lpath_probability_cond_r} given $\mathscr{F}_{r}$, we conclude from \eqref{expctation_function_beta} and \eqref{eq_product_1-psi} that, on the good event $\GNW$,
\begin{align}\label{eq_E-in-PAM-cond}
    \PP^r(E\subseteq \PAndmdel)&= \Theta^{\frac{1}{\log n}}\prod_{s=2}^n
	\frac{(\alpha+\psE-1)_{\psE}(\beta_s+\qsE-1)_{\qsE}}{(\alpha+\beta_s+\psE+\qsE-1)_{\psE+\qsE}},
\end{align}
Since $\vpie\in \TESA_{\nVer_1,\nVer_2}^{\sss e,k}$,  \eqref{eq_hd_prod_simp_pq} follows directly from taking $E$ to be $\vpie$.
\end{proof}

\subsection{Rewriting path-counting bounds in terms of $\kappa$}\label{sec-rew-pc-bound-kappa}
In this section, we aim to compute the sum of $\prob^r(\vpie\subseteq \PAndmdel)$  over  all paths $\vpie$ in $ \TESA_{\nVer_1,\nVer_2}^{\sss e,k}$. With \eqref{eq_hd_prod_simp_pq} in hand, the problem is then reduced to summing the products on  the rhs of \eqref{eq_hd_prod_simp_pq}, and we use the approaches in \Cref{sec_trun_lowerbound} to address this.  

Given a path $\vec\pi$, recall that we assign a label to each vertex in this path  in \eqref{def_label_direction} as 
\begin{align}\label{def_label_direction-r}
    \lb^\pi_i=q(\pi_i-\pi_{i-1})=\begin{cases}
        \rO,\quad&\text{if }\pi_i\leq \pi_{i-1};\\
        \rY,\quad&\text{if }\pi_i>\pi_{i-1};
    \end{cases}
\end{align}
while  we use $\lb^\pi_0$ as a variable, not previously defined, with range in $\cbc{\rO,\rY}$. 
Often, the value of $\lb^\pi_0$ is not important, because, as shown in \eqref{kappa-PAM-fin-def-form-a-rep}, for distinct $x,y\in [0,1]$,
\begin{align}\label{eq-change-label-add-theta1}
    \kappa\big((x,\rO),(y,q(y-x))\big)=\Theta \kappa\big((x,\rY),(y,q(y-x))\big)=\Theta\frac{1}{(x\vee y)^{\chi}(x\wedge y)^{1-\chi}},
\end{align}
i.e., the value of $\lb^\pi_0$ only affects $\kappa\big((x,\lb^\pi_0),(y,q(y-x))\big)$ by a $\Theta $ factor.

The target of this section is the proof of the  following lemma:
\begin{lemma}[Kernel products summation  form of the path probability]\label{lem_hd_sum_restrict_path}
For any $n\geq \zeta^{-2}$, $k\leq 2\log_\nu n$, $\lb^\pi_0\in\cbc{\rO,\rY}$, and integers $\nVer_1,\nVer_2\in[\zeta n,n]$ such that $\nVer_1\not\in B_{r-1}^{\sss(G_n)}(\Ver_1)\cup B_r^{\sss(G_n)}(\Ver_2)$ and $\nVer_2\not\in B_r^{\sss(G_n)}(\Ver_1)\cup B_{r-1}^{\sss(G_n)}(\Ver_2)$,
\begin{align}\label{eq_hd_sum_restrict_path}
    &\sum_{\vpie\in \TESA_{\nVer_1,\nVer_2}^{\sss e,k}}\prod_{s=2}^n\frac{(\alpha+p_s-1)_{p_s}(\beta_s+q_s-1)_{q_s}}{(\alpha+\beta_s+p_s+q_s-1)_{p_s+q_s}}\\
        &\qquad=\Theta n^{-k}\sum_{\vec\pi\in \TSA_{\nVer_1,\nVer_2}^{\sss k}}\prod_{i=1}^k\kappa\big((\pi_{i-1}/n,\lb^\pi_{i-1}),(\pi_{i}/n,\lb^\pi_i)\big).\nn
\end{align}
\end{lemma}

\begin{proof}
Recall $N_{\srO\srY}=\sum_{i\in[k-1]}\indic{(\lb^\pi_i,\lb^\pi_{i+1})=(\rO,\rY)}=\sum_{s\in\vec\pi}\indic{p_s=2}$ from \eqref{def_NOY}. Further, with $k\leq 2\log_\nu n$ and $\bdt =\lceil \zeta n\rceil$, it is proved in \eqref{eq_hd_prob_one_path} that, for any $\vpie\in \TESA_{\nVer_1,\nVer_2}^{\sss e,k}\subseteq \ESA_{\nVer_1,\nVer_2}^{\sss e,k,(b)}$,
\begin{align}\label{eq_hd_prob_one_path-rep}
    &\prod_{s=2}^n\frac{(\alpha+p_s-1)_{p_s}(\beta_s+q_s-1)_{q_s}}{(\alpha+\beta_s+p_s+q_s-1)_{p_s+q_s}}\\
    &\ \ = \Theta^{\frac{\log^2 n}{\bdt }}\bc{\frac{m+\delta}{2m+\delta}}^{k-N_{\srO\srY}}\bc{\frac{m+1+\delta}{2m+\delta}}^{N_{\srO\srY}}\prod_{i\in[k]}\frac{1}{(\pi_{i-1}\wedge\pi_i)^{1-\chi}(\pi_{i-1}\vee\pi_i)^\chi }.\nn
\end{align}
Then, \ch{$\frac{\log^2 n}{\bdt }\leq \frac{\log^2 n}{\sqrt{n}}$ since $n\geq \zeta^{-2}$}. Hence, $\frac{\log^2 n}{\bdt }$ has a uniform upper bound $\sup_{n\geq 1}\frac{\log^2 n}{\sqrt{n}}$ for all $n\geq \zeta^{-2}$. Consequently, the definition of $\Theta$ in \Cref{def-Theta} yields that
\begin{align}\label{eq-b-stay-theta}
\Theta^{\frac{\log^2 n}{\bdt }}=\Theta. 
\end{align}

Fix $\bar{\bf t}_1$, $\bar{\bf t}_2$ and the path $\vec{\pi}$. We consider the number of possible edge-labeled self-avoiding paths $\vpie$ given  $\vec\pi\in \TESA_{\nVer_1,\nVer_2}^{\sss k}$.
As in \Cref{sec_trun_lowerbound}, for each $1\leq \ell\leq k-2$, since $\vec\pi$ is self-avoiding, when we have fixed the edge-labels of edges in the subpath $(\pi_0,\pi_1,\ldots,\pi_\ell)$, the number of  choices for the labeled edge between $\pi_\ell$ and $\pi_{\ell+1}$ is equal to $m-1$ when $\lb^\pi_\ell=\rY$ and $\lb^\pi_{\ell+1}=\rO$, and $m$ otherwise.

On the other hand, there are two special cases, namely $\ell=0$ and $\ell=k-1$, which may depend on the boundary conditions. Here the factors can be $m$ or $m-1$, but since $m=\Theta (m-1)$, we can simply replace this by a factor $\Theta.$

\invisible{
\begin{itemize}
    \item[$\rhd$] When $\ell=0$, the number of  edge choices between $\pi_0$ and $\pi_1$ is equal to $m$, unless $\pi_0\in \partial B_r^{\sss(G_n)}(\Ver_1)$, $n(\pi_0)<\pi_0$, and $\lb(\pi_1)=\rO$, where $n(\pi_0)$ is the unique neighbor of $\pi_0$ in $\bar{\bf t}_1$. In the latter case, an out-edge from $\pi_0$ is connected to $n(\pi_0)$, and the number of  edge choices is equal to $m-1$. 
    \item[$\rhd$] When $\ell=k-1$, the number of  edge choices between $\pi_{k-1}$ and $\pi_k$ is equal to $m-1$ when $\lb(\pi_{k-1})=\rY$ and $\lb(\pi_k)=\rO$. Furthermore, if $\pi_k\in \partial B_r^{\sss(G_n)}(\Ver_2)$, $\lb(\pi_k)=\rY$, and $\pi_k>n(\pi_k)$, the number of choices is also $m-1$, where $n(\pi_k)$ is the unique neighbor of $\pi_k$ in $\bar{\bf t}_2$, since in this case, an out-edge of $\pi_k$ is connected to $n(\pi_k)$. In all other cases, the number of edge choices is equal to $m$.
\end{itemize}
\RvdH{Since $m=\Theta (m-1)$ can we not remove the above discussion?}}


Recall from \eqref{def_NYO} that $N_{\srY\srO}=\sum_{i\in[k-1]}\indic{(\lb^\pi_i,\lb^\pi_{i+1})=(\rY,\rO)}$. \invisible{Following the argument above, the number of possible edge-labeled self-avoiding paths $\vpie$, given $\vec\pi$, is equal to
\begin{align*}
    m^{k-N_{\srY\srO}-h}(m-1)^{N_{\srY\srO}+h}=\Theta m^{k-N_{\srY\srO}}(m-1)^{N_{\srY\srO}},
\end{align*}
where
\begin{align*}
    h=\indic{\pi_0\in \partial B_r^{\sss(G_n)}(\Ver_1),n(\pi_0)<\pi_0,\lb(\pi_1)=\rO}+\indic{\pi_k\in \partial B_r^{\sss(G_n)}(\Ver_2),\lb(\pi_k)=\rY,\pi_k>n(\pi_k)}.
\end{align*}}
Then, as in \eqref{eq_hd_lower_form_prod_kappa-old}, by \eqref{eq_sum_no_con_part2}, \eqref{eq_sum_no_con_part3} and \eqref{eq_hd_prob_one_path-rep},
\begin{align*}
&\sum_{\vpie\in \TESA_{\nVer_1,\nVer_2}^{\sss e,k}}\prod_{s=2}^n\frac{(\alpha+p_s-1)_{p_s}(\beta_s+q_s-1)_{q_s}}{(\alpha+\beta_s+p_s+q_s-1)_{p_s+q_s}}\\
&\qquad=\sum_{\vec\pi\in \TSA_{\nVer_1,\nVer_2}^{\sss k}}\Theta m^{k-N_{\srY\srO}}(m-1)^{N_{\srY\srO}}\bc{\frac{m+\delta}{2m+\delta}}^{k-N_{\srO\srY}}\bc{\frac{m+1+\delta}{2m+\delta}}^{N_{\srO\srY}}\nn\\
    &\qquad\qquad\times\prod_{i\in[k]}\frac{1}{(\pi_{i-1}\wedge\pi_i)^{1-\chi}(\pi_{i-1}\vee\pi_i)^\chi }\nn\\
    &\qquad= \Theta n^{-k}\sum_{\vec\pi\in \TSA_{\nVer_1,\nVer_2}^{\sss k}}\frac{1}{(\frac{\pi_{0}\wedge\pi_1}{n})^{1-\chi}(\frac{\pi_{0}\vee\pi_1}{n})^\chi }\prod_{i=2}^k\kappa\big((\pi_{i-1}/n,\lb^\pi_{i-1}),(\pi_{i}/n,\lb^\pi_i\big).\nn
\end{align*}
Consequently, by \eqref{eq-change-label-add-theta1}, regardless of the choice of $\lb^\pi_0$, 
\begin{align*}
    &\sum_{\vpie\in \TESA_{\nVer_1,\nVer_2}^{\sss e,k}}\prod_{s=2}^n\frac{(\alpha+p_s-1)_{p_s}(\beta_s+q_s-1)_{q_s}}{(\alpha+\beta_s+p_s+q_s-1)_{p_s+q_s}}\\
    &\qquad= \Theta n^{-k}\sum_{\vec\pi\in \TSA_{\nVer_1,\nVer_2}^{\sss k}}\prod_{i=1}^k\kappa\big((\pi_{i-1}/n,\lb^\pi_{i-1}),(\pi_{i}/n,\lb^\pi_i)\big),
\end{align*}
as desired.
\end{proof}

\subsection{Proof of \Cref{prop-path-count-PAMc-LB_simple} subject to path lower bounds}\label{sec-given-com-dif-self-con-proof}


With \Cref{lem_prob_pie_cond,lem_hd_sum_restrict_path} in hand, the only remaining part needed to prove \Cref{prop-path-count-PAMc-LB_simple} \zhu{are analogous inequalities
to those in} \eqref{eq_hd_lower_trans_sum_int} and \eqref{eq-new-Tkappa-2}, now providing lower bounds on the rhs of \eqref{eq_hd_sum_restrict_path}. However,
due to several avoidance restrictions on $\TSA_{\nVer_1,\nVer_2}^{\sss k}$, 
there is no straightforward way to bound this summation from below by a comparable integral. In this section, we first state this lower bound, and then provide the proof of \Cref{prop-path-count-PAMc-LB_simple} subject to it.
\begin{lemma}[Lower bound on the  kernel products over self-avoiding paths]\label{lem_hd_sum_restrict_path-sum-gen}
Fix $m\geq 2$ and $\delta>0$, as well as $r,M$ and $\zeta$, and consider $\PAndmdel$. Then,  for any $\vep\in(0,1)$, there exists a $\zeta_0\in (0,1/10]$ such that
for any ${\zeta}\in (0,\zeta_0]$, there exists an $n_0\geq \zeta^{-2}$,  depending only on $m,\delta,M,r,\vep$ and $\zeta$, {such that the following holds:}

For any $n\geq n_0$ and any $\lb_0^\pi\in\cbc{\rO,\rY}$, with $1\leq k\leq 2\log_\nu n$,
 $c_{\zeta}$ in \Cref{pro_hd_spe_con} and $F_1,F_2$ such that
\begin{enumerate}
    \item $F_1\subseteq [{\zeta}n,n]\backslash \big(B_{r-1}^{\sss(G_n)}(\Ver_1)\cup B_r^{\sss(G_n)}(\Ver_2)\big)$ and 
$F_2\subseteq [{\zeta}n,n]\backslash \big(B_{r-1}^{\sss(G_n)}(\Ver_2)\cup B_r^{\sss(G_n)}(\Ver_1)\big)$,
\item $F_1\cap F_2=\varnothing$,
\end{enumerate}
on the event $\GNW$, the following inequality holds:
\zhu{\begin{align}\label{eq_hd_sum_restrict_path-sum-gen-1-2-change}
&\sum_{\substack{\nVer_1\in F_1,\\\nVer_2\in F_2}}\sum_{\vec\pi\in \TSA_{\nVer_1,\nVer_2}^{\sss k}}\prod_{i=1}^k\kappa\big((\pi_{i-1}/n,\lb^\pi_{i-1}),(\pi_{i}/n,\lb^\pi_i)\big)\\
    &\qquad\geq \Theta \abs{F_1}\abs{F_2}{\zeta}^2c_{\zeta}\nu^{k(1-\vep/4)}n^{k-1}.\nn
\end{align}}
\end{lemma}

\invisible{\RvdH{We choose $r$ first, and only then $M$, right? This is not clear in the quantors of the above proposition.}
\zhu{In the lemma, 
$M$ and 
$r$ are fixed numbers. We do not take the limit here.}}


Given \Cref{lem_hd_sum_restrict_path-sum-gen}, we are in the position to prove \Cref{prop-path-count-PAMc-LB_simple}:
\begin{proof}[Proof of \Cref{prop-path-count-PAMc-LB_simple} subject to \Cref{lem_hd_sum_restrict_path-sum-gen}]
By \Cref{lem_prob_pie_cond,lem_hd_sum_restrict_path}, for $k\leq 2\log_\nu n$, $n\geq \zeta^{-2}$ and any $\lb^\pi_0\in\cbc{\rO,\rY}$, on the good event $\GNW$,
\begin{align*}
&\sum_{\vpie\in \TESA_{\nVer_1,\nVer_2}^{\sss e,k}}\prob^r(\vpie\subseteq \PAndmdel)=\Theta n^{-k}\sum_{\vec\pi\in \TSA_{\nVer_1,\nVer_2}^{\sss k}}\prod_{i=1}^k\kappa\big((\pi_{i-1}/n,\lb^\pi_{i-1}),(\pi_{i}/n,\lb^\pi_i)\big).\nn
\end{align*}

Recall \zhu{that} $\NESA^{\sss e,k,r}$ is the set of all possible  $k$-step edge-labeled self-avoiding paths $\vpie$ with $\pi_0\in \partial B_r^{\sss(G_n)}(\Ver_1)$ and $\pi_k\in \partial B_r^{\sss(G_n)}(\Ver_2)$ in \zhu{$[n]$}, and we restrict to $\pi_i\not\in B_r^{\sss(G_n)}(\Ver_1)\cup B_r^{\sss(G_n)}(\Ver_2)$ for $i\in[k-1]$ and $\pi_i\geq {\zeta} n$ for all $0\leq i\leq k$. 
Then, 
\begin{align}\label{eq_hd_lowerbound_r}
    \expec^r[N_{n,r}(k)]&=\sum_{\vpie\in \NESA^{\sss e,k,r}}\prob^r(\vpie\subseteq \PAndmdel)\nn\\
    &=\sum_{\substack{\nVer_1\in \partial B_r^{\sss(G_n)}(\Ver_1)\cap [{\zeta}n,n],\\\nVer_2\in \partial B_r^{\sss(G_n)}(\Ver_2)\cap [{\zeta}n,n]}}\sum_{\vpie\in \TESA_{\nVer_1,\nVer_2}^{\sss e,k}}\prob^r(\vpie\subseteq \PAndmdel) \\
 &=\Theta n^{-k}\sum_{\substack{\nVer_1\in \partial B_r^{\sss(G_n)}(\Ver_1)\cap [{\zeta}n,n],\\\nVer_2\in \partial B_r^{\sss(G_n)}(\Ver_2)\cap [{\zeta}n,n]}}\sum_{\vec\pi\in \TSA_{\nVer_1,\nVer_2}^{\sss k}}\prod_{i=1}^k\kappa\big((\pi_{i-1}/n,\lb^\pi_{i-1}),(\pi_{i}/n,\lb^\pi_i)\big).\nn
\end{align}
\zhu{Note from \Cref{ass_Br} that, on the good event $\GNW$, 
$$\abs{\partial B_r^{\sss(G_n)}(\Ver_i)\cap [{\zeta}n,n]}\geq (3/2)^r \text{ for $i=1,2$.}$$
By \Cref{lem_hd_sum_restrict_path-sum-gen}, with $F_i=\partial B_r^{\sss(G_n)}(\Ver_i)\cap [{\zeta}n,n]$ $(i=1,2)$ and $n_0$ defined as in \Cref{lem_hd_sum_restrict_path-sum-gen}, for \ch{any} $n\geq n_0$, on the good event $\GNW$, we compute that}
\begin{align*}
    \expec^r[N_{n,r}(k)]\geq\Theta \bc{3/2}^{2r}{\zeta}^2c_{\zeta}\nu^{k(1-\vep/4)}n^{-1}.
\end{align*}
For $k=k_\nu^\star=\rou{(1+\vep)\log_\nu n}$ and $\vep,\zeta\in(0,1)$, as $n\to\infty$,
\begin{align}\label{eq-certify-sharp-prop-nu_vep}
    \expec^r[N_{n,r}(k_\nu^\star)]\geq\Theta \bc{3/2}^{2r}{\zeta}^2c_{\zeta}n^{\vep/2}  \to \infty,
\end{align}
as desired.
\end{proof}

\subsection{Modification of $\kappa$ and lower bound without avoidance restrictions}\label{sub-sec-modify-kappa-lower-bound}
In this section, we state and prove a weaker version of \Cref{lem_hd_sum_restrict_path-sum-gen}, which provides lower bounds on the kernel products without avoidance constraints. This result plays a crucial role in the proof of  \Cref{lem_hd_sum_restrict_path-sum-gen}.

Recall from \eqref{kappa-PAM-fin-def-form-a-rep} that 
\eqn{
	\kappa\big((x,s), (y,t)\big)=\frac{c_{st}(\indic{x>y, t=\srO}+\indic{x< y, t=\srY})}{(x\vee y)^{\chi}(x\wedge y)^{1-\chi}}.\nn
	}
To obtain a straightforward lower bound on the rhs of \eqref{eq_hd_sum_restrict_path} without avoidance restrictions, we add an extra term to $\kappa$ as
\eqn{
	\nka\big((x,s), (y,t)\big)=\frac{c_{st}(\indic{x>y, t=\srO}+\indic{x< y, t=\srY})}{(x\vee y)^{\chi}(x\wedge y)^{1-\chi}}+\frac{c_{\srO\srY}\indic{x=y, t=\srO}}{x}.\label{kappa-PAM-fin-def-form-a-rep-new_term}
	}
We add this term, since $\kappa((x,s),(x,t))=0$ and we will rely on pointwise bounds (see \eqref{eq-decrease-kappa-total-gen} below). The factor $c_{\srO\srY}$ appears, since this is the largest coefficient, while we only add it for $t=\rO$ since we define $\lb^{\pi}_i=\rO$ when $\pi_i=\pi_{i-1}$.

Then, for each $x,y\in [0,1]$,
\begin{align}\label{eq-change-label-add-theta1-new}
    \nka\big((x,\rO),(y,q(y-x))\big)=\Theta \nka\big((x,\rY),(y,q(y-x))\big)=\Theta\frac{1}{(x\vee y)^{\chi}(x\wedge y)^{1-\chi}}.
\end{align}
Hence, like $\kappa$, the value of $\lb^\pi_0$ only affects the value of $\nka\big((x,\lb^\pi_0),(y,q(y-x))\big)$ by a $\Theta $ factor.
\begin{rem}\label{re-kappa-old-new}
We note that $\nka\big((x,s), (y,t)\big)\neq \kappa\big((x,s), (y,t)\big)$ only if $x=y$.
Hence,
$
    \nka\big((\pi_{i-1}/n,\lb^\pi_{i-1}),(\pi_{i}/n,\lb^\pi_i)\big)=\kappa\big((\pi_{i-1}/n,\lb^\pi_{i-1}),(\pi_{i}/n,\lb^\pi_i)\big)
$
for $\vec\pi\in \NA_{\nVer_1,\nVer_2}^{\sss k}$.
\end{rem}
Let $\NP_{u,v}^{\sss k}=\cbc{(\pi_0,\ldots,\pi_k)\in [n]^{k+1}:~\pi_0=u,\pi_k=v,\pi_i\geq \zeta n~\forall i\in[k-1]}$ be the set of 
$k$-step {paths}{} between $u,v$ without any avoidance constraint. 
With the new integral kernel $\nka$, we prove the following weaker version of  \Cref{lem_hd_sum_restrict_path-sum-gen}, where self-loops in the paths are allowed:
\begin{lemma}[Lower bound on the  kernel products without avoidance constraints]\label{lem_hd_sum_restrict_path-sum-gen-nka}
\zhu{With} the notation in \Cref{lem_hd_sum_restrict_path-sum-gen}, for $n\geq \zeta^{-1}$, on the good event $\GNW$,
\begin{align}\label{eq_prod_kappa_need_in_var-F12}
    &\sum_{\substack{\nVer_1\in F_1,\\\nVer_2\in F_2}}\sum_{\vec\pi\in \NP_{\nVer_1,\nVer_2}^{\sss k}}\prod_{i=1}^k\nka\big((\pi_{i-1}/n,\lb^\pi_{i-1}),(\pi_{i}/n,\lb^\pi_i)\big)\\
    &\qquad\geq \Theta \abs{F_1}\abs{F_2}{\zeta}^2 c_{\zeta}\nu^{k(1-\vep/4)}n^{k-1}.\nn
\end{align}
\end{lemma}

\begin{proof}
Recall that
\begin{align*}
    q(x)=\begin{cases}
        \rO,\quad&\text{if }x\leq 0;\\
        \rY,\quad&\text{if }x>0;
    \end{cases}\quad\text{and}\quad \lb^\pi_i=q(\pi_i-\pi_{i-1}),
\end{align*}
from \eqref{def_label_direction}. Given an arbitrary path $\vec\pi=(\pi_i)_{0\leq i\leq k}\in[\zeta n,n]^{k+1}$, let
\begin{itemize}
    \item[$\rhd$] $x_j\in [\pi_j/n,(\pi_j+1)/n)$ for $0\leq j\leq k$;
    \item[$\rhd$] $\rmq_i=q(x_i-x_{i-1})$ for $i\in[k]$;
    \item[$\rhd$] $\rmq_0=\lb^\pi_0$ be an arbitrary label in $\cbc{\rO,\rY}$.
\end{itemize}
Then, $\lb^\pi_i=\rmq_i$ for each $i\in [k]$ such that $\pi_{i-1}\neq \pi_{i}$. Further, if $\lb^\pi_i\neq \rmq_i$, then $i\neq 0$ and $\pi_{i-1}= \pi_{i}$. Hence, $\lb^\pi_i=\rO$ and   $\rmq_i=\rY$.

With the definition of $\nka$ in \eqref{kappa-PAM-fin-def-form-a-rep-new_term}, we claim that\zhu{, for $x_j\in [\pi_j/n,(\pi_j+1)/n)$ \ch{with} $0\leq j\leq k$,}
\begin{align}\label{eq-decrease-kappa-total-gen}
    \nka\big((\pi_{i-1}/n,\lb^\pi_{i-1}),(\pi_{i}/n,\lb^\pi_{i})\big)\geq &\nka\big((x_{i-1},\rmq_{i-1}),(x_i,\rmq_i)\big),
\end{align}
while \eqref{eq-decrease-kappa-total-gen} does not \zhu{hold} for $\kappa$ when $\pi_{i-1}=\pi_i$ and $x_{i-1}\neq x_i$. This is why we add this extra term in $\nka$.

We first note that, as $x_i\in [\pi_i/n,(\pi_i+1)/n)$, for $0\leq i\leq k$,
\begin{align}\label{eq-decrease-kappa-no-c}
    \frac{1}{(\frac{\pi_{i-1}}{n}\vee \frac{\pi_i}{n})^{\chi}(\frac{\pi_{i-1}}{n}\wedge \frac{\pi_i}{n})^{1-\chi}}\geq \frac{1}{(x_{i-1}\vee x_i)^{\chi}(x_{i-1}\wedge x_i)^{1-\chi}}.
\end{align}
Then, \eqref{eq-decrease-kappa-total-gen} follows by the following observations:
\begin{itemize}
    \item[$\rhd$] If $\lb^\pi_{i-1}=\rmq_{i-1}$ and $\lb^\pi_i=\rmq_i$, when $\pi_{i-1}\neq \pi_i$, then \eqref{eq-decrease-kappa-total-gen} follows directly from \eqref{kappa-PAM-fin-def-form-a-rep-new_term} and \eqref{eq-decrease-kappa-no-c}; when $\pi_{i-1}= \pi_i$, then \eqref{eq-decrease-kappa-total-gen} follows from \eqref{kappa-PAM-fin-def-form-a-rep-new_term}, \eqref{eq-decrease-kappa-no-c} and {the fact that}{} $c_{\srO\srY}=\max_{s,t\in\cbc{\srO,\srY}}c_{st}$ by \eqref{cst-def-PAM-rep}.

\item[$\rhd$] If $\lb^\pi_{i-1}\neq\rmq_{i-1}$, then $\lb^\pi_{i-1}=\rO$ and $\rmq_{i-1}=\rY$. As $c_{\srO s}\geq c_{\srY t}$ for any $s,t\in\cbc{\rO,\rY}$, we conclude from \eqref{kappa-PAM-fin-def-form-a-rep-new_term} and \eqref{eq-decrease-kappa-no-c} that \eqref{eq-decrease-kappa-total-gen} holds.
\item[$\rhd$] If $\lb^\pi_{i-1}=\rmq_{i-1}$ and $\lb^\pi_i\neq\rmq_i$, then  $\pi_{i-1}=\pi_i$. Hence, by \eqref{kappa-PAM-fin-def-form-a-rep-new_term},
\begin{align*}
    \nka\big((\pi_{i-1}/n,\lb^\pi_{i-1}),(\pi_{i}/n,\lb^\pi_{i})\big)=\frac{c_{\srO\srY}}{\frac{\pi_i}{n}}\geq &\nka\big((x_{i-1},\rmq_{i-1}),(x_i,\rmq_i)\big).
\end{align*}
\end{itemize}
Consequently, \eqref{eq-decrease-kappa-total-gen} holds. 
{Therefore,}{} for any $i\in[2,k]$,
\begin{align*}
    \nka\big((\pi_{i-1}/n,\lb^\pi_{i-1}),(\pi_{i}/n,\lb^\pi_{i})\big)\geq &n\int_{\pi_{i}/n}^{((\pi_{i}+1)/n)\wedge 1}\nka\big((x_{i-1},\rmq_{i-1}),(x_i,\rmq_i)\big)\dint x_i,
\end{align*}
and
\begin{align*}
    \nka\big((\pi_0/n,\lb^\pi_0),(\pi_1/n,\lb^\pi_1)\big)\geq &n^2\int_{\pi_0/n}^{((\pi_0+1)/n)\wedge 1}\int_{\pi_1/n}^{((\pi_1+1)/n)\wedge 1}\nka\big((x_0,\rmq_0),(x_1,\rmq_1)\big)\dint x_0 \dint x_1.
\end{align*}
Therefore, for $n\geq \zeta^{-1}$, with $\nVer_1,\nVer_2\geq{\zeta} n$, 
\begin{align*}
&\sum_{\vec\pi\in \NP_{\nVer_1,\nVer_2}^{\sss k}}\prod_{i=1}^k\nka\big((\pi_{i-1}/n,\lb^\pi_{i-1}),(\pi_{i}/n,\lb^\pi_i)\big)\nn\\
    &\geq  n^{k+1}\int_{\lceil n\zeta\rceil/n}^1\cdots \int_{\lceil n\zeta\rceil/n}^1  \indic{\lfloor nx_0\rfloor=\nVer_1}\indic{\lfloor nx_k\rfloor=\nVer_2}\prod_{i=1}^k\nka\big((x_{i-1},\rmq_{i-1}),(x_{i},\rmq_i)\big)\dint x_0\cdots \dint x_k\nn\\
    &\geq  n^{k+1}\int_{2{\zeta}}^1\cdots \int_{2{\zeta}}^1 \indic{\lfloor nx_0\rfloor=\nVer_1}\indic{\lfloor nx_k\rfloor=\nVer_2}\prod_{i=1}^k\nka\big((x_{i-1},\rmq_{i-1}),(x_{i},\rmq_i)\big)\dint x_0\cdots \dint x_k,
\end{align*}
By \eqref{eq-change-label-add-theta1-new}, the value of $\lb^\pi_0$ \ch{only} affects $\nka\big((x_0,\rmq_0),(x_1,\rmq_1)\big)$ by a $\Theta $ factor. Hence,
\begin{flalign}\label{eq-lower-bound-sum-np-kernel}
\sum_{\vec\pi\in \NP_{\nVer_1,\nVer_2}^{\sss k}}\prod_{i=1}^k\nka\big((\pi_{i-1}/n,\lb^\pi_{i-1}),(\pi_{i}/n,\lb^\pi_i)\big)
    \geq \Theta n^{k+1}\sum_{\rmq_0\in\cbc{\srO,\srY}}\int_{2{\zeta}}^1\cdots \int_{2{\zeta}}^1 \indic{\lfloor nx_0\rfloor=\nVer_1}\nn\\
    \times\indic{\lfloor nx_k\rfloor=\nVer_2}\prod_{i=1}^k\nka\big((x_{i-1},\rmq_{i-1}),(x_{i},\rmq_i)\big)\dint x_0\cdots \dint x_k.
\end{flalign}

Let $\kappa_{{\zeta}}((x,s),(y,t))=\indic{x,y\in [2{\zeta},1]} \kappa\big((x,s),(y,t)\big)$.
Then,  $\nka\geq \kappa_{\zeta}$. Furthermore, with \eqref{eq-new-Tkappa} for $\bfT_{\kappa_{\zeta}}$, we conclude from \eqref{eq-lower-bound-sum-np-kernel} that
\begin{align}
\label{eq_hd_r-prod-trans-int-new0}
&\sum_{\vec\pi\in \NP_{\nVer_1,\nVer_2}^{\sss k}}\prod_{i=1}^k\nka\big((\pi_{i-1}/n,\lb^\pi_{i-1}),(\pi_{i}/n,\lb^\pi_i)\big)\nn\\
&\qquad\geq \Theta n^{k+1}\langle \indic{\lfloor nx_0\rfloor=\nVer_1}, \bfT_{\kappa_{\zeta}}^k \indic{\lfloor nx_k\rfloor=\nVer_2}\rangle.
\end{align}

\zhu{\ch{On} the other hand, recall that $\NP_{u,v}^{\sss k}$ is the set of 
$k$-step paths between $u,v$ without any avoidance constraint. Then
\begin{align}
    &\sum_{\vec\pi\in \NP_{\nVer_1,\nVer_2}^{\sss k}}\prod_{i=1}^k\nka\big((\pi_{i-1}/n,\lb^\pi_{i-1}),(\pi_{i}/n,\lb^\pi_i)\big)\nn\\
    &\qquad=\sum_{\substack{\pi_0=\nVer_1,\pi_k=\nVer_2\\\pi_j\in [\eta n,n],j\in [k-1]}}\prod_{i=1}^k\nka\big((\pi_{i-1}/n,\lb^\pi_{i-1}),(\pi_{i}/n,\lb^\pi_i)\big).\label{eq_hd_r-prod-trans-int-new1}
\end{align}

Hence, changing the value of $\nVer_1$ and $\nVer_2$ will only affect the first two and the last term in the production $\prod_{i=1}^k\nka\big((\pi_{i-1}/n,\lb^\pi_{i-1}),(\pi_{i}/n,\lb^\pi_i)\big)$. 

Given $\nVer_1,\nVer_2,\nVer_1',\nVer_2'\in [\zeta n,n]$, let $\vec\pi'$ be the path that \ch{replaces} the starting and ending points of $\vec\pi$ from $(\nVer_1,\nVer_2)$ to $(\nVer_1',\nVer_2')$. Then $\lb^{\pi}_i=\lb^{\pi'}_{i}$ for $i\in[2,n-1]$.
Consequently, \eqref{kappa-PAM-fin-def-form-a-rep-new_term} yields that, for $i\in\cbc{0,n-1}$,
\begin{align}\label{eq_hd_r-prod-trans-int-new2}
    &\nka\big((\pi_i/n,\lb^{\pi}_i),(\pi_{i+1}/n,\lb^\pi_{i+1})\big)\geq \Theta\zeta\nka\big((\pi_i'/n,\lb^{\pi'}_i),(\pi_{i+1}'/n,\lb^{\pi_{i+1}'})\big),
\end{align}
and
\begin{align}\label{eq_hd_r-prod-trans-int-new3}
    &\nka\big((\pi_1/n,\lb^{\pi}_1),(\pi_2/n,\lb^\pi_2)\big)\geq \Theta\nka\big((\pi_1'/n,\lb^{\pi'}_1),(\pi_2'/n,\lb^{\pi'}_2)\big),
\end{align}
while, \ch{for }$i\in [2,n-2]$,
\begin{align}\label{eq_hd_r-prod-trans-int-new4}
    &\nka\big((\pi_i/n,\lb^{\pi}_i),(\pi_{i+1}/n,\lb^\pi_{i+1})\big)=\nka\big((\pi_i'/n,\lb^{\pi'}_i),(\pi_{i+1}'/n,\lb^{\pi'}_{i+1})\big).
\end{align}
As a consequence, for $f_1(x)=\sum_{\nVer_1'=\ceil{\zeta n}}^n\indic{\lfloor nx\rfloor=\nVer_1'}$ and $f_2(x)=\sum_{\nVer_2'=\ceil{\zeta n}}^n\indic{\lfloor nx\rfloor=\nVer_2'}$, the combination of \eqref{eq_hd_r-prod-trans-int-new0}-\eqref{eq_hd_r-prod-trans-int-new4} yields that
\begin{align*}
    &\sum_{\vec\pi\in \NP_{\nVer_1,\nVer_2}^{\sss k}}\prod_{i=1}^k\nka\big((\pi_{i-1}/n,\lb^\pi_{i-1}),(\pi_{i}/n,\lb^\pi_i)\big)\\
    \geq& \frac{\Theta\zeta^2}{n^2} \sum_{\nVer_1',\nVer_2'=\ceil{\zeta n}}^n\sum_{\vec\pi\in \NP_{\nVer_1',\nVer_2'}^{\sss k}}\prod_{i=1}^k\nka\big((\pi_{i-1}/n,\lb^\pi_{i-1}),(\pi_{i}/n,\lb^\pi_i)\big)\\
    \geq&\Theta\zeta^2 n^{k-1}\left\langle \sum_{\nVer_1'=\ceil{\zeta n}}^n\indic{\lfloor nx_0\rfloor=\nVer_1'}, \bfT_{\kappa_{\zeta}}^k \sum_{\nVer_2'=\ceil{\zeta n}}^n\indic{\lfloor nx_k\rfloor=\nVer_2'}\right\rangle =\Theta\zeta^2 n^{k-1}\langle f_1, \bfT_{\kappa_{\zeta}}^k f_2\rangle
\end{align*}
Recall 
	$\kappa^\circ_{\zeta}$ from \eqref{hatkappa-vepn-PA}.
Note that $f_1(x),f_2(x)\geq \indic{x\in[3\zeta,1-\zeta]}$ for sufficiently large $n$. Hence,
\begin{align*}
    \langle f_1, \bfT_{\kappa_{\zeta}}^k f_2\rangle
    \geq\langle \indic{x\in[3\zeta,1-\zeta]}, \bfT_{\kappa^\circ_{\zeta}}^k \indic{x\in[3\zeta,1-\zeta]}\rangle=\langle {\bf 1}, \bfT_{\kappa^\circ_{\zeta}}^k {\bf 1}\rangle.
\end{align*}
By \Cref{pro_hd_spe_con}, $\langle {\bf 1}, \bfT_{\kappa_{{\zeta}}^\circ}^{k} {\bf 1}\rangle\geq c_{\zeta}\uhvep^{k}$ for some $c_{\zeta}\in(0,1]$ and $\lim_{\zeta\searrow 0}\uhvep=\nu>1$. Hence, for any $\vep\in(0,1)$, there exists a $\zeta_0\in(0,1/10]$ such that, for any ${\zeta}\in(0,{\zeta_0}]$, $\uhvep\geq \nu^{1-\vep/4}$. For this $\zeta$, 
\begin{align*}
    \langle {\bf 1}, \bfT_{\kappa_{{\zeta}}^\circ}^{k} {\bf 1}\rangle\geq c_{\zeta}\uhvep^{k}\geq c_{\zeta}\nu^{(k)(1-\vep/4)} = \Theta  c_{\zeta}\nu^{k(1-\vep/4)}.
\end{align*}
Hence, we conclude that
\begin{align*}
    \sum_{\vec\pi\in \NP_{\nVer_1,\nVer_2}^{\sss k}}\prod_{i=1}^k\nka\big((\pi_{i-1}/n,\lb^\pi_{i-1}),(\pi_{i}/n,\lb^\pi_i)\big)\geq \Theta\zeta^2 n^{k-1}c_{\zeta}\nu^{k(1-\vep/4)}.
\end{align*}
Summing $\nVer_1$ over $F_1$ and $\nVer_2$ over $F_2$ leads to
\begin{align*}
\sum_{\substack{\nVer_1\in F_1,\\\nVer_2\in F_2}}\sum_{\vec\pi\in \NP_{\nVer_1,\nVer_2}^{\sss k}}\prod_{i=1}^k\nka\big((\pi_{i-1}/n,\lb^\pi_{i-1}),(\pi_{i}/n,\lb^\pi_i)\big)\geq \Theta\abs{F_1}\abs{F_2}\zeta^2 n^{k-1}c_{\zeta}\nu^{k(1-\vep/4)},
\end{align*}
as desired.
}

\end{proof}
Using \Cref{lem_hd_sum_restrict_path-sum-gen-nka},  we can obtain the following stronger version of \Cref{lem_hd_sum_restrict_path-sum-gen}:
\begin{lemma}[Comparison of lower bounds on the  kernel products]\label{lem_hd_sum_restrict_path-sum-gen-nka-compare}
With the notation in \Cref{lem_hd_sum_restrict_path-sum-gen}, on the good event $\GNW$,
\begin{align}\label{eq_hd_sum_restrict_path-sum-gen-1-2-new-main}
    &\sum_{\substack{\nVer_1\in F_1,\\\nVer_2\in F_2}}\sum_{\vec\pi\in \TSA_{\nVer_1,\nVer_2}^{\sss k}}\prod_{i=1}^k\kappa\big((\pi_{i-1}/n,\lb^\pi_{i-1}),(\pi_{i}/n,\lb^\pi_i)\big)\\
        &\qquad\geq\frac{1}{4}\sum_{\substack{\nVer_1\in F_1,\\\nVer_2\in F_2}}\sum_{\vec\pi\in \NP_{\nVer_1,\nVer_2}^{\sss k}}\prod_{i=1}^k\nka\big((\pi_{i-1}/n,\lb^\pi_{i-1}),(\pi_{i}/n,\lb^\pi_i)\big).\nn
    \end{align}    
\end{lemma}
We emphasize \Cref{lem_hd_sum_restrict_path-sum-gen-nka-compare} here since \eqref{eq_hd_sum_restrict_path-sum-gen-1-2-new-main} is also used in \Cref{sec-second-moment-proof-thm-log-PA-delta>0} to bound the variance of 
$N_{n,r}(k_\nu^\star)$ from above (see the proof of \Cref{lem-u=1}). 
\begin{proof}[\shortversion{Sketch of proof}\longversion{Proof of \Cref{lem_hd_sum_restrict_path-sum-gen-nka-compare}}]
\longversion{We defer the proof to \Cref{app-E}.}
\shortversion{The proof is given in \cite[Appendix E]{HofZhu25long}. Here we provide some intuition. 
The main point is that our sums over paths have several avoidance constraints. The analogous lower bound in \Cref{lem_hd_sum_restrict_path-sum-gen-nka} does not have these constraints. We thus first ignore the constraints, and, using \Cref{lem_hd_sum_restrict_path-sum-gen-nka} and inclusion-exclusion, show that the contributions where the constraints are violated is not large. 
}
\end{proof}
\begin{proof}[Proof of \Cref{lem_hd_sum_restrict_path-sum-gen}]
\Cref{lem_hd_sum_restrict_path-sum-gen} follows directly from \Cref{lem_hd_sum_restrict_path-sum-gen-nka,lem_hd_sum_restrict_path-sum-gen-nka-compare}.
\end{proof}





\section{Upper bounding the conditional variance}
\label{sec-second-moment-proof-thm-log-PA-delta>0}
In this section, we prove \Cref{prop-path-count-PAMc-Var-UB}, that is, we show that for sufficiently large $n$, on the good event $\GNW$, with $k_\nu^{\star}=\rou{(1+\vep)\log_\nu n}$, 
\eqn{
	\label{Nnk-Var-PAMc}
	\Var^r(N_{n,r}(k_\nu^{\star}))\leq \Theta \bc{3/2}^{-r} {\zeta}^{-4} c_{\zeta}^{-1} \expec^r[N_{n,r}(k_\nu^\star)]^2.
	}

\subsection{Strategy of proof: Division into $3$ cases, and analysis of Case \ref{it_var_1}}
By the definition of the variance, we can write $\Var^r(N_{n,r}(k_\nu^{\star}))$ as 
{\small
\begin{align}\label{variance_decompose}
    \Var^r&(N_{n,r}(k_\nu^{\star}))=\sum_{\vpie,\vrhoe\in \NESA^{\sss e,k_\nu^{\star},r}} \Big(\prob^r(\vpie, \vrhoe\subseteq \PAndmdel)-\prob^r(\vpie\subseteq \PAndmdel)\prob^r(\vrhoe\subseteq \PAndmdel)\Big).
\end{align}
}
The proof of \eqref{Nnk-Var-PAMc} is based on \eqref{variance_decompose} through the following three partitions of the relation between $\vpie$ and $\vrhoe$:
\begin{enumerate}[label= (\Alph*)]
    \item $\vpie=\vrhoe$;\label{it_var_1}
    \item  $\pi_i\neq\rho_j$ for any $0\leq i,j\leq k$, which we denote as $\vec\pi\cap\vec\rho=\emptyset$;\label{it_var_2}
    \item $\vec\pi\cap\vec\rho\neq\emptyset$ and $\vpie\neq\vrhoe$.\label{it_var_3}
\end{enumerate}

Thanks to \Cref{prop-path-count-PAMc-LB_simple}, Case \ref{it_var_1} is the simplest one among the three cases. Indeed, on the good event $\GNW$, by \Cref{prop-path-count-PAMc-LB_simple}, for sufficiently large $n$, $\expec^r[N_{n,r}(k_\nu^{\star})]\geq \Theta (3/2)^{2r}{\zeta}^2c_{\zeta}n^{\vep/2}$. Then, uniformly in the choices of $\Ver_1$ and $\Ver_2$,
\begin{align}\label{eq_con-case-A}
    \sum_{\substack{\vpie\in \NESA^{\sss e,k_\nu^\star,r},\\\vpie=\vrhoe}}\prob^r(\vpie, \vrhoe\subseteq \PAndmdel)&={\expec^r[N_{n,r}(k_\nu^{\star})]}=o(1)\expec^r[N_{n,r}(k_\nu^{\star})]^2.
\end{align}
Thus, the contribution of Case \ref{it_var_1} to \eqref{variance_decompose} is no more than $o(1)\expec^r[N_{n,r}(k_\nu^{\star})]^2$.
\smallskip

We next bound the contribution of Cases \ref{it_var_2}
and \ref{it_var_3} to \eqref{variance_decompose} in their order of appearance. For Case \ref{it_var_2}, we prove the following lemma in \Cref{sec-con-case-disjoint}:
\begin{lemma}[Contribution from disjoint $\vec\pi$ and $\vec\rho$]\label{lem-con-case-disjoint}
For any $k\leq 2\log_\nu n$, on the good event $\GNW$, {and uniformly in the choices of $\Ver_1$ and $\Ver_2$,}{}
\begin{align*}
\sum_{\substack{\vpie,\vrhoe\in \NESA^{\sss e,k,r},\\\vec\pi\cap\vec\rho=\emptyset}} &\Big(\prob^r(\vpie, \vrhoe\subseteq \PAndmdel)-\prob^r(\vpie\subseteq \PAndmdel)\prob^r(\vrhoe\subseteq \PAndmdel)\Big)=o(1) \expec^r[N_{n,r}(k)]^2.   
\end{align*}
\end{lemma}
For Case \ref{it_var_3}, we prove the following lemma in \Cref{sec-path-decompose-u,sec-con-case-intersect}:
\begin{lemma}[Contribution from distinct but intersecting $\pi$ and $\rho$]\label{lem-con-case-intersect}
For any $\vep\in (0,1)$, sufficiently large $n$ and $k\leq (1+\vep)\log_\nu n$, on the good event $\GNW$, uniformly in the choices of $\Ver_1$ and $\Ver_2$,
    \begin{align*}
    &\sum_{\vpie\in \NESA^{\sss e,k,r}} \sum_{\substack{\vrhoe\in \NESA^{\sss e,k,r},\\\vec\pi\cap\vec\rho\neq\emptyset,\\\vpie\neq\vrhoe}}\prob^r(\vpie, \vrhoe\subseteq \PAndmdel) \leq \Theta \bc{3/2}^{-r} {\zeta}^{-4} c_{\zeta}^{-1} \expec^r[N_{n,r}(k)]^2.
\end{align*}
\end{lemma}
Given \Cref{lem-con-case-disjoint,lem-con-case-intersect}, \Cref{prop-path-count-PAMc-Var-UB} follows {straightforwardly}:
\begin{proof}[Proof of \Cref{prop-path-count-PAMc-Var-UB} subject to \Cref{lem-con-case-disjoint,lem-con-case-intersect}]
    We can bound the rhs of \eqref{variance_decompose} as
\begin{align}\label{eq-decompose-var-upperbound}
&\sum_{\vpie,\vrhoe\in \NESA^{\sss e,k_\nu^{\star},r}} \Big(\prob^r(\vpie, \vrhoe\subseteq \PAndmdel)-\prob^r(\vpie\subseteq \PAndmdel)\prob^r(\vrhoe\subseteq \PAndmdel)\Big)\\
&\qquad\leq \sum_{\substack{\vpie\in \NESA^{\sss e,k_\nu^\star,r},\\\vpie=\vrhoe}}\prob^r(\vpie, \vrhoe\subseteq \PAndmdel)+\sum_{\vpie\in \NESA^{\sss e,k,r}} \sum_{\substack{\vrhoe\in \NESA^{\sss e,k,r},\\\vec\pi\cap\vec\rho\neq\emptyset,\\\vpie\neq\vrhoe}}\prob^r(\vpie, \vrhoe\subseteq \PAndmdel)\nn\\
&\qquad\quad+\sum_{\substack{\vpie,\vrhoe\in \NESA^{\sss e,k_\nu^{\star},r},\\\vec\pi\cap\vec\rho=\emptyset}} \Big(\prob^r(\vpie, \vrhoe\subseteq \PAndmdel)-\prob^r(\vpie\subseteq \PAndmdel)\prob^r(\vrhoe\subseteq \PAndmdel)\Big).\nn
\end{align}
Hence, we conclude from \eqref{variance_decompose} to \eqref{eq-decompose-var-upperbound} and \Cref{lem-con-case-disjoint,lem-con-case-intersect} that, for sufficiently large $n$ that only depends on $m,\delta,M,r,\vep$ and $\zeta$, on the good event $\GNW$,
\begin{align*}
    \Var^r(N_{n,r}(k_\nu^{\star})) \leq\Theta \bc{3/2}^{-r} {\zeta}^{-4} c_{\zeta}^{-1} \expec^r[N_{n,r}(k_\nu^\star)]^2,
\end{align*}
as desired.
\end{proof}

\subsection{Analysis of Case \ref{it_var_2}: Falling factorial form of path probabilities}\label{sec-con-case-disjoint}
In Case \ref{it_var_2}, $\vec\pi\cap\vec\rho=\emptyset$.  Since the dependence between two edges with distinct endpoints in $\PAndmdel$ is very weak, we believe that in this case, the probability that both $\vpie$ and $\vrhoe$ occur in $\PAndmdel$ is close to the product of the probabilities of their individual occurrences. To prove \Cref{lem-con-case-disjoint}, we first formalize this intuition by proving the following weak dependence estimate:
\begin{lemma}[Decomposition of probability  for disjoint $\vec\pi$ and $\vec\rho$]\label{lem-var-disjoint-paths}
For any $k\leq 2\log_\nu n$, on the good event $\GNW$, with  $\vpie,\vrhoe\in \NESA^{\sss e,k,r}$ and $\vec\pi\cap\vec\rho=\emptyset$,
\zhu{
 \begin{align}
 \label{con-var-disjoint-paths}
    \prob^r(\vpie,\vrhoe\subseteq \PAndmdel)\leq \Theta^{\frac{1}{\log n}+{\zeta}^{-2} k^2 n^{-1}}\prob^r(\vpie\subseteq \PAndmdel)\prob^r(\vrhoe\subseteq \PAndmdel).
\end{align}   }
\end{lemma}


\begin{proof}
In \eqref{eq_E-in-PAM-cond}, we \ch{have} proved that
\begin{align}\label{eq_E-in-PAM-cond-rep}
    \PP^r(E\subseteq \PAndmdel)&=\Theta^{\frac{1}{\log n}}\prod_{s=2}^n
	\frac{(\alpha+\psE-1)_{\psE}(\beta_s+\qsE-1)_{\qsE}}{(\alpha+\beta_s+\psE+\qsE-1)_{\psE+\qsE}}.
\end{align}
We now apply \eqref{eq_E-in-PAM-cond-rep} to different choices of $E$:
\begin{itemize}
    \item[$\rhd$] With $\vec\pi\cap\vec\rho=\emptyset$ and $E=\vpie\cup\vrhoe$, \eqref{eq_E-in-PAM-cond-rep} implies that
\eqan{\label{eq_hd_prod_simp_pq_twopath}
	\prob^r(\vpie,\vrhoe\subseteq \PAndmdel) =\Theta^{\frac{1}{\log n}}\prod_{s=2}^n
	\frac{(\alpha+p_s-1)_{p_s}(\beta_s+q_s-1)_{q_s}}{(\alpha+\beta_s+p_s+q_s-1)_{p_s+q_s}},
	} 
where \ch{now}
\eqan{p_s=p_s^{\sss \pi\cup\rho}=\sum_{e\in \vpie\cup\vrhoe}\indic{s=\ushort{e}}=\pspi+\psrho\quad\text{and}\quad q_s=q_s^{\sss \pi\cup\rho}=\sum_{e\in \vpie\cup\vrhoe}\indic{s\in (\ushort{e},\bar{e})}=\qspi+\qsrho,\nn
}
and where $\pspi,\qspi,\psrho,\qsrho$ are defined as in \eqref{def_hd_pq}.
\smallskip

\item[$\rhd$] With $E=\vpie$ and ${E=}\vrhoe$, \eqref{eq_E-in-PAM-cond-rep} implies that
\begin{align}\label{eq_hd_prod_simp_pq_twopath-2}
        &\prob^r(\vpie\subseteq \PAndmdel)\prob^r(\vrhoe\subseteq \PAndmdel)\\
        &\qquad=\Theta^{\frac{1}{\log n}}\prod_{s=2}^n\frac{(\alpha+\pspi-1)_{\pspi}(\beta_s+\qspi-1)_{\qspi}}{(\alpha+\beta_s+\pspi+\qspi-1)_{\pspi+\qspi}}\frac{(\alpha+\psrho-1)_{\psrho}(\beta_s+\qsrho-1)_{\qsrho}}{(\alpha+\beta_s+\psrho+\qsrho-1)_{\psrho+\qsrho}}.\nn
\end{align}
\end{itemize}

{To prove \eqref{con-var-disjoint-paths}, and given \eqref{eq_hd_prod_simp_pq_twopath} and \eqref{eq_hd_prod_simp_pq_twopath-2}, we need to compare the two products on their rhs's.}{}
In what follows, we {will prove}{} that, for any $s\in[2,n]$,
\begin{align}\label{eq_bound_falling_decomposition}
    &\frac{(\alpha+p_s-1)_{p_s}(\beta_s+q_s-1)_{q_s}}{(\alpha+\beta_s+p_s+q_s-1)_{p_s+q_s}}\\
    &\qquad\leq \Theta^{{\zeta}^{-2} k^2 n^{-2}} \frac{(\alpha+\pspi-1)_{\pspi}(\beta_s+\qspi-1)_{\qspi}}{(\alpha+\beta_s+\pspi+\qspi-1)_{\pspi+\qspi}}\frac{(\alpha+\psrho-1)_{\psrho}(\beta_s+\qsrho-1)_{\qsrho}}{(\alpha+\beta_s+\psrho+\qsrho-1)_{\psrho+\qsrho}}.\nn
  \end{align}
{By \eqref{eq_hd_prod_simp_pq_twopath} and \eqref{eq_hd_prod_simp_pq_twopath-2}, this obviously proves \eqref{con-var-disjoint-paths}. We are left to proving \eqref{eq_bound_falling_decomposition}.}{}

Since $\vpie,\vrhoe\in \NESA^{\sss e,k,r}$,  all vertices occurring in {$\vpie,\vrhoe$}{} have an age no less than $\zeta n$. Then, for $s< {\zeta} n$, $\pspi,\qspi,\psrho,\qsrho$ are all $0$, and \eqref{eq_bound_falling_decomposition} holds for these \ch{values of $s$}. 

Now we consider the case that $s\geq {\zeta} n$. Since $\vec\pi\cap\vec\rho=\emptyset$, at least one of $\pspi$ and $\psrho$ is zero. Without loss of generality, we assume that $\psrho=0$. As $q_s=\qspi+\qsrho$, we can then decompose the lhs of \eqref{eq_bound_falling_decomposition} as 
\begin{align}\label{left-bound_falling_decomposition}
    &\frac{(\alpha+p_s-1)_{p_s}(\beta_s+q_s-1)_{q_s}}{(\alpha+\beta_s+p_s+q_s-1)_{p_s+q_s}}\\
    &\qquad=\frac{(\alpha+\pspi-1)_{\pspi}(\beta_s+\qspi-1)_{\qspi}}{(\alpha+\beta_s+\pspi+\qspi-1)_{\pspi+\qspi}}\frac{(\beta_s+\qspi+\qsrho-1)_{\qsrho}}{(\alpha+\beta_s+\pspi+\qspi+\qsrho-1)_{\qsrho}}.\nn
\end{align}
For the rhs of \eqref{eq_bound_falling_decomposition}, since $\psrho=0$, 
\begin{align}\label{right-bound_falling_decomposition}
    &\frac{(\alpha+\pspi-1)_{\pspi}(\beta_s+\qspi-1)_{\qspi}}{(\alpha+\beta_s+\pspi+\qspi-1)_{\pspi+\qspi}}\frac{(\alpha+\psrho-1)_{\psrho}(\beta_s+\qsrho-1)_{\qsrho}}{(\alpha+\beta_s+\psrho+\qsrho-1)_{\psrho+\qsrho}}\\
    &\qquad=\frac{(\alpha+\pspi-1)_{\pspi}(\beta_s+\qspi-1)_{\qspi}}{(\alpha+\beta_s+\pspi+\qspi-1)_{\pspi+\qspi}}\frac{(\beta_s+\qsrho-1)_{\qsrho}}{(\alpha+\beta_s+\qsrho-1)_{\qsrho}}.\nn
\end{align}
\zhu{Considering} \eqref{left-bound_falling_decomposition} alongside \eqref{right-bound_falling_decomposition}, we see that they have a common factor 
\begin{align*}
    \frac{(\alpha+\pspi-1)_{\pspi}(\beta_s+\qspi-1)_{\qspi}}{(\alpha+\beta_s+\pspi+\qspi-1)_{\pspi+\qspi}}.
\end{align*}
Hence, it is sufficient to compare $\frac{(\beta_s+\qspi+\qsrho-1)_{\qsrho}}{(\alpha+\beta_s+\pspi+\qspi+\qsrho-1)_{\qsrho}}$ and $\frac{(\beta_s+\qsrho-1)_{\qsrho}}{(\alpha+\beta_s+\qsrho-1)_{\qsrho}}$.

Recall that $\alpha=m+\delta$ and $\beta_s=(2s-3)m+\delta(s-1)$. For each integer $j\in[0,\qsrho-1]$, since $\qsrho\leq k$ and $s\geq {\zeta} n$,
\begin{align*}
    \frac{\beta_s+\qspi+j}{\alpha+\beta_s+\qspi+j}-\frac{\beta_s+j}{\alpha+\beta_s+j}=\frac{\alpha\qspi}{(\alpha+\beta_s+\qspi+j)(\alpha+\beta_s+j)}\leq \Theta {\zeta}^{-2} k n^{-2}\frac{\beta_s+j}{\alpha+\beta_s+j}.
\end{align*}
Then,
\begin{align*}
     \frac{\beta_s+\qspi+j}{\alpha+\beta_s+\pspi+\qspi+j}\leq\frac{\beta_s+\qspi+j}{\alpha+\beta_s+\qspi+j}\leq (1+\Theta {\zeta}^{-2} k n^{-2})\frac{\beta_s+j}{\alpha+\beta_s+j}.
\end{align*}
Hence,
\begin{align}\label{step-bound_falling_decomposition}
    \frac{(\beta_s+\qspi+\qsrho-1)_{\qsrho}}{(\alpha+\beta_s+\pspi+\qspi+\qsrho-1)_{\qsrho}}\leq (1+\Theta {\zeta}^{-2} k n^{-2})^k\frac{(\beta_s+\qsrho-1)_{\qsrho}}{(\alpha+\beta_s+\qsrho-1)_{\qsrho}},
\end{align}
Recall \Cref{def-Theta}. Since $(1+\Theta {\zeta}^{-2} k n^{-2})^k\leq \e^{\Theta {\zeta}^{-2} k^2 n^{-2}}=\Theta^{{\zeta}^{-2} k^2 n^{-2}}$,
 \eqref{eq_bound_falling_decomposition} follows from the combination of \eqref{left-bound_falling_decomposition}-\eqref{step-bound_falling_decomposition}, {as desired}.
\invisible{Then, by \eqref{eq_hd_prod_simp_pq_twopath}-\eqref{eq_bound_falling_decomposition},
\begin{align*}
    &\prob^r(\vpie,\vrhoe\subseteq \PAndmdel) =\Theta^{\frac{1}{\log n}}\prod_{s=2}^n
	\frac{(\alpha+p_s-1)_{p_s}(\beta_s+q_s-1)_{q_s}}{(\alpha+\beta_s+p_s+q_s-1)_{p_s+q_s}}\\
 &\leq \Theta^{\frac{1}{\log n}+{\zeta}^{-2} k^2 n^{-1}}\prod_{s=2}^n\frac{(\alpha+\pspi-1)_{\pspi}(\beta_s+\qspi-1)_{\qspi}}{(\alpha+\beta_s+\pspi+\qspi-1)_{\pspi+\qspi}}\frac{(\alpha+\psrho-1)_{\psrho}(\beta_s+\qsrho-1)_{\qsrho}}{(\alpha+\beta_s+\psrho+\qsrho-1)_{\psrho+\qsrho}}\\
 &= \Theta^{\frac{1}{\log n}+{\zeta}^{-2} k^2 n^{-1}}\prob^r(\vpie\subseteq \PAndmdel)\prob^r(\vrhoe\subseteq \PAndmdel),
\end{align*}
as desired.}
\end{proof}
With \Cref{lem-var-disjoint-paths}, \Cref{lem-con-case-disjoint} follows directly:

\begin{proof}[Proof of \Cref{lem-con-case-disjoint}]
    By \Cref{lem-var-disjoint-paths}, for any $k\leq 2\log_\nu n$, on the good event $\GNW$, 
\begin{align*}
\sum_{\substack{\vpie,\vrhoe\in \NESA^{\sss e,k,r},\\\vec\pi\cap\vec\rho=\emptyset}} &\Big(\prob^r(\vpie, \vrhoe\subseteq \PAndmdel)-\prob^r(\vpie\subseteq \PAndmdel)\prob^r(\vrhoe\subseteq \PAndmdel)\Big)\\
&\leq (\Theta^{\frac{1}{\log n}+{\zeta}^{-2} k^2 n^{-1}}-1)\sum_{\substack{\vpie,\vrhoe\in \NESA^{\sss e,k,r},\\\vec\pi\cap\vec\rho=\emptyset}}\prob^r(\vpie\subseteq \PAndmdel)\prob^r(\vrhoe\subseteq \PAndmdel)\nn\\
&{\leq (\Theta^{\frac{1}{\log n}+{\zeta}^{-2} k^2 n^{-1}}-1)\expec^r[N_{n,r}(k)]^2=o(1)\expec^r[N_{n,r}(k)]^2,}
\end{align*}
{since $\Theta^{\frac{1}{\log n}+{\zeta}^{-2} k^2 n^{-1}}-1=o(1)$, uniformly in the choices of $\Ver_1$ and $\Ver_2$, as desired.}
\invisible{, and 
\begin{align*}
    \sum_{\substack{\vpie,\vrhoe\in \NESA^{\sss e,k,r},\\\vec\pi\cap\vec\rho=\emptyset}}\prob^r(\vpie\subseteq \PAndmdel)\prob^r(\vrhoe\subseteq \PAndmdel)\leq \expec^r[N_{n,r}(k)]^2,
\end{align*}
we conclude that, uniformly in the choices of $\Ver_1$ and $\Ver_2$,
\begin{align*}
\sum_{\substack{\vpie,\vrhoe\in \NESA^{\sss e,k,r},\\\vec\pi\cap\vec\rho=\emptyset}} &\Big(\prob^r(\vpie, \vrhoe\subseteq \PAndmdel)-\prob^r(\vpie\subseteq \PAndmdel)\prob^r(\vrhoe\subseteq \PAndmdel)\Big)\\
&=o(1) \expec^r[N_{n,r}(k)]^2,    
\end{align*}
as desired.}
\end{proof}


\subsection{Path decomposition in Case \ref{it_var_3}}\label{sec-path-decompose-u}
We now consider Case \ref{it_var_3}, where $\vec\pi\cap\vec\rho\neq\emptyset$ and $\vpie\neq\vrhoe$ for some $\vpie,\vrhoe\in \NESA^{\sss e,k,r}$. In this setting, $\vrhoe\backslash\vpie$, {which equals}{} the set of edges in $\vrhoe$ that are not in $\vpie$, can be decomposed as a collection of subpaths {$\vec \rho(1)^{\sss e},\ldots, \vec \rho(u)^{\sss e}$ of $\vrhoe$, for a certain $u\in[k]$. Here, the subpaths are recorded}{} in the order of their labels in $\vrhoe$ (see \Cref{pic_path_decompose}) such that 
\begin{enumerate}
    \item the vertices in $\vec \rho(j)^{\sss e}$ and $\vpie$ are disjoint, except \zhu{for} both ends of $\vec \rho(j)^{\sss e}$;
    \item both ends of $\vec \rho(j)^{\sss e}$ are in $\vpie$, except \zhu{for} 
    $\rho_0$, \ch{which} can be in \zhu{$\partial B_r^{\sss(G_n)}(\Ver_1)\cap [\zeta n,n]$,} and 
    $\rho_k$, \ch{which} can be in $\partial B_r^{\sss(G_n)}(\Ver_2)\cap[\zeta n,n]$.
\end{enumerate}
As $\vpie$ divides $\vrhoe$ into $u$ subpaths, we refer to $u$ as the \textit{decomposition size} of $\vrhoe$ wrt $\vpie$.
\begin{figure}[htbp]
  \centering\def\svgwidth{0.7\columnwidth} 
  \includesvg{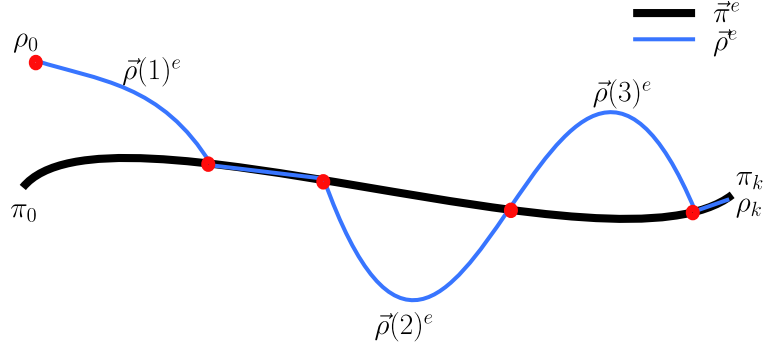}
  \caption{A sample of the path decomposition with decomposition size $u=3$}
  \label{pic_path_decompose}
\end{figure}

Using this decomposition, we can write $\prob^r(\vpie, \vrhoe\subseteq \PAndmdel)$ as follows:
\begin{lemma}[Decomposition of probability for distinct, but intersecting, $\vpie$ and $\vrhoe$]\label{lem_prod_twopath_decompose}
  For $n\geq \zeta^{-2}$ and $k\leq 2\log_\nu n$, with  $\vpie,\vrhoe\in \NESA^{\sss e,k,r}$ such that $\vec\pi\cap\vec\rho\neq\emptyset$ and $\vpie\neq\vrhoe$,
    \zhu{\begin{align}\label{eq_prod_twopath_decompose}
    &\prob^r(\vpie, \vrhoe\subseteq \PAndmdel)=\Theta ^u\prob^r(\vpie\subseteq \PAndmdel)\prod_{j=1}^u \prob^r(\vec\rho(j)^{\sss e}\subseteq \PAndmdel).
\end{align}}
\end{lemma}

\begin{proof}
Since $\vec\pi,\vec\rho$ are self-avoiding paths, in the edge set $\vpie\cup\vrhoe$, each vertex has degree at most $4$ (see \Cref{pic_path_decompose}). Hence, the edge set $\vpie\cup\vrhoe$ satisfies \Cref{ass_e} with $\bdt =\lceil \zeta n\rceil$. Then, analogous to \eqref{eq_hd_prod_simp_pq_twopath}, by \eqref{eq_E-in-PAM-cond},
 \eqan{\label{eq_hd_prod_simp_pq_var}
	\prob^r(\vpie, \vrhoe\subseteq \PAndmdel)
 &=\Theta^{\frac{1}{\log n}}\prod_{s=2}^n \frac{(\alpha+p_s-1)_{p_s}(\beta_s+q_s-1)_{q_s}}{(\alpha+\beta_s+p_s+q_s-1)_{p_s+q_s}},
	}
where now
\eqan{p_s=p_s^{\sss \pi\cup\rho}=\sum_{e\in \vpie\cup\vrhoe}\indic{s=\ushort{e}}=\pspi+\sum_{j=1}^u p_s^{\sss \rho(j)},\nn}
and
\eqan{q_s=q_s^{\sss \pi\cup\rho}=\sum_{e\in \vpie\cup\vrhoe}\indic{s\in (\ushort{e},\bar{e})}=\qspi+\sum_{j=1}^u q_s^{\sss \rho(j)}.\nn
}
{In particular,}{} $p_s\leq 4$ since each vertex has degree at most $4$ in $\vpie\cup\vrhoe$.

\ch{Since} $p_s=p_s^{\sss \pi}+\sum_{j\in[u]}p_s^{\sss \rho(j)}$, we note that, for each $s\in[n]$,
\begin{align}\label{eq-decom-aps1-u}
    (\alpha+p_s-1)_{p_s}= (\alpha+p_s^\pi-1)_{p_s^\pi}\prod_{j=1}^u(\alpha+p_s^{\sss \rho(j)}-1)_{p_s^{\sss \rho(j)}}
\end{align}
precisely when at most one of $p_s^\pi,p_s^{\rho(1)},\ldots,p_s^{\rho(u)}$ is nonzero. Furthermore, even if \eqref{eq-decom-aps1-u} does not hold, since $p_s\leq 4$, $\alpha=m+\delta$ and $(x)_0=1$, we still have
\begin{align*}
    (\alpha+p_s-1)_{p_s}=\Theta  (\alpha+p_s^\pi-1)_{p_s^\pi}\prod_{j=1}^u(\alpha+p_s^{\sss \rho(j)}-1)_{p_s^{\sss \rho(j)}}.
\end{align*}
On the other hand, at least two of $p_s^\pi,p_s^{\rho(1)},\ldots,p_s^{\rho(u)}$ are nonzero only when $s$ is equal to one end of the path $\vec \rho(j)^{\sss e}$
for some $j\in[u]$. Hence, there are at most $2u{\leq 2k}$ such $s$. Consequently,
 \begin{align}\label{eq_divide_prod_var_1}
     \prod_{s=2}^n (\alpha+p_s-1)_{p_s}=\Theta ^u\prod_{s=2}^n (\alpha+p_s^\pi-1)_{p_s^\pi}\prod_{j=1}^u(\alpha+p_s^{\sss \rho(j)}-1)_{p_s^{\sss \rho(j)},}
\end{align}
Moreover, by \eqref{eq_hd_cal_1_E}, for $\bdt =\lceil \zeta n\rceil$ and any edge set $E$ satisfying \Cref{ass_e},
\begin{align}\label{eq_hd_cal_1_E-rep}
    \prod_{s=2}^n \frac{(\beta_s+\qsE-1)_{\qsE}}{(\alpha+\beta_s+\psE+\qsE-1)_{\psE+\qsE}}=\Theta^{\zeta^{-1}\frac{\log^2 n}{n}}(2m+\delta)^{-\abs{E}}\prod_{e\in E}\frac{1}{\ushort{e}^{1-\chi}\bar{e}^{\chi}},
\end{align}
Note that $\Theta^{\zeta^{-1}\frac{\log^2 n}{n}}=\Theta$ for $n\geq\zeta^{-2}$ from \eqref{eq-b-stay-theta}.
Let $\ell_j$ be the length of ${\vec \rho(j)^{\sss e}}$ and $\ell=\sum_{j\in[u]}\ell_j\leq k\leq 2\log_\nu n$. For $E=\vpie\cup\vrhoe$, \eqref{eq_hd_cal_1_E-rep} gives that
\begin{align}
\label{eq_hd_cal_1_picuprho}
    \prod_{s=2}^n \frac{(\beta_s+q_s-1)_{q_s}}{(\alpha+\beta_s+p_s+q_s-1)_{p_s+q_s}}=\Theta(2m+\delta)^{-k-\ell}\prod_{e\in \vpie\cup\vrhoe}\frac{1}{\ushort{e}^{1-\chi}\bar{e}^{\chi}}.
\end{align}
Analogously, for $E={\vec \rho(j)^{\sss e}}$ and $\vpie$, \eqref{eq_hd_cal_1_E-rep} gives that
\begin{align}
\label{eq_hd_cal_1_picuprho-2}
    \prod_{s=2}^n \frac{(\beta_s+q_s^{\sss \rho(j)}-1)_{q_s^{\sss \rho(j)}}}{(\alpha+\beta_s+p_s^{\sss \rho(j)}+q_s^{\sss \rho(j)}-1)_{p_s^{\sss \rho(j)}+q_s^{\sss \rho(j)}}}=\Theta(2m+\delta)^{-\ell_j}\prod_{e\in \vec\rho(j)^{\sss e}}\frac{1}{\ushort{e}^{1-\chi}\bar{e}^{\chi}},
\end{align}
and
\begin{align}
\label{eq_hd_cal_1_picuprho-3}
    \prod_{s=2}^n \frac{(\beta_s+\qspi-1)_{\qspi}}{(\alpha+\beta_s+\pspi+\qspi-1)_{\pspi+\qspi}}=\Theta(2m+\delta)^{-k}\prod_{e\in \vec\pi^{\sss e}}\frac{1}{\ushort{e}^{1-\chi}\bar{e}^{\chi}}.
\end{align}
Hence, we conclude from \eqref{eq_hd_cal_1_picuprho} to \eqref{eq_hd_cal_1_picuprho-3} that
\begin{align}\label{eq_prod_twopath_decompose_0}
    &\prod_{s=2}^n \frac{(\beta_s+\qspi-1)_{\qspi}}{(\alpha+\beta_s+\pspi+\qspi-1)_{\pspi+\qspi}}\prod_{j=1}^u \frac{(\beta_s+q_s^{\sss \rho(j)}-1)_{q_s^{\sss \rho(j)}}}{(\alpha+\beta_s+p_s^{\sss \rho(j)}+q_s^{\sss \rho(j)}-1)_{p_s^{\sss \rho(j)}+q_s^{\sss \rho(j)}}}\\
    &\qquad=\Theta^u(2m+\delta)^{-k-\ell}\prod_{e\in \vpie\cup\vrhoe}\frac{1}{\ushort{e}^{1-\chi}\bar{e}^{\chi}}= \Theta^u\prod_{s=2}^n \frac{(\beta_s+q_s-1)_{q_s}}{(\alpha+\beta_s+p_s+q_s-1)_{p_s+q_s}}.\nn
\end{align}
Combining  \eqref{eq_hd_prod_simp_pq_var}, \eqref{eq_divide_prod_var_1} and \eqref{eq_prod_twopath_decompose_0} gives that
\begin{align*}
        &\prob^r(\vpie, \vrhoe\subseteq \PAndmdel)\\
    &=\Theta ^u\prod_{s=2}^n \frac{(\alpha+p_s^\pi-1)_{p_s^\pi}(\beta_s+\qspi-1)_{\qspi}}{(\alpha+\beta_s+\pspi+\qspi-1)_{\pspi+\qspi}}\prod_{j=1}^u \frac{(\alpha+p_s^{\sss \rho(j)}-1)_{p_s^{\sss \rho(j)}}(\beta_s+q_s^{\sss \rho(j)}-1)_{q_s^{\sss \rho(j)}}}{(\alpha+\beta_s+p_s^{\sss \rho(j)}+q_s^{\sss \rho(j)}-1)_{p_s^{\sss \rho(j)}+q_s^{\sss \rho(j)}}}\nn.
\end{align*}
Then \eqref{eq_prod_twopath_decompose} follows directly from applying \Cref{lem_prob_pie_cond} to the {above equation}.
\end{proof}

\subsection{Analysis of Case \ref{it_var_3}: Decomposition of  paths}\label{sec-con-case-intersect}
In this section, we prove \Cref{lem-con-case-intersect}.
With \Cref{lem_prod_twopath_decompose} in hand, we aim to bound the product factors in \eqref{eq_prod_twopath_decompose}\zhu{, which, by \Cref{lem_prob_pie_cond}, is
\begin{align*}
      \Theta^u\prod_{s=2}^n \prod_{j=1}^u \frac{(\alpha+p_s^{\sss \rho(j)}-1)_{p_s^{\sss \rho(j)}}(\beta_s+q_s^{\sss \rho(j)}-1)_{q_s^{\sss \rho(j)}}}{(\alpha+\beta_s+p_s^{\sss \rho(j)}+q_s^{\sss \rho(j)}-1)_{p_s^{\sss \rho(j)}+q_s^{\sss \rho(j)}}},
\end{align*}
}
using $\expec^r[N_{n,r}(k)]$. To do so, we express $\vec\rho(j)$ as the tuple $(\rho(j)_0,\ldots,\rho(j)_{\ell_j})$ and let  
$\lb^{\rho(j)}_i$ be the label of $\rho(j)_i$ in $\vec\rho$. 

In  \Cref{lem_hd_sum_restrict_path}, we have proved that for any $n\geq \zeta^{-2}$, any $\ell\leq 2\log_\nu n$, any $\lb^\pi_{0}\in \cbc{\rO,\rY}$ and any integers $\nVer_1,\nVer_2\in[\zeta n,n]$ such that $\nVer_1\not\in B_{r-1}^{\sss(G_n)}(\Ver_1)\cup B_r^{\sss(G_n)}(\Ver_2)$ and $\nVer_2\not\in B_r^{\sss(G_n)}(\Ver_1)\cup B_{r-1}^{\sss(G_n)}(\Ver_2)$,
    \begin{align}\label{eq_hd_sum_restrict_path-rep}
    &\sum_{\vpie\in \TESA_{\nVer_1,\nVer_2}^{\sss e,\ell}}\prod_{s=2}^n\frac{(\alpha+p_s-1)_{p_s}(\beta_s+q_s-1)_{q_s}}{(\alpha+\beta_s+p_s+q_s-1)_{p_s+q_s}}\\
        &\qquad= \Theta n^{-\ell}\sum_{\vec\pi\in \TSA_{\nVer_1,\nVer_2}^{\sss \ell}}\prod_{i=1}^k\kappa\big((\pi_{i-1}/n,\lb^\pi_{i-1}),(\pi_{i}/n,\lb^\pi_i)\big).\nn
\end{align}
\ch{Since} $\TSA_{\rho(j)_0,\rho(j)_{\ell_j}}^{\sss \ell_j}\subseteq \NA_{\rho(j)_0,\rho(j)_{\ell_j}}^{\ell_j}$, \eqref{eq_hd_sum_restrict_path-rep} gives that
\begin{align}\label{eq_hd_lower_form_prod_kappa_var}
    &\sum_{\vec\rho(j)^{\sss e}\in \TESA_{\rho(j)_0,\rho(j)_{\ell_j}}^{\sss e,\ell_j}} \prod_{s=2}^n \frac{(\alpha+p_s^{\sss \rho(j)}-1)_{p_s^{\sss \rho(j)}}(\beta_s+q_s^{\sss \rho(j)}-1)_{q_s^{\sss \rho(j)}}}{(\alpha+\beta_s+p_s^{\sss \rho(j)}+q_s^{\sss \rho(j)}-1)_{p_s^{\sss \rho(j)}+q_s^{\sss \rho(j)}}}\\
    &\qquad\leq \Theta n^{-\ell_j}\sum_{\vec\rho(j)\in\NA_{\rho(j)_0,\rho(j)_{\ell_j}}^{\ell_j}}\prod_{i=1}^{\ell_j}\kappa\big((\rho(j)_{i-1}/n,\lb^{\rho(j)}_{i-1}),(\rho(j)_{i}/n,\lb^{\rho(j)}_i)\big).\nn
\end{align} 

In the remainder of this section, we divide the proof of \Cref{lem-con-case-intersect} into two subcases, according the decomposition size $u$.

\paragraph{The decomposition size is at least $2$}
We first consider the case that $u\geq 2$:
\begin{lemma}[Contribution of cases with decomposition size at least $2$]\label{lem-var_upper_part1_0}
For any $\vep\in (0,1)$ and $k\leq (1+\vep)\log_\nu n$, on the good event $\GNW$, uniformly in the choices of $\Ver_1$ and $\Ver_2$,
    \begin{align}
\label{var_upper_part1_0}
    &\sum_{\vpie\in \NESA^{\sss e,k,r}} \sum_{\substack{\vrhoe\in \NESA^{\sss e,k,r},\\ u \geq 2}}\prob^r(\vpie, \vrhoe\subseteq \PAndmdel) = o(1)\expec^r[N_{n,r}(k)]^2.
\end{align}
\end{lemma}

\begin{proof}
To prove \eqref{var_upper_part1_0}, we  assume that $n\geq \zeta^{-2}$. By \zhu{\Cref{lem_prob_pie_cond,lem_prod_twopath_decompose}}, the lhs of \eqref{var_upper_part1_0} is
\begin{flalign}
\label{var_upper_part1_0-decom}
    \sum_{\vpie\in \NESA^{\sss e,k,r}} \sum_{\substack{\vrhoe\in \NESA^{\sss e,k,r},\\ \sss u \geq 2}}\prob^r(\vpie, \vrhoe\subseteq \PAndmdel) 
 = \sum_{\vpie\in \NESA^{\sss e,k,r}}\prob^r(\vpie\subseteq \PAndmdel)\\
 \times\sum_{\substack{\vrhoe\in \NESA^{\sss e,k,r},\\ \sss u \geq 2}}\Theta^u\prod_{s=2}^n\prod_{j=1}^u \frac{(\alpha+p_s^{\sss \rho(j)}-1)_{p_s^{\sss \rho(j)}}(\beta_s+q_s^{\sss \rho(j)}-1)_{q_s^{\sss \rho(j)}}}{(\alpha+\beta_s+p_s^{\sss \rho(j)}+q_s^{\sss \rho(j)}-1)_{p_s^{\sss \rho(j)}+q_s^{\sss \rho(j)}}}.\nn
\end{flalign}
To handle the products on the rhs of \eqref{var_upper_part1_0-decom}, we first note from \Cref{lem-gen-sum-product-kappa} and \eqref{eq_hd_lower_form_prod_kappa_var} that, for each $j\in [u]$,
\begin{align}\label{eq_hd_lower_form_prod_kappa_var_2}
         \sum_{\vec\rho(j)^{\sss e}\in \TESA_{\rho(j)_0,\rho(j)_{\ell_j}}^{\sss e,\ell_j}} \prod_{s=2}^n \frac{(\alpha+p_s^{\sss \rho(j)}-1)_{p_s^{\sss \rho(j)}}(\beta_s+q_s^{\sss \rho(j)}-1)_{q_s^{\sss \rho(j)}}}{(\alpha+\beta_s+p_s^{\sss \rho(j)}+q_s^{\sss \rho(j)}-1)_{p_s^{\sss \rho(j)}+q_s^{\sss \rho(j)}}}\leq \Theta \zeta^{-2}\nu^{\ell_j} n^{-1}.
    \end{align}

On the other hand, given $\vpie\in \NESA^{\sss e,k,r}$, for each subpath $\vec\rho(j)$ of $\vec\rho$, its two ends $\rho(j)_0,\rho(j)_{\ell_j} $ are in the set $ H=\vec\pi\cup \partial B_r^{\sss(G_n)}(\Ver_1)\cup \partial B_r^{\sss(G_n)}(\Ver_2)$. Hence, on the good event $\GNW$, there are at most $k+2M$ choices for each of these ends.  Then, \eqref{eq_hd_lower_form_prod_kappa_var_2} gives that
\begin{align}\label{eq_product_sum_twopath-old}
    &\sum_{\substack{\vrhoe\in \NESA^{\sss e,k,r},\\ u \geq 2}}\Theta^u\prod_{s=2}^n\prod_{j=1}^u \frac{(\alpha+p_s^{\sss \rho(j)}-1)_{p_s^{\sss \rho(j)}}(\beta_s+q_s^{\sss \rho(j)}-1)_{q_s^{\sss \rho(j)}}}{(\alpha+\beta_s+p_s^{\sss \rho(j)}+q_s^{\sss \rho(j)}-1)_{p_s^{\sss \rho(j)}+q_s^{\sss \rho(j)}}}\\
    \leq& \sum_{u=2}^k  \sum_{\ell_1,\ldots,\ell_u}\Theta^u\prod_{j=1}^u\Big(\sum_{\substack{\rho(j)_0,\rho(j)_{\ell_j}\in H,\\\vec\rho(j)^{\sss e}\in \NESA_{\rho(j)_0,\rho(j)_{\ell_j}}^{\sss e, \ell_j}}}\prod_{s=2}^n\frac{(\alpha+p_s^{\sss \rho(j)}-1)_{p_s^{\sss \rho(j)}}(\beta_s+q_s^{\sss \rho(j)}-1)_{q_s^{\sss \rho(j)}}}{(\alpha+\beta_s+p_s^{\sss \rho(j)}+q_s^{\sss \rho(j)}-1)_{p_s^{\sss \rho(j)}+q_s^{\sss \rho(j)}}}\Big)\nn\\
\leq&\sum_{u=2}^k \sum_{\ell_1,\ldots,\ell_u}\Theta^u\prod_{j=1}^u(k+2M)^2\Theta{\zeta}^{-2}\nu^{\ell_j}n^{-1}\leq \sum_{u=2}^k\sum_{\ell_1,\ldots,\ell_u}\Theta^u (k+2M)^{2u}{\zeta}^{-2u}\nu^{k}n^{-u} ,\nn
\end{align}
where $\sum_{\ell_1,\ldots,\ell_u}$ sums over all $\ell_1,\ldots,\ell_u$ such that $\ell_i\in [k]$ and $\ell_1+\cdots+\ell_u\leq k$, and we note that there are at most $k^u$ choices for these $(\ell_1,\ldots,\ell_u)$ that satisfy this condition. Hence, \eqref{eq_product_sum_twopath-old} yields that
{\small
\begin{align}\label{eq_product_sum_twopath-further}
    & \sum_{\substack{\vrhoe\in \NESA^{\sss e,k,r},\\ \sss u \geq 2}}\Theta^{u}\prod_{s=2}^n\prod_{j=1}^u \frac{(\alpha+p_s^{\sss \rho(j)}-1)_{p_s^{\sss \rho(j)}}(\beta_s+q_s^{\sss \rho(j)}-1)_{q_s^{\sss \rho(j)}}}{(\alpha+\beta_s+p_s^{\sss \rho(j)}+q_s^{\sss \rho(j)}-1)_{p_s^{\sss \rho(j)}+q_s^{\sss \rho(j)}}}\\
&\qquad\leq \sum_{u=2}^k k^u\Theta^u (k+2M)^{2u}{\zeta}^{-2u}\nu^{k}n^{-u}.\nn
\end{align}
}

Let $\Theta^{\star,3}$ be a uniform upper bound of $\Theta$ up to now.  Let $n_1$ be such that, for all $n\geq n_1$, 
\begin{align*}
    2\Theta^{\star,3} \log_\nu n(2\log_\nu n+2M)^2\zeta^{-2} n^{-1}\leq 1.
\end{align*}
Then, for $n\geq n_1$ and $k\leq 2\log_\nu n$, \eqref{eq_product_sum_twopath-further} gives that
\begin{align}\label{eq_product_sum_twopath}
    & \sum_{\substack{\vrhoe\in \NESA^{\sss e,k,r},\\ \sss u \geq 2}}\Theta^u\prod_{s=2}^n\prod_{j=1}^u \frac{(\alpha+p_s^{\sss \rho(j)}-1)_{p_s^{\sss \rho(j)}}(\beta_s+q_s^{\sss \rho(j)}-1)_{q_s^{\sss \rho(j)}}}{(\alpha+\beta_s+p_s^{\sss \rho(j)}+q_s^{\sss \rho(j)}-1)_{p_s^{\sss \rho(j)}+q_s^{\sss \rho(j)}}}\\
&\qquad\leq  \Theta \zeta^{-4} k^2(k+2M)^4 \nu^k n^{-2}\frac{1}{1-\Theta^{\star,3} \zeta^{-2}k(k+2M)^2 n^{-1}}\nn\\
&\qquad= \Theta \zeta^{-4} k^2(k+2M)^4 \nu^k n^{-2}.\nn
\end{align}
Consequently,  we conclude from \eqref{var_upper_part1_0-decom} and \eqref{eq_product_sum_twopath} that 
\begin{align*}
    &\sum_{\vpie\in \NESA^{\sss e,k,r}} \sum_{\substack{\vrhoe\in \NESA^{\sss e,k,r},\\ {\scriptscriptstyle u \geq 2}}}\prob^r(\vpie, \vrhoe\subseteq \PAndmdel) 
 \leq \Theta   \zeta^{-4} k^2(k+2M)^4 \nu^k n^{-2}\expec^r[N_{n,r}(k)]^2.\nn
\end{align*}
 Further, as
 $k\leq (1+\vep)\log_\nu n$, uniformly in the choices of $\Ver_1$ and $\Ver_2$,
\begin{align*}
    \Theta   \zeta^{-4} k^2(k+2M)^4 \nu^k n^{-2}=\Theta \zeta^{-4} (\log n)^2(\log n+M)^4 n^{{-(1-\vep)}}=o(1),
\end{align*}
and thus we conclude that \eqref{var_upper_part1_0} holds.
\end{proof}

\paragraph{The decomposition size is $1$}
We now turn to the case that $u=1$: 
\begin{lemma}[Contribution of cases \ch{with} decomposition  \zhu{size} $1$]\label{lem-u=1}
With $n_0$ defined as in \Cref{lem_hd_sum_restrict_path-sum-gen}, for $n>n_0$ and $k\leq(1+\vep)\log_\nu n$, on the good event $\GNW$, 
    \begin{align}\label{eq_u=1}
  &\sum_{\vpie, \vrhoe\in \NESA^{\sss e,k,r}\colon u=1}\prob^r(\vpie, \vrhoe\subseteq \PAndmdel)\leq \Theta \bc{3/2}^{-r} {\zeta}^{-4} c_{\zeta}^{-1}\expec^r[N_{n,r}(k)]^2,
\end{align}
\end{lemma}
Recall from \Cref{sub-sec-modify-kappa-lower-bound} that
\eqn{
	\nka\big((x,s), (y,t)\big)=\frac{c_{st}(\indic{x>y, t=\srO}+\indic{x< y, t=\srY})}{(x\vee y)^{\chi}(x\wedge y)^{1-\chi}}+\frac{c_{\srO\srY}\indic{x>y, t=\srO}}{x},\nn
	}
and $\NP_{u,v}^{\sss \ell}=\cbc{(\pi_0,\ldots,\pi_k)\in [n]^{k+1}:~\pi_0=u,\pi_k=v,\pi_i\geq \zeta n~\forall i\in[k-1]}$ is the set of 
$k$-step path between $u,v$ without the avoidance constraints.
Before \ch{giving} the proof of \Cref{lem-u=1}, we first present the following lemma, which states that the sums of kernel products on paths with different endpoints are comparable:
\begin{lemma}[Comparison of kernel products with different ends]\label{cor_easy-general-bound}
For any $\ell\geq 1$, any $u,v,u',v'\in[\zeta n,n]$, and $\lb^\pi_0,\lb^\rho_0\in\cbc{\rO,\rY}$, 
\begin{align}\label{eq_easy-general-bound-up}
&\sum_{\vec\pi\in \NP_{u,v}^\ell}\prod_{i=1}^{\ell}\nka\big((\pi_{i-1}/n,\lb^\pi_{i-1}),(\pi_{i}/n,\lb^\pi_i)\big)\\
&\qquad\leq \Theta {\zeta}^{-2}\sum_{\vec\rho\in \NP_{u',v'}^\ell}\prod_{i=1}^\ell\nka\big((\rho_{i-1}/n,\lb^\rho_{i-1}),(\rho_{i}/n,\lb^\rho_i)\big).\nn
\end{align}
\end{lemma}

\begin{proof}
\zhu{Fix}  $\ell\geq 0$. For $\vec\pi=(\pi_i)_{0\leq i\leq \ell}\in[\zeta n,n]^\ell$ and $\vec\rho=(\rho_i)_{0\leq i\leq \ell}\in[\zeta n,n]^\ell$, if $\pi_i=\rho_i$ for all $i\in[\ell-1]$, we claim that for any $\lb^\pi_0,\lb^\rho_0\in\cbc{\rO,\rY}$,
\begin{align}\label{eq-a-path-easy-general-ub}
    \prod_{i=1}^{\ell}\nka\big((\pi_{i-1}/n,\lb^\pi_{i-1}),(\pi_{i}/n,\lb^\pi_i)\big)\leq \Theta \zeta^{-2}\prod_{i=1}^{\ell}\nka\big((\rho_{i-1}/n,\lb^\rho_{i-1}),(\rho_{i}/n,\lb^\rho_i)\big).
\end{align}


Indeed, if $\ell=1,2$, as $\pi_i,\rho_i\in[\zeta n,n]$, \eqref{eq-a-path-easy-general-ub} follows directly from \eqref{eq-change-label-add-theta1-new}. \zhu{For} higher values of $\ell$, this is a straightforward computation.

\invisible{Move to apopendix:
\RvdH{Do the labels (which also depend on start and end point) not create a problem?}\zhu{This is absorbed in the $\Theta$ factor.}

If $\ell\geq 3$, since $\pi_i=\rho_i$ for all $i\in[\ell-1]$,  \eqref{def_label_direction} gives that $\lb^\pi_j=\lb^\rho_j$ for integer $j\in[2,\ell-1]$. Hence,
\begin{align*}
    \prod_{i=3}^{\ell-1}\nka\big((\pi_{i-1}/n,\lb^\pi_{i-1}),(\pi_{i}/n,\lb^\pi_i)\big)= \prod_{i=3}^{\ell-1}\nka\big((\rho_{i-1}/n,\lb^\rho_{i-1}),(\rho_{i}/n,\lb^\rho_i)\big).
\end{align*}
For $i=1$ or $\ell$, by  \eqref{eq-change-label-add-theta1-new},
\begin{align*}
   \nka\big((\pi_{i-1}/n,\lb^\pi_{i-1}),(\pi_{i}/n,\lb^\pi_i)\big)\leq \Theta \zeta^{-1}\leq\Theta \zeta^{-1}\nka\big((\rho_{i-1}/n,\lb^\rho_{i-1}),(\rho_{i}/n,\lb^\rho_i)\big).
\end{align*}
For $i=2$, since $\pi_j=\rho_j$ for $j=1,2$, by \eqref{eq-change-label-add-theta1-new},
\begin{align*}
    \nka\big((\pi_1/n,\lb^\pi_1),(\pi_2/n,\lb^\pi_2)\big)=\Theta \nka\big((\rho_1/n,\lb^\rho_1),(\rho_2/n,\lb^\rho_2)\big).
\end{align*}
Consequently, 
\begin{align*}
    \prod_{i=1}^{\ell}\nka\big((\pi_{i-1}/n,\lb^\pi_{i-1}),(\pi_{i}/n,\lb^\pi_i)\big)\leq \Theta \zeta^{-2}\prod_{i=1}^{\ell}\nka\big((\rho_{i-1}/n,\lb^\rho_{i-1}),(\rho_{i}/n,\lb^\rho_i)\big),
\end{align*}
i.e., \eqref{eq-a-path-easy-general-ub} holds.}

Let $\pi_0=u$, $\pi_\ell=v$, $\rho_0=u'$, $\rho_\ell=v'$. Since $\pi_i=\rho_i$ for all $i\in[\ell-1]$, \eqref{eq_easy-general-bound-up} follows directly from summing \eqref{eq-a-path-easy-general-ub} over all integer $\pi_i=\rho_i\in[\zeta n, n]$ for all $i\in[\ell-1]$. 
\end{proof}

With \Cref{cor_easy-general-bound} in hand, we are now ready to prove \Cref{lem-u=1}:

\begin{proof}[Proof of \Cref{lem-u=1}]
For $u=1$, let $\ell=\ell_1$ be the length of path $\vec\rho(1)$. According to the decomposition, $\vec\pi$ and $\vec\rho(1)=\vec\rho\backslash\vec\pi$ can only intersect at the two ends of the subpath  $\vec\rho(1)$. Then, depending on whether $\rho(1)_0$, $\rho(1)_\ell$ or both are in $\vec\pi$, there are three possible cases, which we illustrate in \Cref{pic_path_decompose_u1}, and we note from $\vec\pi\cap\vec\rho\neq\emptyset$ that at least one of the ends of $\vec\rho(1)$ is in the path $\vec\pi$.
\begin{figure}[ht]
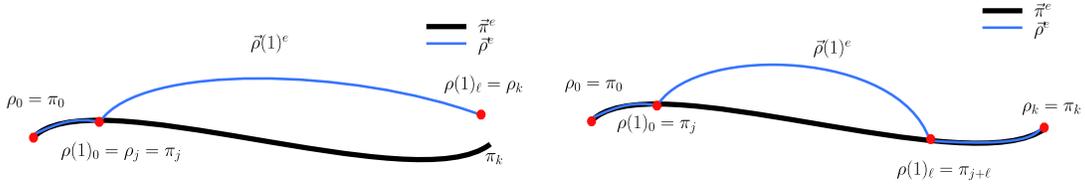

\begin{minipage}[r]{0.48\linewidth}
\includesvg[width=1\linewidth]{path_decomposition_0.svg}
\end{minipage}\hfill
\begin{minipage}[r]{0.48\linewidth}
\includesvg[width=1\linewidth]{path_decomposition_1.svg}
\end{minipage}\hfill
\caption{Two cases for the path decomposition of $\vrhoe$ arise when \zhu{the} decomposition size $u$ \zhu{equals $1$} and $\rho(1)_0\in\vec\pi$, where this categorization depends on  whether  $\rho(1)_\ell$ is in $\vec\pi$. The case where $\rho(1)_0\not\in\vec\pi$ and $\rho(1)_\ell\in\vec\pi$ is simply the reverse of the first graph.}
  \label{pic_path_decompose_u1}
\end{figure}

We first assume that $\rho(1)_0\in\vec\pi$ and show that, on the good event $\GNW$,
    \begin{align}\label{eq_u=1-1}
  &\sum_{\vpie\in \NESA^{\sss e,k,r}}\sum_{\substack{\vrhoe\in \NESA^{\sss e,k,r},\\u=1,\rho(1)_0\in\vec\pi}}\prob^r(\vpie, \vrhoe\subseteq \PAndmdel)\leq \Theta \bc{3/2}^{-r} {\zeta}^{-4} c_{\zeta}^{-1}\expec^r[N_{n,r}(k)]^2.
\end{align}

Fix $j\in[0,k-1]$ and let $\rho(1)_0=\pi_j$. Then, $\rho_i=\pi_i$ for each $ i\leq j$ and $i\geq j+\ell+1$, as shown in the two {plots in}{} \Cref{pic_path_decompose_u1}, \ch{ since, if not,}  $u\geq 2$. 
Furthermore,
\begin{itemize}
    \item[$\rhd$] if $\ell=k-j$, which corresponds to the first {plot in \Cref{pic_path_decompose_u1}}, then $\rho(1)_\ell=\rho_k\in \partial B_r^{\sss(G_n)}(\Ver_2)\cap[\zeta n,n]$ (recall that $\vrhoe\in \NESA^{\sss e,k,r}$);
    \item[$\rhd$] if $\ell< k-j$, which corresponds to the second {plot in \Cref{pic_path_decompose_u1}}, then $\rho(1)_\ell=\pi_{j+\ell}$.
\end{itemize} 
Let $\mathscr{K}_i=\partial B_r^{\sss(G_n)}(\Ver_i)\cap[\zeta n,n]$ for $i=1,2$. 
We then define 
\begin{align}
\label{A-def}
    A=A(j,\ell)=\begin{cases}
        \pi_{j+\ell},\quad&\ell<k-j,\\
        \mathscr{K}_2,\quad&\ell=k-j,
    \end{cases}
\end{align}
\ch{to be} the set of all possible choices for $\rho_{j+\ell}$.

By \zhu{\Cref{lem_prob_pie_cond,lem_prod_twopath_decompose}}, {for $\vpie, \vrhoe$ for which $u=1$}, for $n\geq \zeta^{-2}$, on the good event $\GNW$,
\begin{align}\label{eq_half_replace}
    &\prob^r(\vpie, \vrhoe\subseteq \PAndmdel)=\Theta  \prob^r(\vpie\subseteq \PAndmdel)\prod_{s=2}^n  \frac{(\alpha+p_s^{\sss \rho(1)}-1)_{p_s^{\sss \rho(1)}}(\beta_s+q_s^{\sss \rho(1)}-1)_{q_s^{\sss \rho(1)}}}{(\alpha+\beta_s+p_s^{\sss \rho(1)}+q_s^{\sss \rho(1)}-1)_{p_s^{\sss \rho(1)}+q_s^{\sss \rho(1)}}}.
\end{align}
 
We first fix $\vpie$, $j$ and $\ell$. Then combining  \eqref{eq_half_replace} with \eqref{eq_hd_lower_form_prod_kappa_var} for $u=1$ gives that
\begin{align}\label{product_twopath_n2k-j}
 &\sum_{\substack{\vrhoe\in \NESA^{\sss e,k,r},u=1,\\\rho_j=\pi_j,\rho_{j+\ell}\in A}}\prob^r(\vpie, \vrhoe\subseteq \PAndmdel)\leq \sum_{a\in A}\sum_{\vec\rho(1)^{\sss e}\in \NA_{\pi_j,a}^{\sss e,\ell}}\prob^r(\vpie, \vrhoe\subseteq \PAndmdel)\\
 &= \Theta  \prob^r(\vpie\subseteq \PAndmdel)\sum_{a\in A}\sum_{\vec\rho(1)^{\sss e}\in \NA_{\pi_j,a}^{\sss e,\ell}}\prod_{s=2}^n  \frac{(\alpha+p_s^{\sss \rho(1)}-1)_{p_s^{\sss \rho(1)}}(\beta_s+q_s^{\sss \rho(1)}-1)_{q_s^{\sss \rho(1)}}}{(\alpha+\beta_s+p_s^{\sss \rho(1)}+q_s^{\sss \rho(1)}-1)_{p_s^{\sss \rho(1)}+q_s^{\sss \rho(1)}}}\nn\\
&\leq \Theta n^{-\ell} \prob^r(\vpie\subseteq \PAndmdel)
\sum_{a\in A}\sum_{\vec\rho(1)\in \NA_{\pi_j,a}^{\sss \ell}}\prod_{i=1}^{\ell}\kappa\big((\rho(1)_{i-1}/n,\lb^{\rho(1)}_{i-1}),(\rho(1)_{i}/n,\lb^{\rho(1)}_i)\big).\nn
\end{align}
Since 
\begin{align}\label{eq-sum-two-paths-rewrite}
    &\sum_{\vpie\in \NESA^{\sss e,k,r}}\sum_{\substack{\vrhoe\in \NESA^{\sss e,k,r},u=1}}\prob^r(\vpie, \vrhoe\subseteq \PAndmdel)\\
    &\qquad=\sum_{\vpie\in \NESA^{\sss e,k,r}}\sum_{\ell\in[k]}\sum_{j=0}^{k-\ell}\sum_{\substack{\vrhoe\in \NESA^{\sss e,k,r},u=1,\\\rho_j=\pi_j,\rho_{j+\ell}\in A}}\prob^r(\vpie, \vrhoe\subseteq \PAndmdel),
\nn\end{align}
if one can show that, for suffciently large $n$, uniformly in  $\vpie\in \NESA^{\sss e,k,r}$,
\begin{align}\label{eq-eq-sum-two-paths-rewrite-aim}
    &\sum_{\ell\in[k]}\sum_{j=0}^{k-\ell}n^{-\ell}\sum_{a\in A}\sum_{\vec\rho(1)\in \NA_{\pi_j,a}^{\sss \ell}}\prod_{i=1}^{\ell}\kappa\big((\rho(1)_{i-1}/n,\lb^{\rho(1)}_{i-1}),(\rho(1)_{i}/n,\lb^{\rho(1)}_i)\big)\\
    &\qquad\leq \Theta \bc{3/2}^{-r} {\zeta}^{-4} c_{\zeta}^{-1}\expec^r[N_{n,r}(k)],\nn
\end{align}
then \eqref{eq_u=1-1} follows from the combination of \eqref{product_twopath_n2k-j}-\eqref{eq-eq-sum-two-paths-rewrite-aim}. In what follows, we prove a stronger statement that reads, for $A$ as in \eqref{A-def},
\begin{align}\label{eq-eq-sum-two-paths-rewrite-aim-new}
    &\sum_{\ell\in[k]}\sum_{j=0}^{k-\ell}n^{-\ell}\sum_{a\in A}\sum_{\vec\rho(1)\in \NP_{\pi_j,a}^{\sss \ell}}\prod_{i=1}^{\ell}\nka\big((\rho(1)_{i-1}/n,\lb^{\rho(1)}_{i-1}),(\rho(1)_{i}/n,\lb^{\rho(1)}_i)\big)\\
    &\qquad\leq \Theta \bc{3/2}^{-r} {\zeta}^{-4} c_{\zeta}^{-1}\expec^r[N_{n,r}(k)].\nn
\end{align}

A natural idea would be \ch{to use} inequalities like \eqref{eq_hd_lower_trans_sum_int} and \eqref{eq_hd_lower_part1-ded} to {upper}{} bound 
{the summand for fixed $\ell$ and $j$,}
\invisible{\begin{align*}
n^{-\ell}\sum_{a\in A}\sum_{\vec\rho(1)\in \NP_{\pi_j,a}^{\sss \ell}}\prod_{i=1}^{\ell}\nka\big((\rho(1)_{i-1}/n,\lb^{\rho(1)}_{i-1}),(\rho(1)_{i}/n,\lb^{\rho(1)}_i)\big)
\end{align*}} which {would give}{} rise to a {bound}{} of the order of $\Theta n^{-1}\nu^\ell$. For $k=\rou{(1+\vep)\log_\nu n}$ and $\ell=k$, we have $n^{-1}\nu^\ell=n^{\vep}$. 

{Unfortunately}, by \Cref{prop-path-count-PAMc-LB_simple}, we only have a lower bound of the order of $n^{\vep/2}$ for the rhs of \eqref{eq-eq-sum-two-paths-rewrite-aim-new}. Consequently, {our previous bounds are}{} not enough to deduce \eqref{eq-eq-sum-two-paths-rewrite-aim-new}. 
The essential problem arises from the fact that we use the spectral norm to derive the upper bound and the (truncated) spectral radius to derive the lower bound, resulting in the loss of a sublinear factor when applying \Cref{pro_hd_spe_con}. To overcome this problem, we need to handle this {delicate}{} comparison more carefully, {as we next explain.}
\smallskip



Recall from \eqref{eq_hd_lowerbound_r} that, for $k\leq 2\log_\nu n$, $n\geq \zeta^{-2}$ and any $\lb^\pi_0\in\cbc{\rO,\rY}$, on the good event $\GNW$,
\begin{align*}
\expec^r[N_{n,r}(k)]&=\sum_{\vrhoe\in \NESA^{\sss e,k,r}}\prob^r(\vrhoe\subseteq \PAndmdel)\\
    &=\Theta n^{-k}\sum_{\substack{\nVer_1\in \mathscr{K}_1,\\\nVer_2\in \mathscr{K}_2}}\sum_{\vec\rho\in \TSA_{\nVer_1,\nVer_2}^{\sss k}}\prod_{i=1}^k\kappa\big((\rho_{i-1}/n,\lb^\rho_{i-1}),(\rho_{i}/n,\lb^\rho_i)\big)\nn.\\
    \end{align*}
Then,  \Cref{lem_hd_sum_restrict_path-sum-gen-nka-compare} \ch{gives} that, for $n\geq n_0$, on the event $\GNW$,
\begin{align}\label{eq_hd_sum_restrict_path-sum-gen-var}
\expec^r[N_{n,r}(k)]
    &= \Theta n^{-k}\sum_{\substack{\nVer_1\in \mathscr{K}_1,\\\nVer_2\in \mathscr{K}_2}}\sum_{\vec\rho\in \NP_{\nVer_1,\nVer_2}^{\sss k}}\prod_{i=1}^k\nka\big((\rho_{i-1}/n,\lb^\rho_{i-1}),(\rho_{i}/n,\lb^\rho_i)\big).
    \end{align}

{Comparing}{} $\sum_{\vec\rho\in \NP_{\nVer_1,\nVer_2}^{\sss k}}\prod_{i=1}^k\nka\big((\rho_{i-1}/n,\lb^\rho_{i-1}),(\rho_{i}/n,\lb^\rho_i)\big)$ with $$\sum_{\vec\rho(1)\in \NP_{\pi_j,a}^{\sss \ell}}\prod_{i=1}^{\ell}\nka\big((\rho(1)_{i-1}/n,\lb^{\rho(1)}_{i-1}),(\rho(1)_{i}/n,\lb^{\rho(1)}_i)\big),$$
we see that the difference arises only from the length of the path, and \ch{its ends}, where the latter can only give rise to a $\Theta \zeta^{-2}$ factor according to \Cref{cor_easy-general-bound}. To handle  the difference {arising} from the length of the path, we want to divide the paths in the first summation into two subpaths, one is a path of length $\ell$, which is the same as $\rho(1)$, the other is a path of length $k-\ell$, and we bound the contribution of the latter path from above using  \Cref{lem_hd_sum_restrict_path-sum-gen-nka}.

As previously stated, we divide $\vec\rho$ into two subpaths \zhu{as} $\vec\xi(1)=(\rho_0,\ldots,\rho_{k-\ell})$ \zhu{and} $\vec\xi(2)=(\rho_{k-\ell},\ldots,\rho_{k})$. Then, the length of $\vec\xi(2)$ is $\ell$ and
\begin{align}
&\sum_{\vec\rho\in \NP_{\nVer_1,\nVer_2}^{\sss k}}\prod_{i=1}^k\nka\big((\rho_{i-1}/n,\lb^\rho_{i-1}),(\rho_{i}/n,\lb^\rho_i)\big)\label{eq:decompose-rho-xi}\\
    &\qquad= \sum_{b\in [\zeta n,n]}\sum_{\vec\xi(1)\in \NP_{\nVer_1,b}^{\sss k-\ell}}\prod_{i=1}^{k-\ell}\nka\big((\rho_{i-1}/n,\lb^\rho_{i-1}),(\rho_{i}/n,\lb^\rho_i)\big)\nn\\
    &\qquad\quad\times\sum_{\vec\xi(2)\in \NP_{b,\nVer_2}^{\sss \ell}}\prod_{i=k-\ell+1}^{k}\nka\big((\rho_{i-1}/n,\lb^\rho_{i-1}),(\rho_{i}/n,\lb^\rho_i)\big)\nn.
\end{align}
We then show that
\begin{align}
    &\sum_{\nVer_2\in K_2}\sum_{\vec\xi(2)\in \NP_{b,\nVer_2}^{\sss \ell}}\prod_{i=k-\ell+1}^{k}\nka\big((\rho_{i-1}/n,\lb^\rho_{i-1}),(\rho_{i}/n,\lb^\rho_i)\big)\label{eq:rho1-xi2}\\
    &\qquad\geq \Theta \zeta^{2}\sum_{a\in A}\sum_{\vec\rho(1)\in \NP_{\pi_j,a}^{\sss \ell}}\prod_{i=1}^{\ell}\nka\big((\rho(1)_{i-1}/n,\lb^{\rho(1)}_{i-1}),(\rho(1)_{i}/n,\lb^{\rho(1)}_i)\big).\nn
\end{align}

Recall that  $\mathscr{K}_2=\partial B_r^{\sss(G_n)}(\Ver_2)\cap[\zeta n,n]$. By \Cref{cor_easy-general-bound}, as $\mathscr{K}_2\neq \emptyset$ by \Cref{ass_Br},
\begin{itemize}
    \item[$\rhd$] if $A=\mathscr{K}_2$, then, for any $\nVer_2\in \mathscr{K}_2$ and any $b\in [\zeta n,n]$,
    \begin{align}\label{eq-path-prod-k2-1}
    \sum_{\vec\rho(1)\in \NP_{\pi_j,\nVer_2}^{\sss \ell}}&\prod_{i=1}^{\ell}\nka\big((\rho(1)_{i-1}/n,\lb^{\rho(1)}_{i-1}),(\rho(1)_{i}/n,\lb^{\rho(1)}_i)\big)\\
    &\leq \Theta \zeta^{-2}\sum_{\vec\xi(2)\in \NP_{b,\nVer_2}^{\sss \ell}}\prod_{i=k-\ell+1}^{k}\nka\big((\rho_{i-1}/n,\lb^\rho_{i-1}),(\rho_{i}/n,\lb^\rho_i)\big).\nn
\end{align}
\item[$\rhd$] if $A=\pi_{j+\ell}$, then, for any $\nVer_2\in \mathscr{K}_2$ and any $b\in [\zeta n,n]$,
\begin{align}\label{eq-path-prod-k2-2}
    \sum_{\vec\rho(1)\in \NP_{\pi_j,\pi_{j+\ell}}^{\sss \ell}}&\prod_{i=1}^{\ell}\nka\big((\rho(1)_{i-1}/n,\lb^{\rho(1)}_{i-1}),(\rho(1)_{i}/n,\lb^{\rho(1)}_i)\big)\\
    &\leq \Theta \zeta^{-2}\sum_{\vec\xi(2)\in \NP_{b,\nVer_2}^{\sss \ell}}\prod_{i=k-\ell+1}^{k}\nka\big((\rho_{i-1}/n,\lb^\rho_{i-1}),(\rho_{i}/n,\lb^\rho_i)\big).\nn
\end{align}
\end{itemize}
Combining \eqref{eq-path-prod-k2-1} and \eqref{eq-path-prod-k2-2} yields \eqref{eq:rho1-xi2}.

For the remaining  factor, i.e., that on the last line of \eqref{eq:decompose-rho-xi}, \zhu{note from \Cref{ass_Br} that $\abs{\mathscr{K}_1}\geq (3/2)^r$ on the good event $\GNW$. Then, by   \Cref{lem_hd_sum_restrict_path-sum-gen-nka}, with $F_1=\mathscr{K}_1$ and $F_2=[\zeta n,n]$ in \Cref{lem_hd_sum_restrict_path-sum-gen-nka}, for sufficiently large $n$, on the good event $\GNW$},
\begin{align*}
    \sum_{\nVer_1\in \mathscr{K}_1}&\sum_{b\in [\zeta n,n]}\sum_{\vec\xi(1)\in \NP_{\nVer_1,b}^{\sss j}}\prod_{i=1}^{k-\ell}\nka\big((\rho_{i-1}/n,\lb^\rho_{i-1}),(\rho_{i}/n,\lb^\rho_i)\big)\\
    &\geq \Theta \abs{F_1}\abs{F_2}{\zeta}^2c_{\zeta}\nu^{(k-\ell)(1-\vep/4)}n^{k-\ell-1}=\Theta \bc{3/2}^r {\zeta}^2c_{\zeta}\nu^{(k-\ell)(1-\vep/4)}n^{k-\ell}.
\end{align*}
Combining the above inequality with \eqref{eq:decompose-rho-xi} and \eqref{eq:rho1-xi2} gives that 
\begin{flalign}\label{eq-bound-rho-with-without-pi}
    \sum_{\substack{\nVer_1\in \mathscr{K}_1,\\\nVer_2\in \mathscr{K}_2}}\sum_{\vec\rho\in \NP_{\nVer_1,\nVer_2}^{\sss k}}\prod_{i=1}^k\nka\big((\rho_{i-1}/n,\lb^\rho_{i-1}),(\rho_{i}/n,\lb^\rho_i)\big)
    \geq \Theta \bc{3/2}^r {\zeta}^4 c_{\zeta}\nu^{(k-\ell)(1-\vep/4)}n^{k-\ell}\nn\\
    \times\sum_{a\in A}\sum_{\vec\rho(1)\in \NP_{\pi_j,a}^{\sss \ell}}\prod_{i=1}^{\ell}\nka\big((\rho(1)_{i-1}/n,\lb^{\rho(1)}_{i-1}),(\rho(1)_{i}/n,\lb^{\rho(1)}_i)\big).
\end{flalign}
Note from \eqref{eq_hd_sum_restrict_path-sum-gen-var} that the lhs of \eqref{eq-bound-rho-with-without-pi} is $\Theta n^k \expec^r[N_{n,r}(k)]$. Then,
\begin{align}\label{eq-bound-rho-with-without-pi-2}
    &n^{-\ell}\sum_{a\in A}\sum_{\vec\rho(1)\in \NP_{\pi_j,a}^{\sss \ell}}\prod_{i=1}^{\ell}\nka\big((\rho(1)_{i-1}/n,\lb^{\rho(1)}_{i-1}),(\rho(1)_{i}/n,\lb^{\rho(1)}_i)\big)\nn\\
    &\quad\leq \Theta \bc{3/2}^{-r} {\zeta}^{-4} c_{\zeta}^{-1}\nu^{(\ell-k)(1-\vep/4)}\expec^r[N_{n,r}(k)].
\end{align}
Summing \eqref{eq-bound-rho-with-without-pi-2} over all integers $\ell\in[k]$ and $j\in[0,k-\ell]$ {leads us to}
\begin{align*}
    &\sum_{\ell\in[k]}\sum_{j=0}^{k-\ell}n^{-\ell}\sum_{a\in A}\sum_{\vec\rho(1)\in \NP_{\pi_j,a}^{\sss \ell}}\prod_{i=1}^{\ell}\nka\big((\rho(1)_{i-1}/n,\lb^{\rho(1)}_{i-1}),(\rho(1)_{i}/n,\lb^{\rho(1)}_i)\big)\\
    &\qquad\leq \Theta \bc{3/2}^{-r} {\zeta}^{-4} c_{\zeta}^{-1}\expec^r[N_{n,r}(k)]\sum_{\ell\in[k]}\sum_{j=0}^{k-\ell}\nu^{(\ell-k)(1-\vep/4)}.
\end{align*}
Since $\vep\in(0,1)$, $\sum_{s=0}^\infty (s+1)\nu^{-s(1-\vep/4)}=\Theta $. Thus, for sufficiently large $n$ and $k\leq (1+\vep)\log_\nu n$, {\eqref{eq-eq-sum-two-paths-rewrite-aim-new} holds}{} uniformly in  $\vpie\in \NESA^{\sss e,k,r}$, which eventually leads to \eqref{eq_u=1-1}. Moreover, \zhu{analogously} to the proof {of}{} \eqref{eq_u=1-1}, on the good event $\GNW$, for $n\geq n_0$,
    \begin{align}\label{eq_u=1-2}
  &\sum_{\vpie\in \NESA^{\sss e,k,r}}\sum_{\substack{\vrhoe\in \NESA^{\sss e,k,r},\\u=1,\rho(1)_\ell\in\vec\pi}}\prob^r(\vpie, \vrhoe\subseteq \PAndmdel)\leq \Theta \bc{3/2}^{-r} {\zeta}^{-4} c_{\zeta}^{-1}\expec^r[N_{n,r}(k)]^2.
\end{align}
Combining \eqref{eq_u=1-1} and \eqref{eq_u=1-2} finishes the proof of \Cref{lem-u=1}.
\end{proof}
\begin{proof}[Proof of \Cref{lem-con-case-intersect}]
    \Cref{lem-con-case-intersect} follows directly from \Cref{lem-var_upper_part1_0,lem-u=1}.
\end{proof}


\longversion{
\begin{appendix}
\section{Good events are likely: Proof of Lemma~\ref{LEM_RE_BR}}\label{sec-proof-ass-Br}
In this section, we prove \Cref{LEM_RE_BR}.
For this purpose, we first establish some regularity results for $\PPT$, which is the local limit of $\PAndmdel$ introduced in \Cref{sec_offspring_operator}, as we describe below. Then, using the local convergence, we extend the results to $\PAndmdel$. Let us start by introducing marked local convergence.

\paragraph{A metric on the space of marked-rooted graphs} 
To describe how close two marked-rooted graphs are, we introduce a metric to measure the distance between them.

Given two complete separable metric spaces $\Xi$ and $\Xi'$, let $(G_1,o_1)$ and $(G_2,o_2)$ be two (locally finite) marked-rooted graphs, where each vertex $v_i\in G_i$ is marked with $m_i(v_i)\in \Xi$ and each edge $e_i\in G_i$ is marked with $m_i'(e_i)\in \Xi'$ for $i=1,2$. We denote the metric on $\Xi$ by $\dRG_\Xi$, and the metric on $\Xi'$ by $\dRG_{\Xi'}$.
Then we define the distance between $(G_1,o_1)$ and $(G_2,o_2)$ as 
\begin{align*}
    \dRG_{\mathcal{G}_\star}\big((G_1,o_1),(G_2,o_2)\big)=\frac{1}{1+R},
\end{align*}
where $R$ is the largest integer $r$ such that the two marked-rooted graphs $\zhu{\bar{B}_r^{(G_1)}}(o_1)$ and $\zhu{\bar{B}_r^{(G_2)}}(o_2)$ are isomorphic when ignoring their vertex and edge marks, and there exists a bijection $\varphi$ from the vertex and edge sets of $\zhu{\bar{B}_r^{(G_1)}}(o_1)$ to the vertex and edge sets of $\zhu{\bar{B}_r^{(G_2)}}(o_2)$ such that  
\begin{itemize}
\item[$\rhd$]$\Ver_1=\varphi(\Ver_2)$;
\item[$\rhd$]an edge $e\in \zhu{\bar{B}_r^{(G_1)}}(o_1)$ has endpoints $u$ and $v$ if and only if the edge $\varphi(e)\in  \zhu{\bar{B}_r^{(G_2)}}(o_2)$ and it has endpoints $\varphi(u)$ and $\varphi(v)$;
    \item[$\rhd$]for each $v\in \zhu{\bar{B}_r^{(G_1)}}(o_1)$, $\dRG_\Xi(m_1(v),m_2(\varphi(v)))\leq 1/r$;
    \item[$\rhd$]for each edge $e\in\zhu{\bar{B}_r^{(G_1)}}(o_1)$, $\dRG_{\Xi'}(m_1'(e),m_2'(\varphi(e))\leq 1/r$.
\end{itemize}
\paragraph{The rooted tree $\PPT$ is the local limit of $(\PAndmdel,\Ver)$ under $\dRG_{\mathcal{G}_\star}$} Let $\Ver$ be a uniformly chosen vertex in $\PAndmdel$. 
In order to compare $(\PAndmdel,\Ver)$ with the rooted tree $\PPT$, we define the \textit{mark} of each vertex in $\PAndmdel$ as its age divided by $n$, and the mark of each edge in $\PAndmdel$ as its label in $[m]$. Let
\begin{itemize}
    \item[$\rhd$]$\dRG_\Xi$  be the absolute difference on $\RR$ such that $\dRG_\Xi(x,y)=\abs{x-y}$;
    \item[$\rhd$] $\dRG_{\Xi'}$  be the discrete metric on $[m]$ such that $\dRG_{\Xi'}(x,y)=\indic{x\neq y}$.
\end{itemize}
Then, in \cite[Theorem 1.5]{GarHazHofRay22}, Garavaglia, Hazra, the first author and Ray proved that for any (locally finite) marked-rooted graph $H$ with vertex marks in $[0,1]$ equipped with metric $\dRG_\Xi$ and edge marks in $[m]$ equipped with metric $\dRG_{\Xi'}$, and any $R\geq 0$, as $n$ tends to infinity,
\begin{align}
\label{dis-pam-local-limit-0}
\PP\bc{\dRG_{\mathcal{G}_\star}\big((\PAndmdel,\Ver),H\big)\leq \frac{1}{1+R}}\longrightarrow\PP\bc{\dRG_{\mathcal{G}_\star}\big(\PPT,H\big)\leq \frac{1}{1+R}},
\end{align}
In the context of random graph theory, this means that $\PPT$ is the local weak limit of $\PAndmdel$. In fact, the above convergence can be strengthened to convergence {\em in probability} when looking at proportions of vertices \ch{$v\in [n]$} for which $\dRG_{\mathcal{G}_\star}\big((\PAndmdel,\ch{v}),H\big)\leq \frac{1}{1+R}$. This implies marked local convergence {\em in probability}, which  implies that the marked local neighborhoods around two uniformly chosen vertices are close to being {\em independent.}

According to \eqref{dis-pam-local-limit-0}, if we want to investigate the neighborhood of a uniformly chosen vertex in $\PAndmdel$ for large $n$, it is sufficient to examine the neighborhood of the root of $\PPT$. This is generally how the local limit \ch{helps} in proving \ch{desirable properties} of random graphs, including global properties such as the size of the giant component and diameter \cite{Hofs24}. For us, this link will be crucial in proving \Cref{LEM_RE_BR}.

\begin{proof}[Proof of \Cref{LEM_RE_BR}]
We begin by investigating $\PPT$. Given $\bar{\bf p}$ as a  $\PPT$ truncated at depth $r$ (i.e., removing all the vertices at distance more than $r$ from the root), we consider the probability that the following holds:
\begin{enumerate}[label=(\Alph*)]
    \item the vertices in $\bar{\bf p}$  are marked with elements in $[2\eta,1]$;\label{It_pb_1}
        \item the number of vertices in $\bar{\bf p}$ is no more than $M$;\label{It_pb_2}
    \item at least $(3/2)^r$ leaves of $\bar{\bf p}$ are marked with elements in $[5{\zeta},1-4{\zeta}]$;\label{It_pb_3}
            \item the tree $\bar{\bf p}$ has a root $\varnothing$ and each leaf in $\bar{\bf p}$ has a distance $r$ from $\varnothing$;\label{It_pb_4}
        \item the edges in $\bar{\bf p}$ are labeled in $[m]$;\label{It_pb_5}
            \item  for each non-leaf vertex $v$ in  $\bar{\bf p}$ and each $j\in[m]$, there is precisely one
edge labeled $j$ such that $v$ is the younger endpoint of this edge.\label{It_pb_6}
\end{enumerate}
We note that conditions \ref{It_pb_4}-\ref{It_pb_6} are necessary  for $\bar{\bf p}$ to be a $\PPT$ truncated at depth $r$ (see \Cref{sec_offspring_operator}), and thus for $\bar{\bf p}$ they hold with probability $1$. 
Further, by \cite[Lemma 4.7]{GarHazHofRay22}, for each $r\geq 0$,
\begin{align*}
    \lim_{M\to\infty}\limsup_{\eta\searrow 0} \PP\Big(\text{\ref{It_pb_1} and \ref{It_pb_2} hold for $\bar{\bf p}$}\Big)=1.
\end{align*}

Next, we deal with \ref{It_pb_3}. We first note that for the root or a vertex with label $\rO$ in $\PPT$, it has $m$ offsprings with label $\rO$. As $m\geq 2$, for each $r\geq 2$, there are at least $2^{r-2}$ vertices at  distance $r-2$ to the root in $\PPT$  that are labeled $\rO$ and  are older than the root. We further note from \Cref{sec_offspring_operator} that, with $\chi=\frac{m+\delta}{2(m+\delta)}$,
\begin{itemize}
    \item[$\rhd$] for each vertex $\omega$ with label $\rO$ and age $A_\omega$, the number of its \ch{offspring} with label $\rY$ and age between $[1-4\zeta,1]$ follows a conditional Poisson distribution with random mean 
    \[
    \int_{A_\omega\vee (1-4\zeta)}^1\frac{\Gamma_\omega(1-\chi)}{x^{\chi}A_\omega^{1-\chi}}\dint x=\Gamma_\omega A_\omega^{\chi-1}(1-(A_\omega\vee(1-4\zeta))^{1-\chi}),
    \]
    where $\Gamma_\omega$ is a gamma distribution with shape $m+\delta+1$ and rate $1$.

    \item[$\rhd$] for each vertex $\omega i$ with label $\rY$ and age $A_{\omega i}\geq 1-4\zeta$, it has at least \zhu{one} offspring with label $\rO$, and the age of this vertex has distribution $U_{\omega i}^{1/\chi}A_{\omega i}$, where $U_{\omega i} $ is uniformly distributed in $[0,1]$, independent of everything else.
\end{itemize}
Given $A_\varnothing<1$ as the age of the root, for each $\omega$ with label $\rO$ and $A_\omega\leq A_\varnothing$, since $x^{\chi-1}(1-(x\vee(1-4\zeta))^{1-\chi})$ is positive in $(0,A_\varnothing]$ and tends to $\infty$ as $x\searrow 0$, it has a positive lower bound $c(A_\varnothing)$ in $(0,A_\varnothing]$. For $\Gamma_\omega$, let $c'=\PP\bc{\Gamma_\omega\geq 1}>0$. Then, with probability at least $\e^{-c(A_\varnothing)c'}$, $\omega$ has a child $A_{\omega i}$ with label $\rY$ and age in the interval $[1-4\zeta,1]$. For this child, with probability $\big((1-4\zeta)^{\chi}-(5\zeta)^{\chi}\big)/A_{\omega i}^{\chi}\geq (1-4\zeta)^{\chi}-(5\zeta)^{\chi}=c''$, it has a child with label $\rO$ and age between $[5\zeta,1-4\zeta]$. Therefore, with probability at least $\e^{-c(A_\varnothing)c'}c''$, $\omega$ has a grandchild with age between $[5\zeta,1-4\zeta]$. 

Consequently, when $r\geq 3$, for the $2^{r-2}$ previously mentioned vertices  at distance $r-2$ to the root in $\PPT$ that are labeled $\rO$ and  are older than the root, the total number of their grandchildren with age in $[5\zeta,1-4\zeta]$ is stochastically bounded from below by $\Bin\big(2^{r-2},\e^{-c(A_\varnothing)c'}c''\big)$, a binomial random variable with $2^{r-2}$ trials and success probability $\e^{-c(A_\varnothing)c'}c''$. For sufficiently large $r$ such that $2^{r-2}\e^{-c(A_\varnothing)c'}c''>(3/2)^r$, by Chebyshev's inequality, as $r\to\infty$,
\begin{align*}
    &\PP\bc{\Bin\big(2^{-r+2},\e^{-c(A_\varnothing)c'}c''\big)\leq \bc{3/2}^r\mid A_\varnothing}\\
    &\qquad\leq \PP\bc{\big|\Bin\big(2^{-r+2},\e^{-c(A_\varnothing)c'}c''\big)-2^{r-2}\e^{-c(A_\varnothing)c'}c''\big|\geq 2^{r-2}\e^{-c(A_\varnothing)c'}c''-\bc{3/2}^r\mid A_\varnothing}\\
    &\qquad\leq \frac{2^{r-2}\e^{-c(A_\varnothing)c'}c''(1-\e^{-c(A_\varnothing)c'}c'')}{(2^{r-2}\e^{-c(A_\varnothing)c'}c''-\bc{3/2}^r)^2}\to 0.
\end{align*}
By the law of total expectation and \zhu{the} dominated convergence theorem, we conclude that
\begin{align*}
    \lim_{r\to \infty}\PP\bc{\Bin\big(2^{-r+2},\e^{-c(A_\varnothing)c'}c''\big)\leq \bc{3/2}^r}=0,
\end{align*}
which means that as $r\to\infty$, the probability that there are at least $(3/2)^r$ vertices with age in $[5\zeta,1-4\zeta]$ and at distance $r$ of the root in $\PPT$ tends to $1$. Hence, uniformly in the values of $M$ and $\eta$,
\begin{align*}
\lim_{r\to\infty}\PP\Big(\text{\ref{It_pb_3} holds for $\bar{\bf p}$}\Big)=1.
\end{align*}
Consequently,
\begin{align}\label{eq-ap-local-limit-sat-ass}
   \lim_{r\to\infty}\lim_{M\to\infty}\lim_{\eta\searrow 0}\PP\bc{\text{\ref{It_pb_1} to \ref{It_pb_6} hold for $\bar{\bf p}$}}=1.
\end{align}

Recall the definition of $\mathcal{G}_\star$ and $\dRG_{\mathcal{G}_\star}$ from \Cref{sec_offspring_operator}, where $\mathcal{G}_\star$ is the set of all (locally finite) rooted-marked graphs with vertex marks in $[0,1]$ equipped with the absolute difference {metric}{} on $\RR$ and edge marks in $[m]$ equipped with {the discrete}{} metric.

\zhu{To extend the result from $\PPT$ to $\PAndmdel$, by \eqref{dis-pam-local-limit-0}, 
for any marked-rooted graph $H$ with vertex marks in $[0,1]$ and edge marks in $[m]$, and any $R\geq 0$, with $\Ver$ a uniformly chosen vertex in $\PAndmdel$, as $n$ tends to infinity,}
\begin{align}\label{dis-pam-local-limit}
\PP\bc{\dRG_{\mathcal{G}_\star}\big((\PAndmdel,\Ver),H\big)\leq \frac{1}{1+R}}\longrightarrow\PP\bc{\dRG_{\mathcal{G}_\star}\big(\PPT,H\big)\leq \frac{1}{1+R}},
\end{align}
where the mark of each vertex in $\PAndmdel$ is its age divided by $n$.
{Equation}{} \eqref{dis-pam-local-limit} yields that $(\PAndmdel,\Ver)$ converges to $\PPT$ in distribution under the metric $\dRG_{\mathcal{G}_\star}$, using a standard argument similar to the proof of \cite[Theorem 2.7]{Hofs24}. \footnote{We treat two isomorphic marked-rooted graphs as the same graph in this convergence.}
\smallskip

Further, we note that for random variables $(X_n)_{n\geq 1}$ on $\RR$, by \cite[Theorem 3.2.8]{durrett2019probability}, if $X_n$ converges in distribution to $X$, we can couple $X_n$ and $X$ to $(\hat{X}_n,\hat{X})$ such that $\hat{X}_n$ converges almost surely to $\hat{X}$. By \cite[Theorem A.8]{Hofs24}, with metric $\dRG_{\mathcal{G}_\star}$, the space $\mathcal{G_\star}$ is a Polish space. Then, \cite[Theorem 13.1.1]{Dudl02} yields that this space of marked-rooted graphs and $\RR$ are isomorphic. 
Hence, we can couple $(\PAndmdel,\Ver)$ and $\PPT$ such that $(\PAndmdel,\Ver)$ converges almost surely to $\PPT$. Through this coupling, for any non-negative $R$, as $n\to\infty$,
\begin{align}\label{eq-ap-con-mark-graph}
\PP\bc{\dRG_{\mathcal{G}_\star}\big((\PAndmdel,\Ver),\PPT\big)\leq \frac{1}{1+R}}\to 0.
\end{align}

Let $G_n=\PAndmdel$. With $R=\max\cbc{r,\zeta^{-1},\eta^{-1}}$ and $\bar{\bf p}$ as a $\PPT$ truncated at depth $r$, once 
\begin{align*}
    \dRG_{\mathcal{G}_\star}\big((\PAndmdel,\Ver),\PPT\big)\leq \frac{1}{1+R},
\end{align*}
there is a bijection $\varphi$ from the vertex and edge sets of $\zhu{\bar{B}_r^{(G_n)}}(\Ver)$ to the 
vertex and edge sets of $\bar{\bf p}$ such that
\begin{itemize}
\item[$\rhd$] $\Ver=\varphi(\varnothing)$;
\item[$\rhd$] an edge $e\in \zhu{\bar{B}_r^{(G)}}(o)$ has endpoints $u$ and $v$ if and only if the edge $\varphi(e)\in  \bar{\bf p}$ has endpoints $\varphi(u)$ and $\varphi(v)$;

    \item[$\rhd$] $\abs{m(v)-m(\varphi(v))}\leq 1/r$ for each $v\in \zhu{\bar{B}_r^{(G_n)}}(\Ver)$;
    \item[$\rhd$] $m'(e)=m'(\varphi(e))$ for each edge $e$ in $\zhu{\bar{B}_r^{(G_n)}}(\Ver)$.
\end{itemize}
{Here}{} $m(v)$ and $m'(u,v)$ are the marks of vertex $v$ and the {directed}{} edge $(u,v)$ in the graph, respectively. Specifically, $\abs{m(v)-m(\varphi(v))}\leq \max\cbc{\zeta,\eta}$. Hence, the ages of each vertex in $ \zhu{\bar{B}_r^{(G_n)}}(\Ver)$ are in $[\eta n,n]$, and there are at least $(3/2)^r$ vertices in $\PAndmdel$ at distance $r$ of $\Ver$ such that their ages are in $[4\zeta n,(1-3\zeta)n]$. We further note that $\gdist{\PAndmdel}\bc{\Ver_1,\Ver_2}\geq 2r+1$ implies that the $r$-neighborhoods of $\Ver_1$ and $\Ver_2$ are disjoint. Hence, for two uniformly chosen vertices $\Ver_1,\Ver_2$ in $\PAndmdel$, by a union bound,
\begin{align*}
&\PP\bc{\text{$\big(\bar{B}_r^{\sss(G_n)}(\Ver_1), \bar{B}_r^{\sss(G_n)}(\Ver_2)\big)$ does not satisfy \ref{it_pb_1} to \ref{it_pb_6} in \Cref{ass_Br}}}\\
    &\leq \PP\bc{\gdist{\PAndmdel}\bc{\Ver_1,\Ver_2}\leq 2r}+2\PP\bc{\dRG_{\mathcal{G}_\star}\big((\PAndmdel,\Ver),\PPT\big)\geq \frac{1}{1+R}}\\
    &\quad+2\PP\bc{\text{At least one of \ref{It_pb_1} to \ref{It_pb_6} does not hold for $\bar{\bf p}$}}.
\end{align*}
Therefore, by \Cref{pro_lowerbound}, \eqref{eq-ap-local-limit-sat-ass} and \eqref{eq-ap-con-mark-graph},
\begin{align}\label{eq-ap-Br-ass}
\lim_{r\to\infty}\limsup_{M\to\infty}\limsup_{\eta\searrow 0}&\limsup_{n\to\infty}\PP\Big(\text{$\big(\bar{B}_r^{\sss(G_n)}(\Ver_1), \bar{B}_r^{\sss(G_n)}(\Ver_2)\big)$ does not satisfy}\\
&\hspace{4cm}\text{\ch{\ref{it_pb_1} to \ref{it_pb_6} in \Cref{ass_Br}}}\Big)=0\nonumber.
\end{align}

For \eqref{eq-ass-on-psi}, we assume that $\big(\bar{B}_r^{\sss(G_n)}(\Ver_1), \bar{B}_r^{\sss(G_n)}(\Ver_2)\big)=(\bar{{\bf t}}_1,\bar{{\bf t}}_2)$ is good, that is, it satisfies \ref{it_pb_1} to \ref{it_pb_6} in \Cref{ass_Br}. Note that
\begin{align*}
    \Big(1-\sum_{v\in V^{\sss\circ}(\bar{{\bf t}}_1)\cup V^{\sss\circ}(\bar{{\bf t}}_2)}\psi_\nu\Big)^{\log^2 n}\geq 1-\log^2 n\sum_{v\in V^{\sss\circ}(\bar{{\bf t}}_1)\cup V^{\sss\circ}(\bar{{\bf t}}_2)}\psi_\nu.
\end{align*}
Recall that $(\psi_j)_{2\leq j\leq n}$ are independent beta random
variables with parameters $\alpha=m+\delta$ and $\beta_j=(2j-3)m+\delta(j-1)$. {Thus,}{} $\Erw\brk{\psi_j}=\frac{m+\delta}{(2j-2)m+j\delta}$. Then,
\ch{since} $V^{\sss\circ}(\bar{{\bf t}}_1)\cup V^{\sss\circ}(\bar{{\bf t}}_2)\subseteq [\eta n,n]$ and $\abs{V^{\sss\circ}(\bar{{\bf t}}_1)\cup V^{\sss\circ}(\bar{{\bf t}}_2)}\leq 2M$, 
\begin{align*}
    \Erw\Big[\sum_{v\in V^{\sss\circ}(\bar{{\bf t}}_1)\cup V^{\sss\circ}(\bar{{\bf t}}_2)}\psi_\nu\Big]\leq 2M\frac{m+\delta}{(2\zeta n-2)m+\zeta n \delta}.
\end{align*}
Hence, Markov's inequality yields that, as $n\to\infty$,
\begin{align*}
    \PP\Big(\log^2 n\sum_{v\in V^{\sss\circ}(\bar{{\bf t}}_1)\cup V^{\sss\circ}(\bar{{\bf t}}_2)}\psi_\nu\geq 1-\e^{-1}\Big)\leq (1-\e^{-1})^{-1}2M\log^2 n\frac{m+\delta}{(2\zeta n-2)m+\zeta n \delta}\longrightarrow 0.
\end{align*}
Therefore, as $n\to\infty$,
\begin{align}\label{eq-ap-psi-ass}
    &\PP\Big(1-\sum_{v\in V^{\sss\circ}(\bar{{\bf t}}_1)\cup V^{\sss\circ}(\bar{{\bf t}}_2)}\psi_\nu\geq \e^{-\frac{1}{\log^2 n}}\mid \text{$(\bar{{\bf t}}_1,\bar{{\bf t}}_2)$ is good}\Big)\\
    &\qquad\geq \PP\Big(\Big(1-\sum_{v\in V^{\sss\circ}(\bar{{\bf t}}_1)\cup V^{\sss\circ}(\bar{{\bf t}}_2)}\psi_\nu\Big)^{\log^2 n}\geq \e^{-1}\mid\text{$(\bar{{\bf t}}_1,\bar{{\bf t}}_2)$ is good}\Big)\nn\\
    &\qquad\geq \PP\Big(\log^2 n\sum_{v\in V^{\sss\circ}(\bar{{\bf t}}_1)\cup V^{\sss\circ}(\bar{{\bf t}}_2)}\psi_\nu\leq 1-\e^{-1}\mid\text{$(\bar{{\bf t}}_1,\bar{{\bf t}}_2)$ is good}\Big)\longrightarrow 1.\nn
\end{align}
Combining \eqref{eq-ap-Br-ass} and \eqref{eq-ap-psi-ass} gives \Cref{LEM_RE_BR}, as desired.
\end{proof}

\section{Convergence of spectral radius: Proof of Remark~\ref{RE-CON-SPE-RADIUS}}\label{sec-app-conv-nu-kappa-hvep}
Recall that 
\begin{align*}
    \kappa^\circ_{\zeta}((x,s),(y,t))=\indic{x,y\in [3{\zeta},1-{\zeta}]} \kappa\big((x,s),(y,t)\big).
\end{align*}
Let $\nu_{\zeta}$ be the spectral radius of $\bfT_{\kappa^\circ_{\zeta}}$. In this section, we prove \Cref{RE-CON-SPE-RADIUS}, that is,
\begin{align}\label{eq-conv-nu-kappa}
    \lim_{\zeta\searrow 0}\nu_{\zeta}=\nu.
\end{align}

\begin{proof}[Proof of \eqref{eq-conv-nu-kappa}]

    
We first show that $\bfT_{\kappa^\circ_{\zeta}}^2$ is irreducible in $\mathcal{S}_{{\zeta}}=[3{\zeta},1-{\zeta}]\times\cbc{\rO,\rY}$. \zhu{Recall from \eqref{cst-def-PAM-rep} that $c_{\srY\srO}=\min_{s,t\in\cbc{\srO,\srY}} c_{st}$. Then, by \eqref{kappa-PAM-fin-def-form-a-rep-new_term} and \eqref{cst-def-PAM-rep}, $\kappa\big((x,s),(y,t)\big)\geq c_{\srY\srO}$ when ${x>y,~t=\rO}$ or ${x<y,~t=\rY}$. Hence}, given $(x,s)\in (3{\zeta},1-{\zeta})\times\cbc{\rO,\rY}$, for a nonzero and non-negative function $f$ in $L^2(\mathcal{S}_{{\zeta}})$,
\eqan{\label{eq_hd_bound_T2fcirc}
\bfT_{\kappa^\circ_{\zeta}}^2 f(x,s)\geq&\int_{x}^{1-{\zeta}}\int_{3{\zeta}}^{x}\kappa\big((x,s),(y,\rO)\big)\kappa\big((y,\rO),(z,\rY)\big)f(z,\rY)\dint y\dint z\\
&+\int_{3{\zeta}}^{x}\int_{x}^{1-\zeta}\kappa\big((x,s),(y,\rY)\big)\kappa\big((y,\rY),(z,\rO)\big)f(z,\rO)\dint y\dint z\nn\\
\geq& c_{\srY\srO}^2\min\cbc{x-3\zeta,1-\zeta-x}\int_{3{\zeta}}^{1-{\zeta}}(f(z,\rO)+f(z,\rY))\dint z>0,\nn
}
that is, $\bfT_{\kappa^\circ_{\zeta}}^2$ is irreducible {on the space}{} $\mathcal{S}_{{\zeta}}=[3{\zeta},1-{\zeta}]\times\cbc{\rO,\rY}$.
Then, by \cite[Theorem 43.8 and items (ii) and (iii) after that]{Zaan97} (see \cite[Theorem 41.6]{Zaan97} for the definition of \ch{a} Hille-Tamarkin operator), with  $r(\bfT_{\kappa^\circ_{\zeta}}^2)$ the spectral radius of $\bfT_{\kappa^\circ_{\zeta}}^2$, there exists a $u_{\zeta}$ in $L^2(\mathcal{S}_{{\zeta}})$ such that $u_{\zeta}>0$ almost surely in $\mathcal{S}_{{\zeta}}$ and 
\begin{align}\label{eq-T^2u}
\bfT_{\kappa^\circ_{\zeta}}^2u_{\zeta}=r(\bfT_{\kappa^\circ_{\zeta}}^2)u_{\zeta}.
\end{align}
Further, the combination of the Cauchy–Schwarz inequality and $u_{\zeta} \in L^2(\mathcal{S}_{{\zeta}})$ yields that $u_{\zeta} \in L(\mathcal{S}_{{\zeta}})$. Hence, we can assume that 
\begin{align}\label{eq_normalize-u}
\int_{3{\zeta}}^{1-{\zeta}}(u_{\zeta}(z,\rO)+u_{\zeta}(z,\rY))\dint z=1.   
\end{align}
Moreover, the definition of the spectral  radius gives that $$r(\bfT_{\kappa^\circ_{\zeta}}^2)=\lim_{k\to\infty}\lVert \bfT_{\kappa^\circ_{\zeta}}^{2k}\rVert^{1/k}=\lim_{k\to\infty}\big(\lVert \bfT_{\kappa^\circ_{\zeta}}^{2k}\rVert^{1/2k}\big)^2=r(\bfT_{\kappa^\circ_{\zeta}})^2=\nu_{\zeta}^2.$$
Hence, the combination of \eqref{eq_hd_bound_T2fcirc}, \eqref{eq-T^2u} and \eqref{eq_normalize-u} yields that, for any $x\in[4\zeta,1-2\zeta]$ and any $s\in\cbc{\rO,\rY}$,
\begin{align}\label{eq_r-spec-1}
\nu_{\zeta}^2 u_{\zeta}(x,s)=\bfT_{\kappa^\circ_{\zeta}}^2u_{\zeta}(x,s)&\geq c_{\srY\srO}^2\min\cbc{x-3\zeta,1-\zeta-x}\int_{3{\zeta}}^{1-{\zeta}}(u_{\zeta}(z,\rO)+u_{\zeta}(z,\rY))\dint z\nn\\
&\geq c_{\srY\srO}^2\zeta,
\end{align}
that is, $u_{\zeta}$ has a positive lower bound $c_{\srY\srO}^2\zeta\nu_{\zeta}^{-2}$ on $[4\zeta,1-2\zeta]\times \cbc{\rO,\rY}$.

On the other hand, for  $\zeta<\zeta_0/2$, where $\zeta_0$ is defined as in \Cref{pro_hd_spe_con},  \Cref{pro_hd_spe_con} implies  that
there exist $\underline{\nu}_{2\zeta}>0$ and $c_{2\zeta}\in (0,1)$ such that $\lim_{\zeta\searrow 0}\underline{\nu}_{2\zeta}=\nu$, and, for any $k\geq 0$,
\begin{align}\label{eq_r-spec-2}
    \langle {\bf 1}, \bfT_{\kappa^\circ_{2\zeta}}^{k}{\bf 1}\rangle\geq c_{2\zeta}\underline{\nu}_{2\zeta}^k.
\end{align}
Consequently, \ch{since} $c_{\srY\srO}^{-2}\zeta^{-1}\nu_{\zeta}^{2}u_{\zeta}\geq 1$ on $[4\zeta,1-2\zeta]\times \cbc{\rO,\rY}$ by \eqref{eq_r-spec-1},  the combination of \eqref{eq_r-spec-1} and \eqref{eq_r-spec-2} gives that
\begin{align*}
    c_{\srY\srO}^{-4}\zeta^{-2}\nu_{\zeta}^{2k+4}\norm{u_{\zeta}}^2&=c_{\srY\srO}^{-4}\zeta^{-2}\nu_{\zeta}^{4}\langle u_{\zeta}, \bfT_{\kappa^\circ_{\zeta}}^{2k}u_{\zeta}\rangle=\langle c_{\srY\srO}^{-2}\zeta^{-1}\nu_{\zeta}^{2}u_{\zeta}, \bfT_{\kappa^\circ_{\zeta}}^{2k}c_{\srY\srO}^{-2}\zeta^{-1}\nu_{\zeta}^{2}u_{\zeta}\rangle\\
    &\geq \langle {\bf 1}, \bfT_{\kappa^\circ_{2\zeta}}^{2k}{\bf 1}\rangle\geq c_{2\zeta}\underline{\nu}_{2\zeta}^{2k},\nn
\end{align*}
i.e.,
\begin{align*}
    \nu_{\zeta}\geq \norm{u_{\zeta}}^{1/(k+2)} c_{2\zeta}^{1/(2k+4)}c_{\srY\srO}^{2/(k+2)}\zeta^{1/(k+2)}\underline{\nu}_{2\zeta}^{k/(k+2)}.
\end{align*}
Letting $k\to\infty$ and then $\zeta\searrow 0$, we conclude that
\begin{align*}
    \liminf_{\zeta\searrow 0}\nu_{\zeta}\geq\liminf_{\zeta\searrow 0}\underline{\nu}_{2\zeta}=\nu,
\end{align*}
which gives the lower bound part for \eqref{eq-conv-nu-kappa}.

We further note that
\begin{align*}
    \nu_{\zeta}=\lim_{k\to\infty}\lVert \bfT_{\kappa^\circ_{\zeta}}^{k}\rVert^{1/k}\leq \lim_{k\to\infty}\lVert \bfT_{\kappa}^{k}\rVert^{1/k}=\nu,
\end{align*}
which gives the upper bound part for \eqref{eq-conv-nu-kappa}, as desired.
\end{proof}

\section{Typical distances for $\PAnamdel$ and $\PAnbmdel$: Collapsing}\label{sec-td-other-pams}
In this section, we show that \Cref{thm-log-PA-delta>0} holds for $\PAnamdel$ and $\PAnbmdel$ as well. {We start by defining $\PAnamdel$. For this, we}{} start from the definition of $\PAnaonedel$, and then use the collapsing procedure to construct $\PAnamdel$:
\begin{itemize}
    \item[$\rhd$] {$\PAnaonedel$:} We first consider $n=1$ and $\delta>-1$. This is a graph with one vertex labeled $1$, and a self-loop labeled $1$. As the graph size grows from $n$ to $n+1$, a vertex $n+1$ and an out-edge of $n+1$ labeled $1$ are added to the graph.  The other endpoint of this out-edge from $n+1$ is equal to $n+1$ with probability $\frac{1+\delta}{n(2+\delta)+1+\delta}$, and equal to $i$ with probability $\frac{D_i(n)+\delta}{n(2+\delta)+1+\delta}$ for each $i\in [n]$, where $D_i(n)$ is the degree of vertex $i$ in $\PAnaonedel$.

    \item[$\rhd$] {$\PAnamdel$:} The graph $\PAnamdel$ is derived from $\PAnaonedelcol$ by simultaneously collapsing the vertices $m(\ell-1)+1,\ldots,m\ell$ into a new vertex $\ell$ for each $\ell\in [n]$. We label an edge in $\PAnamdel$ as $j\in[m]$ if it is an out-edge of $m(\ell-1)+j$ in $\PAnaonedelcol$ for some $\ell\in[n]$.
\end{itemize}
Analogously, $\PAnbonedel$ and $\PAnbmdel$ are defined as follows:
\begin{itemize}
    \item[$\rhd$] {$\PAnbonedel$:} We first consider $n=2$ and $\delta>-1$. This is a graph with two vertices, labeled $1$ and $2$, an edge labeled $1$ from $1$ to $2$ and an edge labeled $1$ from $2$ to $1$. As the graph size grows from $n$ to $n+1$, a vertex $n+1$ and an out-edge labeled $1$ from $n+1$ are added to the graph.  The other endpoint of this out-edge of $n+1$ is equal to $i$ with probability $\frac{D_i(n)+\delta}{n(2+\delta)}$ for each $i\in [n]$, where $D_i(n)$ is the degree of vertex $i$ in $\PAnbonedel$.
    \item[$\rhd$] {$\PAnbmdel$:} The graph $\PAnbmdel$ is derived from $\PAnbonedelcol$ by simultaneously collapsing the vertices $m(\ell-1)+1,\ldots,m\ell$ into a new vertex $\ell$ for each $\ell\in [n]$. We label an edge in $\PAnbmdel$ as $j\in[m]$ if it is an out-edge of  $m(\ell-1)+j$ in $\PAnbonedelcol$ for some $\ell\in[n]$.
\end{itemize}
Similar definitions can also be found in \cite{Hofs24}. We observe that the construction of $\PAnbonedelcol$ is very similar to \ch{that of} $\PAndonedelcol$, differing only in the starting \ch{graph}. This observation is crucial for the subsequent proof.

\begin{proof}[Proof of \Cref{thm-log-PA-delta>0} for $\PAnamdel$ and $\PAnbmdel$]
We begin by proving that \Cref{thm-log-PA-delta>0} holds for $\PAnbmdel$. To achieve this, we use the Pólya urn representation of $\PAnbonedelcol$. Next, we study the collapsing procedure, and demonstrate that many results for $\PAndmdel$ also holds for $\PAnbmdel$. Finally, we extend the proof to $\PAnamdel$. Let us now give the details of this proof.
\smallskip

Since each vertex in $\PAnbonedelcol$, except vertex $1$, is connected to an older vertex, the  collapsing procedure yields that
 $\PAnbmdel$ is a connected graph.

Let $(\psi_j)_{j\in [mn]}$ be mutually independent random variables such that
\begin{itemize}
    \item[$\rhd$] $\psi_1=1$, and $\psi_2$ is a beta random
variable with parameters $\alpha_2=2+\delta/m$;
\item[$\rhd$] $(\psi_j)_{3\leq j\leq mn}$ are  beta random
variables with parameters $\alpha_j=\alpha=1+\delta/m$ and $\beta_j=2j-1+\delta(j-1)/m$.
\end{itemize}
\ch{Recall the definition of $\PU$ in Section \ref{sec_polya_pam}.} By \cite[Proposition A.4]{GarHazHofRay22}, $\PAnbonedelcol$ has the same distribution as $\PUm$ with weights $(\psi_j)_{2\leq j\leq mn}$.

Each edge $(\ell,i,j)$ in $\PAnbmdel$ corresponds to an edge in $\PAnbonedelcol$ before the collapsing. The younger endpoint of this pre-collapsed edge is $(\ell-1)m+i$, while there are $m$ choices for the older endpoints, that is, $(j-1)m+u$ for some $u\in[m]$. 
Given a $k$-step edge-labeled  self-avoiding path $\vpie=\cbc{(\pi_{h-1}\vee \pi_h,i_h,\pi_{h-1}\wedge\pi_h):~h\in[k]}$ in $ \PAnbmdel$, 
we denote $e_h=(\pi_{h-1}\vee \pi_h,i_h,\pi_{h-1}\wedge \pi_h),~h\in[k]$ as the edges in this path. Then, 
\begin{align*}
   \cbc{ e_{h,u_h}=(m(\pi_{h-1}\vee \pi_h-1)+i_h,1,(m-1)(\pi_{h-1}\wedge\pi_h)+u_h):~u_h\in[m]}
\end{align*}
is the set of all possible corresponding pre-collapsing edges of $e_h$ in \zhu{$[n]$}. Given $\vec u=(u_1,\ldots,u_k)\in[m]^k$ and 
$\vpie$, let
\begin{align*}
    \vpie_{\vec u}=\cbc{(m(\pi_{h-1}\vee \pi_h-1)+i_h,1,(m-1)(\pi_{h-1}\wedge\pi_h)+u_h)~\forall~h\in [k]}
\end{align*}
be a possible pre-collapsing edge set in $\PAnbonedelcol$ which gives rise to $\vpie$ in $\PAnbmdel$ after the collapsing. We further note that for $\vec u_1\neq \vec u_2$, two edge sets $\vpie_{\vec u_1}$ and $\vpie_{\vec u_2}$ cannot occur in $\PAnbonedelcol$ at the same time, due to the fact that $\vpie_{\vec u_1}\cup\vpie_{\vec u_2}$ is not an edge set (recall the definition of edge set in \Cref{def:edge-set}). Denote the probability conditionally on $(\psi_j)_{j\in[mn]}$ by $\PP_{mn}$. 
Since $\PAnbonedelcol$ is a Pólya urn graph $\PUm$ with weights $(\psi_j)_{2\leq j\leq mn}$, analogous to \eqref{eq_hd_lpath_probability},
\begin{align*}
    \PP_{mn}\bc{\vpie\subseteq \PAnbonedelcol}=\prod_{s=2}^{mn} \psi_s^{\pspiu}(1-\psi_s)^{\qspiu},
\end{align*}
where 
\begin{align}\label{eq_hd_lpath_probability-b}
    \pspiu=\sum_{e\in \vpie_{\vec u}}\indic{s=\ushort{e}} \quad\text{and}\quad
\qspiu=\sum_{e\in \vpie_{\vec u}}\indic{s\in (\ushort{e},\bar{e})},
\end{align}
and where, for $e\in \vpie_{\vec u}$, $\bar{e}$ and $\ushort{e}$ denote the younger and older endpoints of edge $e$ in $\PAnbonedelcol$, respectively.
Then, by \eqref{expctation_function_beta}, 
\begin{align}\label{eq_hd_lower_part0-b}
    \PP\bc{\vpie\subseteq \PAnbmdel}&=\sum_{\vec u\in [m]^k}\PP\bc{\vpie_{\vec u}\subseteq \PAnbonedelcol}\\
    &=\sum_{\vec u\in [m]^k}\prod_{s=2}^{mn}\frac{(\alpha_s+\pspiu-1)_{\pspiu}(\beta_s+\qspiu-1)_{\qspiu}}{(\alpha+\beta_s+\pspiu+\qspiu-1)_{\pspiu+\qspiu}}.\nn
\end{align}
For $\bdt \geq 3$ and $\vpie$ satisfying \Cref{ass_e}, $\max\cbc{\pspiu,\qspiu}\geq 1$ only if $s\geq \bdt $. Further, by \eqref{eq_hd_lpath_probability-b}, $p_s\in\cbc{0,1,2}$. Hence, {analogously}{} to \eqref{eq_path-first-prod}, as $\alpha_s=1+\delta/m$, for $\bdt \geq 3$,
\begin{align}\label{eq_path-first-prod-b}
    \prod_{s=2}^{mn}(\alpha_s+\pspiu-1)_{\pspiu}=(1+\delta/m)^{k-\sum_{s\in\vec\pi_{\vec u}}\indic{\pspiu=2}}(2+\delta/m)^{\sum_{s\in\vec\pi_{\vec u}}\indic{\pspiu=2}}.
\end{align}

Further, {analogously}{} to \eqref{eq_hd_cal_1_E},
\begin{align}\label{eq_hd_cal_1_E-b}
\prod_{s=2}^{mn}\frac{(\beta_s+\qspiu-1)_{\qspiu}}{(\alpha+\beta_s+\pspiu+\qspiu-1)_{\pspiu+\qspiu}}=\Theta^{\frac{\log^2 n}{\bdt }}(2+\delta/m)^{-k}\prod_{e\in \vpie_{\vec u}}\frac{1}{\ushort{e}^{1-\chi}\bar{e}^{\chi}},
\end{align}
where $\chi=\frac{1+\delta/m}{2+\delta/m}=\frac{m+\delta}{2m+\delta}$.
Moreover, by \Cref{ass_e}, $k\leq 4\log_\nu n$ and $\pi_h\geq \bdt $. Then,
\begin{align}\label{eq-app-1/nu-col}
    \prod_{e\in \vpie_{\vec u}}\frac{1}{\ushort{e}^{1-\chi}\bar{e}^{\chi}}&=\prod_{i=1}^k\frac{1}{\big(m(\pi_{h-1}\wedge \pi_h-1)+i_h)\big)^{1-\chi}\big((m-1)(\pi_{h-1}\vee \pi_h)+u_h\big)^\chi}\nn\\
    &=\Theta^{\frac{\log n}{\bdt }}m^{-k}\prod_{i=1}^k\frac{1}{(\pi_{h-1}\wedge \pi_h)^{1-\chi}(\pi_{h-1}\vee \pi_h)^\chi}.
\end{align}
Hence, we conclude from \eqref{eq_hd_lower_part0-b} to \eqref{eq-app-1/nu-col} that 
\begin{align}\label{eq-app-need-label-1}
    \PP\bc{\vpie\subseteq \PAnbmdel}=&\Theta^{\frac{\log^2 n}{\bdt }}\bc{\frac{m+\delta}{2m+\delta}}^k\sum_{\vec u\in [m]^k}\bc{\frac{m+\delta}{2m+\delta}}^{-\sum_{s\in\vec\pi_{\vec u}}\indic{\pspiu=2}}\\
    &\qquad\times m^{-k}\prod_{i=1}^k\frac{1}{(\pi_{h-1}\wedge \pi_h)^{1-\chi}(\pi_{h-1}\vee \pi_h)^\chi}.\nn
\end{align}

To calculate $\sum_{\vec u\in [m]^k}\bc{\frac{m+\delta}{2m+\delta}}^{-\sum_{s\in\vec\pi_{\vec u}}\indic{\pspiu=2}}$, we note that two edges in $\vpie$ share a vertex only if their indices in $\vpie$ are adjacent. If $\pspiu=2$, then there exist two pre-collapsing edges in $\vpie_{\vec u}$ that share an older endpoint. Hence, they also share an endpoint after collapsing. Recall that $\ushort{e}$ is the older endpoint of {the}{} edge $e$. Then, there exists $h\in[k-1]$ such that $s=\ushort{e}_{h,u_h}=\ushort{e}_{h+1,u_{h+1}}$, that is,  $u_h=u_{h+1}$ and   $\ushort{e}_h=\ushort{e}_{h+1}$. 

Given $\vpie$ and $\vec u$, let
 $T_{\pi}=\cbc{h\in[k-1]:~\ushort{e}_h=\ushort{e}_{h+1}}$ and $S_{\pi,\vec u}=\cbc{h\in T_{\pi}:~u_h=u_{h+1}}$. 
Then,
\begin{align}\label{eq-app-need-label-2}
    \sum_{\vec u\in [m]^k}\bc{\frac{m+\delta}{2m+\delta}}^{-\sum_{s\in\vec\pi_{\vec u}}\indic{\pspiu=2}}=\sum_{\vec u\in [m]^k}\big(\frac{m+\delta}{2m+\delta}\big)^{-\abs{S_{\pi,\vec u}}}.
\end{align}
Given $T_{\pi}$, for each $h\in T_{\pi}$ and $(u_i)_{i\in[h]}$, among all $m$ choices of $u_{h+1}$, there is precisely one that satisfies $u_h=u_{h+1}$.  Hence, when $\vec u$ taking a uniform distribution on $[m]^k$, \ch{and} given  $\abs{T_{\pi}}$, $\abs{S_{\pi,\vec u}}$ follows a binomial distribution with $\abs{T_{\pi}}$ trials and success probability ${1/m}$. Consequently,
\begin{align}\label{eq-app-need-label-3}
    \expec\Bigg[\sum_{\vec u\in [m]^k}\bc{\frac{m+\delta}{2m+\delta}}^{-\abs{S_{\pi,\vec u}}}~&\Big|~\abs{T_{\pi}}\Bigg]=m^k\sum_{i=0}^{\abs{T_{\pi}}}\bc{\frac{1}{m}}^{i}\bc{1-\frac{1}{m}}^{\abs{T_{\pi}}-i}\bc{\frac{m+\delta}{2m+\delta}}^{-i}\nn\\
    &=m^k\bc{1-\frac{1}{m}+\frac{1}{m}\frac{2m+\delta}{m+\delta}}^{\abs{T_{\pi}}}=m^u\bc{\frac{m+1+\delta}{m+\delta}}^{\abs{T_{\pi}}}.
\end{align}
Note that $\abs{T_{\pi}}=\sum_{s\in\vec\pi}\indic{\pspi=2}=N_{\srO\srY}$ (recall \eqref{def_NOY}). Combining \eqref{eq-app-need-label-1} to  \eqref{eq-app-need-label-3} yields that
\begin{align*}
    &\PP\bc{\vpie\subseteq \PAnbmdel}\\
    &\qquad=\Theta^{\frac{\log^2 n}{\bdt }}\bc{\frac{m+\delta}{2m+\delta}}^{k-N_{\srO\srY}}\bc{\frac{m+1+\delta}{2m+\delta}}^{N_{\srO\srY}}\prod_{i\in[k]}\frac{1}{(\pi_{i-1}\wedge\pi_i)^{1-\chi}(\pi_{i-1}\vee\pi_i)^\chi },
\end{align*}
which is exactly the same as \eqref{eq_hd_prob_one_path}, except that $\PAndmdel$ is replaced by $\PAnbmdel$. By the same adaptation, we can show that \eqref{eq-path-prob-upper-a1-j} and \Cref{lem_prob_pie_cond} hold when replacing $\PAndmdel$ by $\PAnbmdel$. The rest of the proof of \Cref{thm-log-PA-delta>0} for $\PAnbmdel$ is identical to that for $\PAndmdel$. {This completes the proof for $\PAndmdel$.}{}
\smallskip

For $\PAnamdel$, by \cite[Theorem 5.27]{Hofs24}, for $m\geq 2$, $\PAnamdel$ is a connected graph whp.
Recall that we {have}{} defined $\SPU$ in \Cref{re-pama-spu}. We note from \cite[Proposition 3.2]{GarHazHofRay22} that $\PAnaonedelcol$ has the same distribution as $\SPUm$ with weights $(\psi_j)_{2\leq j\leq mn}$, where $(\psi_j)_{2\leq j\leq mn}$ are independent beta random
variables with parameters $\alpha=1+\delta/m$ and $\beta_j=(2+\delta/m)(j-1)$. Denote the probability $\PP$ conditionally on $(\psi_j)_{j\in[mn]}$ {by $\PP_{mn}$}. Now, in {the definition of the edge set in}{} \Cref{def:edge-set}, we set $j_h\leq \ell_h$ instead of $j_h\leq \ell_h-1$.  By \Cref{re-pama-spu}, when considering $\PAnaonedelcol$ as this $\SPUm$, in $\PAnaonedelcol$,
\eqan{\label{eq_con_prob_present-a}
\PP_{mn}\bc{(\ell,1,j)\in \PAnaonedelcol}=\PP_n\bc{S_{\ell}^{\sss(n)} U_{\ell,1}\in I_j^{\sss(n)}}=\frac{\varphi_j^{\sss(n)}}{S_{\ell}^{\sss(n)}}=\psi_j\prod_{i=j+1}^{\ell}(1-\psi_i).
}
Then, {analogously}{} to \eqref{eq_hd_lpath_probability},
\begin{align*}
    \PP_{mn}\bc{\vpie\subseteq \PAnamdel}=\prod_{s=2}^{mn} \psi_s^{\pspiu}(1-\psi_s)^{\qspiu},
\end{align*}
where
$\pspiu=\sum_{e\in \vpie_{\vec u}}\indic{s=\ushort{e}}$, but now
$\qspiu=\sum_{e\in \vpie_{\vec u}}\indic{s\in (\ushort{e},\bar{e}]}$. 
The rest of the proof of \Cref{thm-log-PA-delta>0} for $\PAnamdel$ is identical to that of $\PAnbmdel$.
\end{proof}

\section{Crude path upper bound: Proof of Lemma~\ref{lem-gen-sum-product-kappa}}
\label{sec-path-upper-bound}
In this section, we derive the crude upper bound on the  kernel products in \Cref{lem-gen-sum-product-kappa}.

\begin{proof}[\ch{Proof of \Cref{lem-gen-sum-product-kappa}}]
For $\ell=1$, \eqref{kappa-PAM-fin-def-form-a-rep} gives that
\begin{align*}
    \sum_{\vec\pi\in\NA_{\nVer_1,\nVer_2}^{\sss \ell}}\prod_{i=1}^{\ell}\kappa\big((\pi_{i-1}/n,\lb^\pi_{i-1}),(\pi_{i}/n,\lb^\pi_i)\big)=\kappa\big((\nVer_1/n,\lb^\pi_0),(\nVer_2/n,\lb^\pi_1)\big)\leq \Theta \zeta^{-1}.
\end{align*}
Hence, \eqref{eq-gen-sum-product-kappa} holds for $\ell=1$.

For $\ell\geq 2$, by \eqref{eq-change-label-add-theta1}, for each $\vec\pi\in\NA_{\nVer_1,\nVer_2}^{\sss \ell}$,
\begin{align*}
    \prod_{i=1}^{\ell}\kappa\big((\pi_{i-1}/n,\lb^\pi_{i-1}),(\pi_{i}/n,\lb^\pi_i)\big)\leq \Theta \zeta^{-2}\prod_{i=2}^{\ell-1}\kappa\big((\pi_{i-1}/n,\lb^\pi_{i-1}),(\pi_{i}/n,\lb^\pi_i)\big).
\end{align*}
Summing the \ch{above inequality} over all $\vec\pi\in\NA_{\nVer_1,\nVer_2}^{\sss \ell}$,
\begin{align}
&\sum_{\vec\pi\in\NA_{\nVer_1,\nVer_2}^{\sss \ell}}\prod_{i=1}^{\ell}\kappa\big((\pi_{i-1}/n,\lb^\pi_{i-1}),(\pi_{i}/n,\lb^\pi_i)\big)\label{eq-gen-sum-product-kappa-2-2}\\
&\qquad\leq \Theta \zeta^{-2}\sum_{\substack{\pi_i\in [n],\\i\in[\ell-1]}}\prod_{i=2}^{\ell-1}\kappa\big((\pi_{i-1}/n,\lb^\pi_{i-1}),(\pi_{i}/n,\lb^\pi_i)\big).\nn
\end{align}
Note from \eqref{decrease-kappa} and \eqref{decrease-kappa-0} that
\begin{align*}
    \kappa\big((\pi_{i-1}/n,\lb^\pi_{i-1}),(\pi_{i}/n,\lb^\pi_i)\big)\leq &n\int_{(\pi_i-1)/n}^{\pi_i/n}\kappa\big((x_{i-1},\rmq_{i-1}),(x_i,\rmq_i)\big)\dint x_i,
\end{align*}
and
\begin{align*}
        \kappa\big((\pi_1/n,\lb^\pi_1),(\pi_2/n,\lb^\pi_2)\big)\leq &n^2\int_{(\pi_1-1)/n}^{\pi_1/n}\int_{(\pi_2-1)/n}^{\pi_2/n}\kappa\big((x_1,\rmq_1),(x_2,\rmq_2)\big)\dint x_1\dint x_2,
\end{align*}
where $\rmq_i=q(x_i-x_{i-1})$ for $i\in[\ell]$ (recall the definition of $q(\cdot)$ in  \eqref{def_label_direction}).
Hence, we conclude from \eqref{eq-gen-sum-product-kappa-2-2} that
\begin{align}\label{eq-gen-sum-product-kappa-2}
\sum_{\vec\pi\in\NA_{\nVer_1,\nVer_2}^{\sss \ell}}&\prod_{i=1}^{\ell}\kappa\big((\pi_{i-1}/n,\lb^\pi_{i-1}),(\pi_{i}/n,\lb^\pi_i)\big)\nn\\
&\leq \Theta \zeta^{-2}n^{\ell-1}\int_0^1\cdots\int_0^1\prod_{i=2}^{\ell-1}\kappa\big((x_{i-1},\rmq_{i-1}),(x_i,\rmq_i)\big)\dint x_1\cdots \dint x_{\ell-1}.
\end{align}
Consequently, combining    \eqref{eq-gen-sum-product-kappa-2} with \eqref{eq-new-Tkappa-2} yields that
\begin{align*}
        &\sum_{\vec\pi\in\NA_{\nVer_1,\nVer_2}^{\sss \ell}}\prod_{i=1}^{\ell}\kappa\big((\pi_{i-1}/n,\lb^\pi_{i-1}),(\pi_{i}/n,\lb^\pi_i)\big)\leq \Theta \zeta^{-2} n^{\ell-1}\langle {\bf 1},\bfT_\kappa^{\ell-2} {\bf 1}\rangle\leq \Theta\zeta^{-2}\nu^\ell n^{\ell-1},
\end{align*}
i.e., \eqref{eq-gen-sum-product-kappa} holds for $\ell\geq 2$.
\end{proof}

\section{Restoring avoidance constraints: Proof of Lemma~\ref{lem_hd_sum_restrict_path-sum-gen-nka-compare}}
\label{app-E}
In this section, we prove \Cref{lem_hd_sum_restrict_path-sum-gen-nka-compare}. Recall that in \Cref{lem_hd_sum_restrict_path-sum-gen-nka}, we established a lower bound on the summation of kernel products over paths in $\NP_{\nVer_1,\nVer_2}^{\sss k}$. 
Note that in $\TSA_{\nVer_1,\nVer_2}^{\sss k}$, we have $\pi_i\not\in B_r^{\sss(G_n)}(\Ver_1)\cup B_r^{\sss(G_n)}(\Ver_2)$ for all $i\in[k-1]$, while also 
$\pi_i\neq \pi_j$ for all $i\neq j$. In $\NP_{\nVer_1,\nVer_2}^{\sss k}$, such avoidance constraints are not present, so we wish to reinstate these avoidance constraints and eliminate the extra terms in \eqref{eq_prod_kappa_need_in_var-F12}.

In this proof, we separately address the contributions of cases involving self-loops, cycles of length at least $2$ and intersections with the $r$-neighborhoods of $\Ver_1$ and $\Ver_2$. This distinction arises because we use \Cref{lem-gen-sum-product-kappa} in the proofs, which requires the path to be neighbor-avoiding. We compare summations of kernel products over paths in $\NP_{\nVer_1,\nVer_2}^{\sss k}$ and $\TSA_{\nVer_1,\nVer_2}^{\sss k}$ with those in $\NA_{\nVer_1,\nVer_2}^{\sss k}$ in \Cref{sec-proof-eq_hd_sum_restrict_path-sum-gen-1-2} and \Cref{sec-proof-eq_hd_sum_restrict_path-sum-gen-1-1}, respectively.

\subsection{Putting back the self-loops}\label{sec-proof-eq_hd_sum_restrict_path-sum-gen-1-2}
We start by deriving a upper bound 
on the contribution of \ch{paths} containing self-loops, that is, the paths in $\NP_{\nVer_1,\nVer_2}^{\sss k}$ for some $\nVer_1\in F_1$ and $\nVer_2\in F_2$ that are not in $ \NA_{\nVer_1,\nVer_2}^{\sss k}$. 
The underlying idea is as follows: For each $i\in [k]$, the number of choices of $\pi_i$ is around $(1-\zeta)n$, with only one option that satisfies $\pi_i=\pi_{i-1}$; for each self-loop, there are at most $k$ positions where it can occur. Hence, the contribution of paths with $\ell$ self-loops to \eqref{eq_prod_kappa_need_in_var-F12} is about $k^\ell/n^\ell$ smaller than the main contribution.
\begin{figure}[htbp]
  \centering\def\svgwidth{0.7\columnwidth} 
  \includesvg{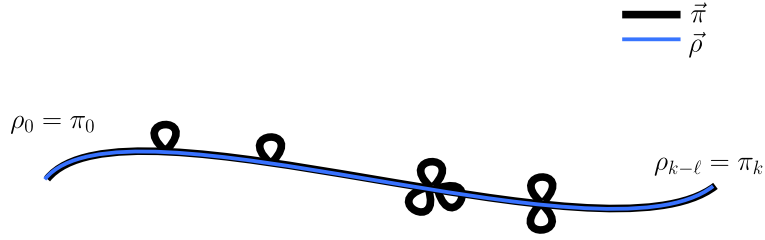}
  \caption{A sample of the decomposition}
  \label{pic_path_decompose_pi}
\end{figure}

 

{Each path}{} $\vec\pi\in \NP_{\nVer_1,\nVer_2}^{\sss k}$ can be decomposed into a self-avoiding path and several self-loops (see \Cref{pic_path_decompose_pi}). 

Let $\SL=\cbc{i\in[k]:~\pi_i=\pi_{i-1}}$ be the set of indices of the self-loops in $\vec\pi$ and $\ell$ be the cardinality of $\SL$. Define $\vec\rho:=(\rho_0,\ldots,\rho_{k-\ell})$ as the path obtained from $\vec\pi$ through removing all its self-loops. Then $\vec\rho\in \NA^{k-\ell}_{\pi_0,\pi_k}$. 


Furthermore, we can decompose $\prod_{i=1}^k\nka\big((\pi_{i-1}/n,\lb^\pi_{i-1}),(\pi_{i}/n,\lb^\pi_i)\big)$ according to the previous path decomposition as 
\begin{align}\label{eq-path-decomp-self-loop}
&\prod_{i=1}^k\nka\big((\pi_{i-1}/n,\lb^\pi_{i-1}),(\pi_{i}/n,\lb^\pi_i)\big)\\
    &\quad=\prod_{i=1}^{k-\ell}\nka\big((\rho_{i-1}/n,\lb^\rho_{i-1}),(\rho_{i}/n,\lb^\rho_i)\prod_{i\in \SL} \nka\big((\pi_{i-1}/n,\lb^\pi_{i-1}),(\pi_{i}/n,\lb^\pi_i)\big).\nn
\end{align}
Since $\nVer_1\in F_1$, $\nVer_2\in F_2$ and $\vec\pi\in \NP_{\nVer_1,\nVer_2}^{\sss k}$, $\pi_i\geq \zeta n$ for all $0\leq i\leq k$. Then, by \eqref{eq-change-label-add-theta1-new}, 
\begin{align*}
\nka\big((\pi_{i-1}/n,\lb^\pi_{i-1}),(\pi_{i}/n,\lb^\pi_i)\big)\leq \Theta\zeta^{-1}.    
\end{align*}
Applying this inequality to all $i\in\SL$, \eqref{eq-path-decomp-self-loop} {is bounded by}
\begin{align*}
    \prod_{i=1}^k\nka\big((\pi_{i-1}/n,\lb^\pi_{i-1}),(\pi_{i}/n,\lb^\pi_i)\big)\leq\Theta ^{\ell}\zeta^{-\ell}\prod_{i=1}^{k-\ell}\nka\big((\rho_{i-1}/n,\lb^\rho_{i-1}),(\rho_{i}/n,\lb^\rho_i)\big).
\end{align*}
Then, summing over all $ \vec\pi\in \NP_{\nVer_1,\nVer_2}^{\sss k}$ given $\SL=\cbc{s_1,\ldots,s_\ell}$, and combining with \Cref{lem-gen-sum-product-kappa} yields that, for $\ell\geq 1$, 
\begin{align*}
&\sum_{\substack{\vec\pi\in \NP_{\nVer_1,\nVer_2}^{\sss k},\\\SL=\cbc{s_1,\ldots,s_\ell}}}\prod_{i=1}^k\nka\big((\pi_{i-1}/n,\lb^\pi_{i-1}),(\pi_{i}/n,\lb^\pi_i)\big)\\
    &\qquad\leq \Theta ^\ell\zeta^{-\ell}\sum_{\vec\pi\in \NA^{k-\ell}_{\nVer_1,\nVer_2}}\prod_{i=1}^{k-\ell}\nka\big((\rho_{i-1}/n,\lb^\rho_{i-1}),(\rho_{i}/n,\lb^\rho_i)\big)\leq \Theta^\ell\zeta^{-\ell-2}\nu^{k-\ell}n^{k-\ell-1}
\end{align*}
Given $\ell\in [k]$, since $s_i\in[k]$ for $i\in[\ell]$, there are at most $k^\ell$ choices of $\SL$. Summing over all $\nVer_1\in F_1$, $\nVer_2\in F_2$, all $\ell$ from $1$ to $k$ and all possible $\SL$ of size $\ell$, we conclude that 
\begin{align*}
    &\sum_{\substack{\nVer_1\in F_1,\\\nVer_2\in F_2}}\sum_{\substack{\vec\pi\in \NP_{\nVer_1,\nVer_2}^{\sss k},\\\abs{\SL}\geq 1}}\prod_{i=1}^k\nka\big((\pi_{i-1}/n,\lb^\pi_{i-1}),(\pi_{i}/n,\lb^\pi_i)\big)\leq \abs{F_1}\abs{F_2}\sum_{\ell=1}^k k^\ell\Theta^\ell\zeta^{-\ell-2}\nu^{k}n^{k-\ell-1}.
\end{align*}
Let $\Theta^{\star,1}$ be a uniform upper bound of $\Theta$ up to now.
Let $n_1\geq \zeta^{-2}$ be a number such that $2\Theta^{\star,1} \zeta^{-1}(\log_\nu n) n^{-1}<\frac{1}{2}$ for all $n\geq n_1$. Then, for $n\geq n_1$ and $k\leq 2\log_\nu n$, 
\begin{align}\label{eq-con-self-loop}
    \sum_{\substack{\nVer_1\in F_1,\\\nVer_2\in F_2}}&\sum_{\substack{\vec\pi\in \NP_{\nVer_1,\nVer_2}^{\sss k},\\\abs{\SL}\geq 1}}\prod_{i=1}^k\nka\big((\pi_{i-1}/n,\lb^\pi_{i-1}),(\pi_{i}/n,\lb^\pi_i)\big)\leq \abs{F_1}\abs{F_2}\sum_{\ell=1}^\infty k^\ell\Theta ^{\ell}\zeta^{-\ell-2}\nu^k n^{k-1-\ell}\nn\\
    &\leq\abs{F_1}\abs{F_2}n^{k-1}\nu^k\frac{\Theta \zeta^{-3}kn^{-1}}{1-\Theta^{\star,1} \zeta^{-1}kn^{-1}}=\Theta \abs{F_1}\abs{F_2}\zeta^{-3}k\nu^k n^{k-2}.
\end{align}
Note that
\begin{align*}
    &\sum_{\substack{\nVer_1\in F_1,\\\nVer_2\in F_2}}\sum_{\vec\pi\in \NP_{\nVer_1,\nVer_2}^{\sss k}\backslash\NA_{\nVer_1,\nVer_2}^{\sss k}}\prod_{i=1}^k\nka\big((\pi_{i-1}/n,\lb^\pi_{i-1}),(\pi_{i}/n,\lb^\pi_i)\big)\\
    &\qquad\leq  \sum_{\substack{\nVer_1\in F_1,\\\nVer_2\in F_2}}\sum_{\substack{\vec\pi\in \NP_{\nVer_1,\nVer_2}^{\sss k},\\\abs{\SL}\geq 1}}\prod_{i=1}^k\nka\big((\pi_{i-1}/n,\lb^\pi_{i-1}),(\pi_{i}/n,\lb^\pi_i)\big).
\end{align*}
Then, combining \eqref{eq-con-self-loop} with \eqref{eq_prod_kappa_need_in_var-F12} yields that
\zhu{\begin{align*}
&\frac{\sum_{\substack{\nVer_1\in F_1,\\\nVer_2\in F_2}}\sum_{\vec\pi\in \NP_{\nVer_1,\nVer_2}^{\sss k}\backslash\NA_{\nVer_1,\nVer_2}^{\sss k}}\prod_{i=1}^k\nka\big((\pi_{i-1}/n,\lb^\pi_{i-1}),(\pi_{i}/n,\lb^\pi_i)\big)}{\sum_{\substack{\nVer_1\in F_1,\\\nVer_2\in F_2}}\sum_{\vec\pi\in \NP_{\nVer_1,\nVer_2}^{\sss k}}\prod_{i=1}^k\nka\big((\pi_{i-1}/n,\lb^\pi_{i-1}),(\pi_{i}/n,\lb^\pi_i)\big)}\\
 &\qquad\leq \Theta\frac{\abs{F_1}\abs{F_2}\zeta^{-3}k\nu^k n^{k-2}}{\abs{F_1}\abs{F_2}{\zeta}^2 c_{\zeta}\nu^{k(1-\vep/4)}n^{k-1}}=\Theta{\zeta}^{-5} c_{\zeta}^{-1}k\nu^{k\vep /4}n^{-1}.
%
%
\end{align*}}
Hence, we conclude that, uniformly in the choices of $\Ver_1$ and $\Ver_2$, there exists an $n_2>\zeta^{-2}$,  depending only on $m,\delta,M,r,\vep$ and $\zeta$, such that for any $n>n_2$ and $k\leq 2\log_\nu n$, on the event $\GNW$,
\begin{align}\label{eq-frac-kappa-1/2}
   &\frac{\sum_{\substack{\nVer_1\in F_1,\\\nVer_2\in F_2}}\sum_{\vec\pi\in \NP_{\nVer_1,\nVer_2}^{\sss k}\backslash\NA_{\nVer_1,\nVer_2}^{\sss k}}\prod_{i=1}^k\nka\big((\pi_{i-1}/n,\lb^\pi_{i-1}),(\pi_{i}/n,\lb^\pi_i)\big)}{\sum_{\substack{\nVer_1\in F_1,\\\nVer_2\in F_2}}\sum_{\vec\pi\in \NP_{\nVer_1,\nVer_2}^{\sss k}}\prod_{i=1}^k\nka\big((\pi_{i-1}/n,\lb^\pi_{i-1}),(\pi_{i}/n,\lb^\pi_i)\big)}\\
   &\qquad\leq \Theta{\zeta}^{-5} c_{\zeta}^{-1}(\log n) n^{\vep /2-1} \leq \frac{1}{2}.\nn
\end{align}
Hence, for $n_0\geq n_2$, if $n\geq n_0$,
\begin{align}\label{eq_hd_sum_restrict_path-sum-gen-1-2-new}
    &\sum_{\substack{\nVer_1\in F_1,\\\nVer_2\in F_2}}\sum_{\vec\pi\in \NA_{\nVer_1,\nVer_2}^{\sss k}}\prod_{i=1}^k\nka\big((\pi_{i-1}/n,\lb^\pi_{i-1}),(\pi_{i}/n,\lb^\pi_i)\big)\\
    &\qquad\geq \frac{1}{2}\sum_{\substack{\nVer_1\in F_1,\\\nVer_2\in F_2}}\sum_{\vec\pi\in \NP_{\nVer_1,\nVer_2}^{\sss k}}\prod_{i=1}^k\nka\big((\pi_{i-1}/n,\lb^\pi_{i-1}),(\pi_{i}/n,\lb^\pi_i)\big)\nn.
\end{align}

\subsection{Putting back longer cycles and intersections}\label{sec-proof-eq_hd_sum_restrict_path-sum-gen-1-1}
In this section, we bound the summations of kernel products over paths in$\TSA_{\nVer_1,\nVer_2}^{\sss k}$ from below by those in $\NA_{\nVer_1,\nVer_2}^{\sss k}$.

For this purpose, we put back the avoidance constraints that  $\vec\pi\in \TSA_{\nVer_1,\nVer_2}^{\sss k}$ neither contains any cycle of length at least $2$, nor intersects with the $r$-neighborhoods of $\Ver_1$ and $\Ver_2$. We do these one by one.
\paragraph{The avoidance constraints: bringing back the cycles of length at least $2$}
We first calculate the contribution of the non-self-avoiding case where $\vec\pi\in \NA_{\nVer_1,\nVer_2}^{\sss k}$ satisfies that $\pi_{j_1}=\pi_{j_2}$ for some $0\leq j_1<j_1+1< j_2\leq k$.  Here we do not allow $j_2=j_1+1$ since $\vec\pi$ is neighbor-avoiding.

In this case, the path contains a cycle of length $j_2-j_1$. 
Given $0\leq j_1<j_1+1< j_2\leq k$, we claim that, for $\nVer_1\in F_1$ and $\nVer_2\in F_2$,
\begin{align}\label{eq-bring-back-cycle-a1a2}
   \sum_{\substack{\vec\pi\in \NA_{\nVer_1,\nVer_2}^{\sss k},\\\pi_{j_1}=\pi_{j_2}}} \prod_{i=1}^k\nka\big((\pi_{i-1}/n,\lb^\pi_{i-1}),(\pi_{i}/n,\lb^\pi_i)\big)\leq \Theta n^{k-2}\zeta^{-6}\nu^{k}.
\end{align}
We prove \eqref{eq-bring-back-cycle-a1a2} using a case analysis as follows:
\begin{itemize}
    \item[$\rhd$] If $j_1=0$ and $j_2=k$, then $\nVer_1=\nVer_2$. In this scenario, the lhs of \eqref{eq-bring-back-cycle-a1a2} is zero as $\cbc{\nVer_1}\not\in F_1\cap F_2=\emptyset$. Hence, \eqref{eq-bring-back-cycle-a1a2} holds.
    \item[$\rhd$] If $j_1=0$ and $j_2<k$, then we can divide $\vec\pi$ into two subpaths $\vec\pi(1)=\bc{\pi_0,\ldots,\pi_{j_2}}$ and $\vec\pi(2)=\bc{\pi_{j_2},\ldots,\pi_k}$. Consequently,
\begin{align*}
   &\sum_{\substack{\vec\pi\in \NA_{\nVer_1,\nVer_2}^{\sss k},\\\pi_{j_1}=\pi_{j_2}}} \prod_{i=1}^k\nka\big((\pi_{i-1}/n,\lb^\pi_{i-1}),(\pi_{i}/n,\lb^\pi_i)\big)\\
   &\qquad\leq \sum_{\vec\pi(1)\in \NA_{\nVer_1,\nVer_1}^{j_2}} \prod_{i=1}^{j_2}\nka\big((\pi_{i-1}/n,\lb^\pi_{i-1}),(\pi_{i}/n,\lb^\pi_i)\big)\\
   &\qquad\quad\times\sum_{\vec\pi(2)\in \NA_{\nVer_1,\nVer_2}^{k-j_2}} \prod_{i=j_2+1}^k\nka\big((\pi_{i-1}/n,\lb^\pi_{i-1}),(\pi_{i}/n,\lb^\pi_i)\big),
\end{align*}
and then \eqref{eq-bring-back-cycle-a1a2} follows directly by applying \Cref{lem-gen-sum-product-kappa} to the above inequalities.

\item[$\rhd$] If $j_1>0$ and $j_2=k$, \ch{then,} by an analogous argument as in the previous case, \eqref{eq-bring-back-cycle-a1a2} holds.

\item[$\rhd$] For $1\leq j_1\leq j_1+1\leq j_2\leq k-1$, we can divide $\vec\pi$ into three subpaths $\vec\pi(1)=\cbc{\pi_0,\ldots,\pi_{j_1}}$, $\vec\pi(2)=\cbc{\pi_{j_1},\ldots,\pi_{j_2}}$ and $\vec\pi(3)=\cbc{\pi_{j_2},\ldots,\pi_k}$. Then, 
\begin{align*}
   \sum_{\substack{\vec\pi\in \NA_{\nVer_1,\nVer_2}^{\sss k},\\\pi_{j_1}=\pi_{j_2}}}& \prod_{i=1}^k\nka\big((\pi_{i-1}/n,\lb^\pi_{i-1}),(\pi_{i}/n,\lb^\pi_i)\big)\\
   &\quad\leq \sum_{b=\ceil{\zeta n}}^n\sum_{\vec\pi(1)\in \NA_{\nVer_1,b}^{j_1}} \prod_{i=1}^{j_1}\nka\big((\pi_{i-1}/n,\lb^\pi_{i-1}),(\pi_{i}/n,\lb^\pi_i)\big)\\
   &\qquad\times\sum_{\vec\pi(2)\in \NA_{b,b}^{j_2-j_1}} \prod_{i=j_1+1}^{j_2}\nka\big((\pi_{i-1}/n,\lb^\pi_{i-1}),(\pi_{i}/n,\lb^\pi_i)\big)\\
   &\qquad\times\sum_{\vec\pi(3)\in \NA_{b,\nVer_2}^{k-j_2}} \prod_{i=j_2+1}^k\nka\big((\pi_{i-1}/n,\lb^\pi_{i-1}),(\pi_{i}/n,\lb^\pi_i)\big),
\end{align*}
and then \eqref{eq-bring-back-cycle-a1a2} follows directly by applying \Cref{lem-gen-sum-product-kappa} to the {above three contributions}.
\end{itemize}
Hence, we conclude that \eqref{eq-bring-back-cycle-a1a2} holds.
\smallskip

Summing \eqref{eq-bring-back-cycle-a1a2} over all $\nVer_1\in F_1$, $\nVer_2\in F_2$ and all $0\leq j_1<j_1+1<j_2\leq k$ gives that
\begin{align}\label{eq-bring-back-cycle}
   &\sum_{\substack{\nVer_1\in F_1,\\\nVer_2\in F_2}}\sum_{j_1=0}^{k-2}\sum_{j_2=j_1+2}^{k}\sum_{\substack{\vec\pi\in \NA_{\nVer_1,\nVer_2}^{\sss k},\\\pi_{j_1}=\pi_{j_2}}} \prod_{i=1}^k\nka\big((\pi_{i-1}/n,\lb^\pi_{i-1}),(\pi_{i}/n,\lb^\pi_i)\big)\\
   &\qquad\leq \Theta\abs{F_1}\abs{F_2}\zeta^{-6}k^2\nu^{k}n^{k-2}.\nn
\end{align}

\paragraph{The avoidance constraints: \zhu{restoring} the intersections with the $r$-neighborhoods} We continue by calculating the contribution of the case where the neighbor-avoiding path $\vec\pi$ intersects $B_r^{\sss(G_n)}(\Ver_1)\cup B_r^{\sss(G_n)}(\Ver_2)$ at some points other than $\pi_0$ and $\pi_k$. 
Given a $j\in [k-1]$ such that $\pi_j\in B_r^{\sss(G_n)}(\Ver_1)\cup B_r^{\sss(G_n)}(\Ver_2)$, we can divide $\vec\pi$ into two subpaths $\vec\pi(1)=\bc{\pi_0,\ldots,\pi_j}$ and $\vec\pi(2)=\bc{\pi_j,\ldots,\pi_k}$. We further note that $M$ is an upper bound for both $\abs{B_r^{\sss(G_n)}(\Ver_1)}$ and $\abs{B_r^{\sss(G_n)}(\Ver_2)}$ when $\big(\bar{B}_r^{\sss(G_n)}(\Ver_1), \bar{B}_r^{\sss(G_n)}(\Ver_2)\big)$ is a good pair. Then, on the good event $\GNW$, \Cref{lem-gen-sum-product-kappa} yields that
\begin{align*}
   \sum_{\substack{\vec\pi\in \NA_{\nVer_1,\nVer_2}^{\sss k},\\\pi_j\in B_r^{\sss(G_n)}(\Ver_1)\cup B_r^{\sss(G_n)}(\Ver_2)}} &\prod_{i=1}^k\nka\big((\pi_{i-1}/n,\lb^\pi_{i-1}),(\pi_{i}/n,\lb^\pi_i)\big)\\
   &\leq\sum_{b\in B_r^{\sss(G_n)}(\Ver_1)\cup B_r^{\sss(G_n)}(\Ver_2)} \sum_{\vec\pi(1)\in \NA_{\nVer_1,b}^{j}} \prod_{i=1}^{j}\nka\big((\pi_{i-1}/n,\lb^\pi_{i-1}),(\pi_{i}/n,\lb^\pi_i)\big)\\
   &\quad\times\sum_{\vec\pi(2)\in \NA_{b,\nVer_2}^{k-j}} \prod_{i=j+1}^k\nka\big((\pi_{i-1}/n,\lb^\pi_{i-1}),(\pi_{i}/n,\lb^\pi_i)\big)\\
   &\leq \Theta \zeta^{-4}M\nu^{k} n^{k-2}.
\end{align*}
Summing over all $\nVer_1\in F_1$, $\nVer_2\in F_2$ and all $j\in [k-1]$, we conclude that 
\begin{align}\label{eq-contri-intersect}
    \sum_{\substack{\nVer_1\in F_1,\\\nVer_2\in F_2}}\sum_{j\in[k-1]}&\sum_{\substack{\vec\pi\in \NA_{\nVer_1,\nVer_2}^{\sss k},\\\pi_j\in B_r^{\sss(G_n)}(\Ver_1)\cup B_r^{\sss(G_n)}(\Ver_2)}} \prod_{i=1}^k\nka\big((\pi_{i-1}/n,\lb^\pi_{i-1}),(\pi_{i}/n,\lb^\pi_i)\big)\nn\\
    &\leq \Theta\abs{F_1}\abs{F_2} \zeta^{-4}Mk\nu^{k} n^{k-2}.
\end{align}

Note that 
\begin{align}\label{pth_con_na_without_sa}
    &\sum_{\substack{\nVer_1\in F_1,\\\nVer_2\in F_2}}\sum_{\vec\pi\in \NA_{\nVer_1,\nVer_2}^{\sss k}\backslash\TSA_{\nVer_1,\nVer_2}^{\sss k}}\prod_{i=1}^k\nka\big((\pi_{i-1}/n,\lb^\pi_{i-1}),(\pi_{i}/n,\lb^\pi_i)\big)\\
    &\qquad\leq\sum_{\substack{\nVer_1\in F_1,\\\nVer_2\in F_2}}\sum_{j_1=0}^{k-2}\sum_{j_2=j_1+2}^{k}\sum_{\substack{\vec\pi\in \NA_{\nVer_1,\nVer_2}^{\sss k},\\\pi_{j_1}=\pi_{j_2}}} \prod_{i=1}^k\nka\big((\pi_{i-1}/n,\lb^\pi_{i-1}),(\pi_{i}/n,\lb^\pi_i)\big)\nn\\
    &\qquad\quad+\sum_{\substack{\nVer_1\in F_1,\\\nVer_2\in F_2}}\sum_{j\in[k-1]}\sum_{\substack{\vec\pi\in \NA_{\nVer_1,\nVer_2}^{\sss k},\\\pi_j\in B_r^{\sss(G_n)}(\Ver_1)\cup B_r^{\sss(G_n)}(\Ver_2)}} \prod_{i=1}^k\nka\big((\pi_{i-1}/n,\lb^\pi_{i-1}),(\pi_{i}/n,\lb^\pi_i)\big).\nn
\end{align}
Recall $n_2$, which depends only on $m,\delta,M,r,\vep$ and $\zeta$, as defined in \eqref{eq-frac-kappa-1/2}.  Combining \eqref{eq_prod_kappa_need_in_var-F12},  \eqref{eq_hd_sum_restrict_path-sum-gen-1-2-new}, \eqref{eq-bring-back-cycle}, \eqref{eq-contri-intersect} \zhu{and}  \eqref{pth_con_na_without_sa} yields that, for $n\geq n_2$,
\zhu{\begin{align}\label{eq-sum-sa-tse-dif-con}
    &\frac{\sum_{\substack{\nVer_1\in F_1,\\\nVer_2\in F_2}}\sum_{\vec\pi\in \NA_{\nVer_1,\nVer_2}^{\sss k}\backslash\TSA_{\nVer_1,\nVer_2}^{\sss k}}\prod_{i=1}^k\nka\big((\pi_{i-1}/n,\lb^\pi_{i-1}),(\pi_{i}/n,\lb^\pi_i)\big)}{\sum_{\substack{\nVer_1\in F_1,\\\nVer_2\in F_2}}\sum_{\vec\pi\in \NA_{\nVer_1,\nVer_2}^{\sss k}}\prod_{i=1}^k\nka\big((\pi_{i-1}/n,\lb^\pi_{i-1}),(\pi_{i}/n,\lb^\pi_i)\big)}\\ 
    &\qquad\geq \Theta \frac{\abs{F_1}\abs{F_2}\zeta^{-6}k^2\nu^{k}n^{k-2}+\abs{F_1}\abs{F_2} \zeta^{-4}M k\nu^{k} n^{k-2}}{\abs{F_1}\abs{F_2}{\zeta}^2c_{\zeta}\nu^{k(1-\vep/4)}n^{k-1}}\nn\\
    &\qquad\geq \Theta c_{\zeta}^{-1}\zeta^{-8}Mk^2\nu^{k\vep/4} n^{-1}.\nn
\end{align}}
Note that there exists an $n_0\geq n_2$, depending only on $m,\delta,M,r,\vep$ and $\zeta$, such that for any $n\geq n_0$, with $\Theta^{\star,2}$ the uniform upper bound \zhu{on} $\Theta$ up to now, 
\zhu{\begin{align*}
    (2\log_\nu n)^{-2} n^{1-\vep/2}\geq 2\Theta^{\star,2}   c_{\zeta}^{-1}\zeta^{-8}M.
\end{align*}}
Then,
for $n\geq n_0$ and $k\leq 2\log_\nu n$, the rhs of \eqref{eq-sum-sa-tse-dif-con} has a lower bound of $1/2$. Hence, we conclude that
\begin{align}\label{eq_hd_sum_restrict_path-sum-gen-1-1-new}
    &\sum_{\substack{\nVer_1\in F_1,\\\nVer_2\in F_2}}\sum_{\vpie\in \TESA_{\nVer_1,\nVer_2}^{\sss e,k}}\prod_{i=1}^k\nka\big((\pi_{i-1}/n,\lb^\pi_{i-1}),(\pi_{i}/n,\lb^\pi_i)\big)\\
    &\qquad\geq \frac{1}{2}\sum_{\substack{\nVer_1\in F_1,\\\nVer_2\in F_2}}\sum_{\vec\pi\in \NA_{\nVer_1,\nVer_2}^{\sss k}}\prod_{i=1}^k\nka\big((\pi_{i-1}/n,\lb^\pi_{i-1}),(\pi_{i}/n,\lb^\pi_i)\big).\nn
\end{align}

\begin{proof}[Proof of \Cref{lem_hd_sum_restrict_path-sum-gen-nka-compare}]
    \zhu{\Cref{lem_hd_sum_restrict_path-sum-gen-nka-compare}} follows directly from the combination of \eqref{eq_hd_sum_restrict_path-sum-gen-1-2-new}, \eqref{eq_hd_sum_restrict_path-sum-gen-1-1-new} and \Cref{re-kappa-old-new}.
\end{proof}

\end{appendix}
}

\begin{acks}[Acknowledgments]
The authors would like to thank Joost Jorritsma, Julia Komj\'athy and Rounak Ray for helpful discussions. \zhu{The authors would further like to thank the reviewers for their careful reading and helpful suggestions, especially for pointing out easy proofs for \Cref{pro_hd_spe_con} and \Cref{lem_hd_sum_restrict_path-sum-gen-nka}.}
\end{acks}

\begin{funding}
Both authors are supported by the Netherlands Organisation for Scientific Research (NWO) through the Gravitation NETWORKS grant 024.002.003 and by the National Science Foundation under Grant No. DMS-1928930 while they were in residence at the Simons Laufer Mathematical Sciences Institute in Berkeley, California, during the spring semester. The work of the second author is further supported by the European Union’s Horizon 2020 research and innovation programme under the Marie Skłodowska-Curie grant agreement No. 945045.
\end{funding}

\DeclareRobustCommand{\Hofstad}[3]{#3}
\DeclareRobustCommand{\Esker}[3]{#3}
\bibliographystyle{imsart-number} 
\bibliography{bibRGCN}       
\end{document}